\documentclass[aos]{imsart}

\RequirePackage{amsthm,amsmath,amsfonts,amssymb}
\RequirePackage{xcolor}
\RequirePackage{mathrsfs}
\RequirePackage{algcompatible}
\RequirePackage{tabularx,array}
\RequirePackage{lineno}
\RequirePackage[protrusion=true,expansion=true,final]{microtype}
\RequirePackage{graphicx}
\RequirePackage[numbers,sort&compress]{natbib}
\RequirePackage[colorlinks,citecolor=blue,urlcolor=blue,linkcolor=blue]{hyperref}
\RequirePackage{booktabs}
\RequirePackage{multirow}
\RequirePackage{placeins}

\makeatletter
\def\journal@name{}
\def\journal@url{}
\setpkgattr{copyright}{text}{}
\makeatother

\startlocaldefs
\theoremstyle{plain}
\newtheorem{theorem}{Theorem}[section]
\newtheorem{lemma}[theorem]{Lemma}

\newtheorem{corollary}[theorem]{Corollary}
\newtheorem{claim}[theorem]{Claim} 

\theoremstyle{definition}
\newtheorem{assumption}{Assumption}[section]

\newtheorem{remark}{Remark}[section]

\def\bxi{\boldsymbol{\xi}}

\def\mid{\,|\,}
\def\sep{\;\big|\;}
\def\sign{\mathop{\mathrm{sign}}}
\def\sgn{\mathop{\mathrm{sign}}}

\def\argmax{\mathop{\text{\rm arg\,max}}}
\def\diag{\operatorname{diag}}

\def\Cov{\operatorname{Cov}}
\newcommand{\RR}{\mathbb{R}}
\newcommand{\NN}{\mathbb{N}}
\newcommand{\PP}{\mathbb{P}}
\newcommand{\EE}{\mathbb{E}}
\newcommand{\II}{\mathbb{I}}
\newcommand{\R}{\mathbb{R}}
\newcommand{\eps}{\epsilon}

\newcommand{\cA}{\mathcal{A}} \newcommand{\cB}{\mathcal{B}} \newcommand{\cC}{\mathcal{C}}
\newcommand{\cD}{\mathcal{D}} \newcommand{\cE}{\mathcal{E}} \newcommand{\cF}{\mathcal{F}}
 \newcommand{\cH}{\mathcal{H}} \newcommand{\cI}{\mathcal{I}}
  
\newcommand{\cM}{\mathcal{M}} \newcommand{\cN}{\mathcal{N}} 
  
\newcommand{\cS}{\mathcal{S}}  \newcommand{\cU}{\mathcal{U}}
\newcommand{\cV}{\mathcal{V}}

\newcommand{\bbeta}{\boldsymbol{\beta}}
\newcommand{\bSigma}{\boldsymbol{\Sigma}}
\newcommand{\bTheta}{\boldsymbol{\Theta}}
\newcommand{\bmu}{\boldsymbol{\mu}}
\newcommand{\bv}{\boldsymbol{v}}
\newcommand{\br}{\boldsymbol{r}}

\def\bLambda{\mathbf{\Lambda}}

\def\eb{\boldsymbol{e}}

\def\Ab{\mathbf{A}}

\def\Ib{\mathbf{I}}
\def\Bb{\mathbf{B}}

\def\bcS{\boldsymbol{\mathcal{S}}}

\let\hat\widehat
\let\tilde\widetilde
\def\bS{\boldsymbol{S}}

\let\emptyset\varnothing

\def\Rb{\mathbf{R}}
\newcommand{\bel}{\begin{eqnarray}\label}
\newcommand{\eel}{\end{eqnarray}}
\newcommand{\bes}{\begin{eqnarray*}}
\newcommand{\ees}{\end{eqnarray*}}

\def\##1\#{\begin{align}#1\end{align}}
\def\$#1\${\begin{align}#1\end{align}}

\newcolumntype{M}{>{\centering\arraybackslash}m{0.17\textwidth}}

\long\def\comment#1{}

\def\thetab{\boldsymbol{\theta}}

\newtheorem*{example*}{Example}

\newtheorem{example}{Example}

\endlocaldefs

\begin{document}

\begin{frontmatter}

\title{Post-Learning Inference for Combinatorial Optimizers with High-Dimensional Sparse Contextual Information via Minimal Directional Perturbation}
\runtitle{Post-Learning Inference for Combinatorial Optimizers}

\begin{aug}
\author[A]{\fnms{Pengyu}~\snm{Li}\thanksref{t1}\ead[label=e1]{pengyuli@u.nus.edu}}
\and
\author[A]{\fnms{Shuting}~\snm{Shen}\thanksref{t1,t2}\ead[label=e2]{shuting\_shen@nus.edu.sg}}

\address[A]{Department of Statistics and Data Science,
National University of Singapore\printead[presep={,\ }]{e1,e2}}

\thankstext{t1}{The authors are listed in alphabetical order.}
\thankstext{t2}{Corresponding author.}
\runauthor{P. Li and S. Shen}
\end{aug}
\begin{abstract}
We study post-learning inference for structural properties of data-dependent combinatorial optimizers. The target is whether an oracle optimizer, rather than a latent parameter or smooth functional, belongs to a prescribed class, such as a category-mix, inventory, or resource-feasibility class. We focus on a high-dimensional contextual multinomial logit model with sequentially adaptive data collection, where the parameter-to-optimizer map is discontinuous and the policy induces temporal dependence. We propose a novel perturbation test based on a nonsmooth max-difference revenue statistic comparing the best null assortment with the best alternative assortment. The test perturbs the estimated terminal revenue surface on the selected support: random unit directions capture directional uncertainty, while the minimal perturbation radius captures magnitude uncertainty and yields a p-value. This localizes inference near the null--alternative boundary and avoids uniform error control over the full candidate class. The data are collected by an \(\ell_1\)-penalized online likelihood policy that performs variable selection while controlling regret. Using a new anti-concentration argument for Gaussian maxima differences and martingale Gaussian coupling, we establish uniform estimation rates, effective support recovery, and asymptotic validity of the proposed p-value under adaptive assortment selection. We prove asymptotic size control and power consistency under a localized signal condition.
\end{abstract}


\end{frontmatter}


\section{Introduction}\label{sec:intro}

Many modern statistical decision problems involve learning an objective from data and then selecting an optimizer from a large combinatorial class. In such settings, the inferential target is often not the latent parameter, nor a smooth low-dimensional functional of it, but a property of the optimizer itself: whether an optimal action belongs to a prescribed structural class, satisfies a resource rule, or is compatible with a deployment constraint.

This is an irregular post-selection problem. The map from the latent parameter to the optimizer is discontinuous; near ties, active maximizers may change under arbitrarily small perturbations. The statistic determining the hypothesis is naturally a difference of two maxima rather than a smooth functional. Adaptive data collection further introduces temporal dependence, since the observations used for terminal inference are generated under decisions that depend on earlier outcomes. In modern applications, the ambient feature dimension may also be large, while only a sparse subset of variables is relevant. Standard Wald-type, delta-method, and fixed-design high-dimensional inference tools therefore do not directly apply.

This paper develops an inferential framework for this irregular high-dimensional post-learning problem. We use contextual assortment optimization as a concrete and practically important model, while the max-difference perturbation principle can extend to other post-learning inference problems for combinatorial optimizers.

Assortment optimization is central in revenue management and personalized recommendation, where a seller or platform selects a subset of products to maximize expected revenue subject to operational constraints. Classical work studies retail assortment design and operational restrictions \citep{Cachon2005,Mantrala2009,Kok2015}, while later work considers richer choice and ranking preference models \citep{Blanchet2016,Aouad2018}. A standard model for customer choice is the random utility model, with the multinomial logit (MNL) model being especially prominent because of its analytical tractability \citep{McFadden1973,Talluri2004,Gallego2004}. In modern platforms, product attractiveness may vary with item attributes, user context, and time-varying information, motivating personalized and contextual MNL formulations \citep{Golrezaei2014,Cheung2017,Chen2020context}.

Most existing work on MNL assortment optimization focuses on learning the optimal assortment and minimizing regret. In online assortment optimization, the platform sequentially updates its estimate of customer preferences and chooses assortments to balance exploration and exploitation under capacity or structural constraints \citep{Caro2007,Rusmevichientong2010,Saure2013,Agrawal2017,ChenWang2018,Agrawal2019,Chen2020context}. This literature has been extended to robust settings under model misspecification \citep{chen2019robust}, personalized decisions using customer-level features \citep{BesbesZeevi2015,Golrezaei2014,Cheung2017}, and contextual or recommendation-based settings with richer side information \citep{ChenMaSimchiLeviXin2024,ChenOwenPixtonSimchiLevi2022}. These works provide tools for sequential decision-making. Our problem is different: after adaptive learning produces a terminal objective, we ask for a valid p-value for a structural property of the oracle optimizer.

Such inferential questions arise naturally in practice. A decision maker may not need to recover the exact revenue-maximizing assortment, especially when several assortments have nearly identical revenues. Instead, the relevant question may be whether the optimal assortment contains a designated collection of core products, preserves category coverage, satisfies an operational or inventory restriction, or excludes a product class without a statistically meaningful revenue loss \citep{Cachon2005,Mantrala2009,Kok2015,Talluri2004,Rusmevichientong2010,ChenMaSimchiLeviXin2024,shen2023combinfassort}. Each question asks whether the optimizer belongs to a prescribed subclass of feasible actions. Equivalently, it asks for inference on the sign of the gap between the best action inside the structural class and the best action outside it. This max-difference representation is the statistical object at the center of the paper.

The closest related works on assortment-level inference are \cite{shen2023combinfassort} and \cite{belloni2025mnlinf}. \cite{shen2023combinfassort} develop an offline combinatorial inference framework for the optimal assortment under the uncapacitated MNL model. Their approach exploits the revenue-ordered structure of the classical MNL model \citep{Talluri2004}, reducing the problem to inference on sign changes in revenue gaps. This structure does not extend to general constrained MNL settings \citep{Rusmevichientong2010,ChenMaSimchiLeviXin2024} or to contextual MNL models, where utilities depend on time-varying contextual information \citep{Chen2020context}. Moreover, the offline setting does not capture the dependence induced by online learning, where each offered assortment depends on past outcomes through the current estimate of customer preferences. \cite{belloni2025mnlinf} study post-learning online inference for the MNL model in a low-dimensional contextual setting and propose an \(\epsilon\)-net procedure for inference on the optimal assortment. Their method does not address the high-dimensional sparse regime considered here. In addition, their inference relies on uniform error control over all candidate assortments; even with the proposed approximation procedure, this global control can be conservative and may limit power near the decision boundary.

We consider a high-dimensional contextual MNL model under adaptive data collection. At each period \(t=1,\ldots,T\), for each product \(j=1,\ldots,n\), we observe a contextual vector \(\bv_{tj}\in\RR^p\) and a revenue \(r_{tj}\in\RR\), with \(0\) denoting the no-purchase option. The vector \(\bv_{tj}\) may encode item attributes, merchandising variables, customer- or session-level context, and interaction features \citep{Golrezaei2014,Cheung2017,Chen2020context}. We allow \(p\) to be much larger than \(T\), reflecting rich feature dictionaries constructed from transaction logs, catalog metadata, clickstream records, and interaction histories \citep{Chen2020context,JiangLiZhang2024,GillenMonteroMoonShum2019}. The utility weight of product \(j\) at time \(t\) is modeled as
\[
    u_{tj}^*=\exp(\bv_{tj}^{\top}\bbeta^*),
\]
where \(\bbeta^*\in\RR^p\) is an unknown sparse coefficient vector. For any offered assortment \(\cS\subseteq[n]\), the choice probability is
\[
\PP_{\bbeta^*,\bv_t}(j\mid \cS)
=
\frac{\exp(\bv_{tj}^{\top}\bbeta^*)}
{1+\sum_{j'\in\cS}\exp(\bv_{tj'}^{\top}\bbeta^*)},
\qquad j\in\cS.
\]
The sparsity assumption is a structural statistical restriction: many candidate features may be available, but only a relatively small subset is expected to have first-order predictive relevance. Parsimonious demand models have long been useful in retail applications \citep{GuadagniLittle1983}, and sparsity regularization is standard for stabilization and feature selection in high-dimensional generalized linear and discrete choice models \citep{vanDeGeer2008,NegahbanOhThekumparampilXu2018,GillenMonteroMoonShum2019}.

As choice outcomes are collected, the platform updates its estimate of \(\bbeta^*\) and adaptively selects the offered assortment \(\cS_t\). At the terminal period \(T\), our inferential object is the oracle optimal assortment \(\cS_T^*\), defined as a maximizer of the expected revenue under the realized terminal contextual information \((\bv_T,\br_T)\). For clarity, we first describe the testing problem in the case of a unique optimizer; the formal formulation in Section~\ref{sec:setup} treats the optimizer as a set and allows for ties. Let \(\bcS^K\) denote the feasible class of candidate assortments satisfying the cardinality constraint \(K\le n\), and let \(\bcS_0\subseteq\bcS^K\) denote the subclass satisfying a structural property of interest. We study
\[
H_0:\ \cS_T^*\in\bcS_0
\qquad\text{versus}\qquad
H_1:\ \cS_T^*\notin\bcS_0,
\]
conditional on \((\bv_T,\br_T)\) and using the data collected before period \(T\). The subclass \(\bcS_0\) may encode any discrete structural property of the feasible assortment class. The following example provides a concrete formulation.

\begin{example*}[Product inclusion test]
Let \(\cA\subseteq[n]\) be a designated set of products. We test whether all products in
\(\cA\) are included in the optimal assortment at time \(T\):
\[
H_0:\ \cA\subseteq\cS_T^*
\qquad\text{versus}\qquad
H_1:\ \cA\not\subseteq\cS_T^* .
\]
Equivalently, this corresponds to \(\cS_T^*\in\bcS_0\), where
\[
\bcS_0:=\{\cS\in\bcS^K:\cA\subseteq\cS\}.
\]
\end{example*}

The main statistical challenge is that the hypothesis concerns the discontinuous map
\[
\bbeta\mapsto
\argmax_{\cS\in\bcS^K}R(\cS\mid\bbeta,\bv_T,\br_T),
\]
where \(R(\cS\mid\bbeta,\bv_T,\br_T)\) is the expected revenue of assortment \(\cS\) and the candidate class \(\bcS^K\) may be exponentially large. A natural benchmark is to control the maximal plug-in revenue error
\begin{equation}\label{eq: max rev diff}
\max_{\cS\in\bcS^K}
\left|
R(\cS\mid\hat\bbeta,\bv_T,\br_T)
-
R(\cS\mid\bbeta^*,\bv_T,\br_T)
\right|.
\end{equation}
A first-order approximation and Gaussian or bootstrap calibration of this statistic can be used to build a confidence set for \(\cS_T^*\). Such uniform-error calibration is conceptually simple, but it is typically conservative for the present testing problem: it calibrates uncertainty over the full feasible class, including assortments that are far from optimal and irrelevant for deciding whether the best null assortment can compete with the best alternative assortment. Moreover, confidence-set inversion gives only a binary decision, whereas a p-value provides a graded measure of evidence.

Our approach uses the max-difference structure directly. Under a unique optimizer, the event \(\cS_T^*\in\bcS_0\) is equivalent to
\[
\max_{\cS\in\bcS_0}
R(\cS\mid\bbeta^*,\bv_T,\br_T)
-
\max_{\cS\notin\bcS_0}
R(\cS\mid\bbeta^*,\bv_T,\br_T)
\ge 0 .
\]
Thus the combinatorial testing problem reduces to inference on the sign of a nonsmooth max-difference functional, avoiding construction of a confidence set for the full optimizer.

To estimate this max-difference functional, we learn \(\bbeta^*\) from the adaptively collected choice data using an \(\ell_1\)-penalized local likelihood procedure. Let \(\cI\subseteq[p]\) denote the selected support, with \(\hat s=|\cI|\). On this support, we construct a debiased estimator and quantify its local uncertainty through random directional perturbations. Specifically, we draw \(\zeta_1,\ldots,\zeta_m\) independently and uniformly from the unit sphere in \(\RR^{\hat s}\). For a perturbation radius \(a\ge0\), the perturbed max-difference along direction \(\zeta_i\) is
\[
\max_{\cS\in\bcS_0}
\left\{
R(\cS\mid\tilde\bbeta^d,\bv_T,\br_T)
+
a\,\hat g_{\cS}^{\top}\widehat\bTheta^{1/2}\zeta_i
\right\}
-
\max_{\cS\notin\bcS_0}
\left\{
R(\cS\mid\tilde\bbeta^d,\bv_T,\br_T)
+
a\,\hat g_{\cS}^{\top}\widehat\bTheta^{1/2}\zeta_i
\right\},
\]
where \(\tilde\bbeta^d\) is the debiased estimator, \(\hat g_{\cS}\) is the estimated revenue gradient restricted to the selected support, and \(\widehat\bTheta\) is the estimated inverse information matrix on that support. When the plug-in max-difference is negative, the observed evidence favors the alternative. We increase \(a\) until the perturbed max-difference becomes compatible with the null boundary, up to tolerance \(\kappa\), along at least one sampled direction. The resulting minimal perturbation radius, denoted by \(\cU_T\), forms the basis of our test statistic.

Using the minimal perturbation radius, we define the p-value as
\[
p_m
:=
\bigl(\bar F_{\chi^2_{\hat{s}}}(\cU_T^2)+\delta_m\bigr)\wedge1,
\]
where \(\bar F_{\chi^2_{\hat{s}}}\) denotes the upper-tail probability of a \(\chi^2_{\hat{s}}\) random variable, and \(\delta_m\) accounts for the directional uncertainty induced by the finite random-direction approximation. The radial magnitude is calibrated through the \(\chi^2_{\hat{s}}\) distribution, while the angular component is controlled by the random directional discretization. This separates magnitude and directional uncertainty and focuses calibration on the max-difference boundary rather than on a uniform error bound over the full candidate class.

The max-difference perturbation principle is not specific to contextual MNL. More generally, consider a combinatorial class \(\cC\) with objective values \(\{\mu_{\cS}^*\}_{\cS\in\cC}\), an optimizer 
\[
    \cS^*\in\arg\max_{\cS\in\cC}\mu_{\cS}^*,
\]
and a structural null class \(\bcS_0\subsetneq\cC\). The testing problem \(\cS^*\in\bcS_0\) is governed by the sign of
\[
    \max_{\cS\in\bcS_0}\mu_{\cS}^*
    -
    \max_{\cS\in\cC\setminus\bcS_0}\mu_{\cS}^* .
\]
The perturbation approach applies whenever the learned score surface admits a local expansion
\[
    \hat\mu_{\cS}
    =
    \mu_{\cS}^*
    +
    \nu_{\cS}^{\top}(\hat\thetab-\thetab^*)
    +
    r_{\cS},
    \qquad \cS\in\cC,
\]
with the higher-order remainders \(r_{\cS}\) controlled over relevant near-maximizers and with \(\hat\thetab-\thetab^*\) admitting a Gaussian coupling after normalization or debiasing. The contextual MNL model studied here is a nontrivial instance in which this expansion must be justified under adaptive data collection, high-dimensional sparsity, and a nonlinear combinatorial revenue surface.

Our contributions are as follows. First, we formulate post-learning combinatorial inference as inference on the sign of a nonsmooth max-difference functional. In the contextual MNL assortment model, this yields a general p-value procedure for testing whether the oracle terminal optimizer satisfies an arbitrary discrete structural property. The data collection policy is based on online \(\ell_1\)-penalized local likelihood estimation, and we establish a sublinear regret guarantee showing that terminal inference can be embedded in online learning without relying on a purely exploratory phase.

Second, we develop theoretical tools for high-dimensional inference after adaptive selection of both the data and the optimizer. We derive convergence rates and effective support recovery guarantees for the online penalized estimator, together with a debiased expansion on the selected support. The analysis combines Gaussian smoothing, martingale arguments for adaptively collected choice data, and anti-concentration for differences of maxima \citep{belloni2024anticon}. These ingredients allow us to control the nonsmooth max-difference revenue functional that determines the structural hypothesis.

Third, we establish the validity and power of the proposed p-value. The validity proof relies on a decoupling of directional and magnitude uncertainty, together with martingale Gaussian coupling techniques \citep{maias2025mtgcoupling}. The power analysis exploits localization: instead of controlling revenue errors uniformly over all candidate assortments, it restricts attention to assortments that can plausibly maximize the null or alternative revenue surface near the boundary. This yields a weaker signal-strength requirement than uniform-error calibration, especially when the effective support size is small relative to the logarithmic size of the full candidate class. Numerical experiments confirm that the proposed p-value substantially improves power over confidence-set procedures based on maximal uniform revenue error.

\subsection{Related literature}

Our work is connected to several strands of literature. The first is high-dimensional post-regularization inference and debiased estimation \citep{ZhangZhang2014,JavanmardMontanari2014,vandeGeer2014}. In that literature, the target is typically a low-dimensional coordinate or a smooth functional of an unknown parameter. Here the target is selected through a nonsmooth combinatorial optimization map, and the hypothesis is governed by a max-difference functional.

The second strand is high-dimensional Gaussian approximation and bootstrap theory for maximal statistics. The works \citep{cck2013aos,confidencebands2014aos,CHERNOZHUKOV20163632,cck2017aop,cck2022improvedbootstrap,chernozhukov2019inference} develop tools for approximating maxima or suprema of empirical processes and for testing many moment inequalities. These tools underlie many uniform-error approaches to high-dimensional and combinatorial inference. Our statistic differs because the null is expressed through a difference of two maximization operators rather than a single coordinatewise maximum. The resulting functional is nonconvex and requires anti-concentration tools for differences of maxima \citep{belloni2024anticon}.

The third strand is inference for discrete optimizers. The closest assortment-specific work is \cite{shen2023combinfassort}, which studies offline inference for the uncapacitated MNL model and exploits revenue ordering to reduce the problem to sign changes in revenue gaps. More broadly, \cite{zhang2025winnersconfidencediscreteargmin} construct confidence sets for the argmin index set of a noisy high-dimensional vector using sample splitting and a soft-min device. Their framework is developed for an IID sampling model, with sample splitting used to remove dependence created by competitor selection. Our setting is sequential and adaptive, and the optimizer is a context-dependent combinatorial assortment rather than the argmin of a fixed mean vector.

Our problem is also connected to inference after data-driven selection. \cite{AndrewsKitagawaMcCloskey2024} study inference on winners, accounting for the selection bias created by choosing the empirically best candidate among finitely many alternatives. Post-selection and selective inference methods provide validity after model selection through simultaneous guarantees over selectable models \citep{BerkBrownBujaZhangZhao2013} or by conditioning on the realized selection event, as in exact selective inference for the lasso \citep{LeeSunSunTaylor2016}. Our setting differs because the selected object is a context-dependent combinatorial optimizer, the data are collected adaptively, and uncertainty is calibrated through directional perturbations of a nonsmooth max-difference revenue functional rather than by conditioning on a selection event.

The fourth strand is inference with adaptively collected data. \cite{zhang2022statisticalinferenceadaptivesampling} develop post-adaptive-sampling inference for longitudinal data through a Z-estimation framework, while \cite{chen2022onlinestatisticalinferencecontextual} study online inference for contextual bandit parameters via weighted stochastic gradient descent. In these works, the target is a model parameter or a low-dimensional functional under an adaptive design. Here the target is a property of a nonsmooth combinatorial optimizer. The adaptive policy induces both temporal dependence through sequential data collection and combinatorial dependence through repeated maximization over a large feasible class.

Finally, the contextual MNL model connects the paper to online assortment optimization, contextual choice modeling, and discrete-choice inference. Existing work has developed regret-minimizing policies for MNL assortment optimization under capacity or structural constraints \citep{Caro2007,Rusmevichientong2010,Saure2013,Agrawal2017,ChenWang2018,Agrawal2019,Chen2020context}, as well as robust and contextual extensions \citep{chen2019robust,BesbesZeevi2015,Golrezaei2014,Cheung2017,ChenMaSimchiLeviXin2024,ChenOwenPixtonSimchiLevi2022}. Spectral, likelihood-based, and regularized methods have also been studied for the Bradley--Terry--Luce model and related ranking models, often with sharp estimation or minimax guarantees \citep{NegahbanOhShah2017,ShahBalakrishnanWainwright2015,ChenFanMaWang2019,NegahbanOhThekumparampilXu2018}, and recent work develops uncertainty quantification for ranking scores \citep{GaoShenZhang2023,liu2023lagrangian}. These papers focus primarily on learning policies, latent parameters, or item-level scores; our focus is inference on a structural property of the optimizer after adaptive learning.

\subsection*{Notation}
For a positive integer \(n\), write \([n]=\{1,\ldots,n\}\). For a subset \(\cS\subseteq[n]\), write \(\cS_+=\cS\cup\{0\}\). For an index set \(\cI\subseteq[p]\), \([x]_{\cI}\) denotes the subvector of \(x\) indexed by \(\cI\), and \([A]_{\cI_1,\cI_2}\) denotes the submatrix of \(A\) with rows in \(\cI_1\) and columns in \(\cI_2\). For a vector \(x\), \(\|x\|_q\) denotes the vector \(\ell_q\) norm, with \(\|x\|_\infty=\max_j |x_j|\), and \(\|x\|_0\) denotes the number of nonzero coordinates. For a matrix \(A=(A_{ij})\), \(\|A\|_{\max}:=\max_{i,j}|A_{ij}|\) denotes the entrywise maximum norm, \(\|A\|_\infty:=\max_i\sum_j |A_{ij}|\) denotes the induced \(\ell_\infty\) operator norm, \(\|A\|_2\) denotes the spectral norm, and \(\lambda_{\min}(A)\) denotes the smallest eigenvalue of a symmetric matrix. \(\II(\cdot)\) denotes the indicator function. We write \(a_n\lesssim b_n\) if \(a_n\le Cb_n\) for an absolute constant \(C>0\), \(a_n\asymp b_n\) if both \(a_n\lesssim b_n\) and \(b_n\lesssim a_n\) hold, and \(a_n\ll b_n\) if \(a_n/b_n\to0\).

\subsection*{Paper Organization}
The remainder of the paper is organized as follows. Section~\ref{sec:setup} introduces the model and hypothesis testing framework. Section~\ref{sec:method} develops the estimation, debiasing, and testing procedure. Section~\ref{sec:theory} presents the main theoretical results. Section~\ref{sec:numerical} reports simulation studies, followed by concluding remarks.
\section{Problem Setup}\label{sec:setup}
Consider $n$ products indexed by $[n]=\{1,\ldots,n\}$, with $0$ denoting the no-purchase option. At each time $t \in [T]$, we observe item-specific contextual features $\bv_t := \{\bv_{tj}\}_{j \in [n]} \in \RR^{n \times p}$ and revenues $\br_t := \{r_{tj}\}_{j \in [n]} \in \mathbb{R}^n$. For normalization, we set $\bv_{t0}=\mathbf{0}$ for all $t \in [T]$. Based on the past customer choice outcomes and the current contextual information $(\bv_t,\br_t)$, we select an assortment $\cS_t \subseteq [n]$ according to a data-driven policy and offer it to the incoming customer. The resulting choice outcome is denoted by $i_t \in \cS_t \cup \{0\}$.

We adopt a contextual MNL model, under which the choice outcome follows a multinomial distribution:
\begin{equation}
\label{eq:model}
\PP_{\bbeta^*,\bv_t}(i_t = j \mid \cS_t)
=
\frac{u_{tj}^*}{1+\sum_{j' \in \cS_t} u_{tj'}^*}
=
\frac{\exp\{\bv_{tj}^{\top}\bbeta^*\}}{1+\sum_{j' \in \cS_t}\exp\{\bv_{tj'}^{\top}\bbeta^*\}}, \quad \forall j \in \cS_t \cup \{0\},
\end{equation}
where $\bbeta^* \in \mathbb{R}^p$ is the unknown parameter and
$$
u_{tj}^* = \exp\{\bv_{tj}^{\top}\bbeta^*\}
$$
is the utility weight of item $j$ at time $t \in [T]$.
We work in a high-dimensional regime that allows \(p\gg T\), and assume that \(\bbeta^*\) is sparse, with support size \(s=\|\bbeta^*\|_0\ll p\).

To reflect practical constraints such as limited display or window size, we restrict the feasible assortments to
\[
\bcS^K \subseteq \{\cS \subseteq [n] : 1 \le |\cS| \le K\},
\]
where \(K\) is the maximum cardinality, and we assume without loss of generality that \(\max_{\cS \in \bcS^K} |\cS| = K\).
Under \eqref{eq:model}, for any offered assortment $\cS \in \bcS^K$, the expected revenue at time $t \in [T]$, evaluated at $\bbeta \in \RR^p$ and conditional on the contextual information $(\bv_t,\br_t)$, is
\begin{equation}\label{eq: exp rev}
R(\cS \mid \bbeta, \bv_t, \br_t)
:=
\EE_{\bbeta}\bigl(r_{t,i_t} \mid \cS, \bv_t, \br_t\bigr)
=
\frac{\sum_{j \in \cS} r_{tj}\exp\{\bv_{tj}^{\top}\bbeta\}}{1+\sum_{j \in \cS}\exp\{\bv_{tj}^{\top}\bbeta\}},
\end{equation}
where
$$
r_{t,i_t}
=
\sum_{j \in \cS_t \cup \{0\}} r_{tj}\,\II(i_t=j)
$$
denotes the realized revenue at time $t$, with $r_{t0}=0$ corresponding to the no-purchase option.

We are interested in the assortments that maximize expected revenue, allowing for the possibility of ties among maximizers. For any $\bbeta \in \mathbb{R}^p$, let
\begin{equation}
\cS_t(\bbeta)
:=
\argmax_{\cS \in \bcS^K} R(\cS \mid \bbeta, \bv_t, \br_t)
\end{equation}
denote the set of revenue-maximizing assortments at time $t$. Note that $\cS_t(\bbeta)$ depends on the contextual information $(\bv_t,\br_t)$. When the maximizer is unique, so that $\cS_t(\bbeta)$ is a singleton, we abuse notation slightly and use $\cS_t(\bbeta)$ to denote the unique maximizing assortment. We write $\cS_t^* := \cS_t(\bbeta^*)$ for the corresponding maximizing set under the true parameter.

Our inferential target is the maximizing set at the terminal time point $T$. Specifically, conditional on the contextual information of the incoming customer at time $T$, we consider the general hypothesis test
\begin{equation}\label{eq: hypo test}
H_0:\ \cS_T^* \cap \bcS_0 \neq \emptyset
\qquad \text{versus} \qquad
H_1:\ \cS_T^* \cap \bcS_0 = \emptyset,
\quad \text{given } (\bv_T,\br_T),
\end{equation}
where $\bcS_0 \subseteq \bcS^K$ denotes the class of assortments satisfying the structural constraint of interest. In other words, given the information on the incoming customer at time $T$, we ask whether the structural property encoded by $\bcS_0$ can be achieved without sacrificing optimal revenue, or equivalently, whether it has the potential to be revenue-optimal.

The product inclusion test in Section~\ref{sec:intro} provides one concrete instance of the general testing problem in \eqref{eq: hypo test} when the optimal assortment is unique. We restate it below in the more general setting that allows for ties, and then present several additional formulations to illustrate possible forms of $\bcS_0$.

\begin{example}[Product inclusion test]\label{exm: subset test}
Let $\cA \subseteq [n]$ denote a designated set of products. Conditional on $(\bv_T,\br_T)$, we test whether there exists a revenue-maximizing assortment at time $T$ that includes all products in $\cA$:
$$
H_0:\ \exists \cS \in \cS_T^* \text{ such that } \cA \subseteq \cS
\qquad\text{versus}\qquad
H_1:\ \forall \cS \in \cS_T^*,\ \cA \not\subseteq \cS.
$$
This corresponds to testing $\cS_T^* \cap \bcS_0 \neq \emptyset$, where
$$
\bcS_0 := \{\cS \in \bcS^K : \cA \subseteq \cS\}.
$$
\end{example}
The following examples provide additional formulations of \eqref{eq: hypo test}.
\begin{example}[Category proportion test]\label{exm: category prop test}
Let $\cA \subseteq [n]$ denote a given product category, and let $q\%$ be a prescribed threshold. Conditional on $(\bv_T,\br_T)$, we test whether there exists a revenue-maximizing assortment at time $T$ for which more than $q\%$ of the offered products belong to $\cA$:
$$
H_0:\ \exists \cS \in \cS_T^* \text{ such that } \frac{|\cA \cap \cS|}{|\cS|} > q\%
\qquad \text{versus} \qquad
H_1:\ \forall \cS \in \cS_T^*,\ \frac{|\cA \cap \cS|}{|\cS|} \le q\%.
$$
This corresponds to testing $\cS_T^* \cap \bcS_0 \neq \emptyset$, where
$$
\bcS_0 := \left\{\cS \in \bcS^K : \frac{|\cA \cap \cS|}{|\cS|} > q\%\right\}.
$$
This formulation is relevant when assortment decisions are guided by category balance, variety-depth targets, or shelf-space allocation across product families \citep{Mantrala2009,Kok2015}.
\end{example}

\begin{example}[Feature test]\label{exm: feature test}
Let $\cV \subseteq \RR^{p+1}$ denote a prescribed set of item-level feature vectors. Conditional on $(\bv_T,\br_T)$, we test whether there exists a revenue-maximizing assortment at time $T$ whose products all have augmented feature vectors lying in $\cV$:
\begin{align*}
H_0:&\ \exists \cS \in \cS_T^* \text{ such that } (\bv_{Tj}^{\top}, r_{Tj})^{\top} \in \cV \quad \text{for all } j \in \cS,\\
H_1:&\ \forall \cS \in \cS_T^*,\ \text{there exists } j \in \cS \text{ such that } (\bv_{Tj}^{\top}, r_{Tj})^{\top} \notin \cV.
\end{align*}
This corresponds to testing $\cS_T^* \cap \bcS_0 \neq \emptyset$, where
$$
\bcS_0 := \{\cS \in \bcS^K : (\bv_{Tj}^{\top}, r_{Tj})^{\top} \in \cV \text{ for all } j \in \cS\}.
$$
This formulation is relevant when decision makers require all items in a candidate optimal assortment to satisfy attribute screens such as price range, freshness, compatibility, or merchandising profile \citep{Golrezaei2014,Chen2020context}.
\end{example}

As illustrated by the concrete formulations in Examples~\ref{exm: subset test}--\ref{exm: feature test}, the constraint class $\bcS_0$ is combinatorial. A direct approach based on constructing a confidence set for the maximizing set $\cS_T^*$ is therefore both technically challenging and statistically inefficient. Indeed, since $\cS_T^*$ is a set-valued argmax that may contain multiple tied maximizers, valid confidence-set construction would require uniform error control over a potentially very large class of candidate assortments. Such an approach is typically driven by a maximal error statistic over $\bcS^K$, which is inherently conservative, and it does not naturally yield a scalar p-value for the structural hypothesis of interest. 

Instead, we exploit the equivalence of \eqref{eq: hypo test} to the following revenue-gap formulation:
\begin{equation}\label{eq: equiv test}
H_0:\quad
\max_{\cS\in\bcS_0} R_{T,\cS}^*-\max_{\cS\notin\bcS_0} R_{T,\cS}^* \ge 0
\qquad\text{versus}\qquad
H_1:\quad
\max_{\cS\in\bcS_0} R_{T,\cS}^*-\max_{\cS\notin\bcS_0} R_{T,\cS}^* < 0,
\end{equation}
given $(\bv_T,\br_T)$, where $R_{T,\cS}^* := R(\cS \mid \bbeta^*, \bv_T, \br_T)$. Thus, \eqref{eq: equiv test} reduces inference on a combinatorial structural constraint to inference on the sign of a scalar max-difference in revenues. However, this max-difference remains a nonconvex and nonsmooth functional of the latent revenue surface. The technical tools we develop to address this inferential problem are introduced in the next section.
\begin{remark}
The formulation in \eqref{eq: hypo test} focuses on whether a decision maker can impose the structural constraint encoded by $\bcS_0$ without sacrificing optimal revenue. In some applications, however, the relevant practical question is stronger: whether every optimal assortment must satisfy the structural property. In that case, the hypothesis is reformulated as
$$
H_0:\ \cS_T^* \subseteq \bcS_0
\qquad \text{versus} \qquad
H_1:\ \cS_T^* \not\subseteq \bcS_0,
\quad \text{given } (\bv_T,\br_T).
$$
This null hypothesis is stronger than \eqref{eq: hypo test}. Indeed, it is equivalent to
$$
H_0:\quad
\max_{\cS\in\bcS_0} R_{T,\cS}^*-\max_{\cS\notin\bcS_0} R_{T,\cS}^* > 0
\qquad\text{versus}\qquad
H_1:\quad
\max_{\cS\in\bcS_0} R_{T,\cS}^*-\max_{\cS\notin\bcS_0} R_{T,\cS}^* \le 0.
$$
 With suitable technical modifications, the proposed procedure can also be adapted to this stronger testing problem. In the present paper, however, we focus on the weaker formulation in \eqref{eq: hypo test}, which is more directly aligned with the question of whether the structural constraint is compatible with revenue optimality.
\end{remark}

\section{Method}\label{sec:method}

We begin by introducing the dynamic policy used to adaptively select assortments and update the coefficient estimator, and then present the post-learning inferential procedure.

\subsection{Dynamic Policy}\label{sec: policy}
We first assume that a pilot estimator $\hat\bbeta_0$ is available, obtained either from prior online exploration \cite{Chen2020context} or from offline historical data \cite{shen2023combinfassort}, such that
\begin{equation}\label{eq: beta 0 rate}
\|\hat\bbeta_0-\bbeta^*\|_1 \le \tau,
\end{equation}
where $\tau=o(1)$ is assumed known. Unless stated otherwise, all subsequent results are understood to hold on the event \eqref{eq: beta 0 rate}. The pilot estimator plays the role of an initial exploration device, anchoring subsequent exploitation within a suitable neighborhood of the true parameter and thereby mitigating the effect of the combinatorial discontinuity induced by adaptive assortment selection.

With the aid of $\hat\bbeta_0$, at each time point $t \in [T]$ we update the estimator of $\bbeta^*$ through a local maximum likelihood procedure centered near $\hat\bbeta_0$. Specifically, define the negative log-likelihood based on observations up to time $t-1$ by
\[
\ell_{t-1}(\bbeta)
=
-\sum_{t'=1}^{t-1} \log \PP_{\bbeta,\bv_{t'}}\!\big(i_{t'} \mid \cS_{t'}\big)
=
-\sum_{t'=1}^{t-1} \log \left\{
\frac{\exp\{\bv_{t',\,i_{t'}}^{\top}\bbeta\}}{1+\sum_{k\in \cS_{t'}} \exp\{\bv_{t'k}^{\top}\bbeta\}}
\right\},
\]
where $\bv_{t',\,i_{t'}}
:=
\sum_{j\in \cS_{t'}\cup\{0\}} \bv_{t'j}\,\II(j=i_{t'})$ denotes the contextual feature vector of the chosen item.
The gradient and Hessian of $\ell_{t-1}(\bbeta)$ will be used later in the inferential procedure:
\begin{align}
\nabla_{\bbeta}\ell_{t-1}(\bbeta)
&=
-\sum_{t'=1}^{t-1}
\Big\{
\bv_{t',i_{t'}}
-
\EE_{\bbeta,t',\cS_{t'}(\hat\bbeta_{t'-1})}\big(\bv_{t',i_{t'}}\big)
\Big\},
\label{eq: grad}
\\
\nabla_{\bbeta}^2\ell_{t-1}(\bbeta)
&=
\sum_{t'=1}^{t-1}
\Big[
\EE_{\bbeta,t',\cS_{t'}(\hat\bbeta_{t'-1})}
\big(\bv_{t',i_{t'}}\bv_{t',i_{t'}}^\top\big)
\notag\\
&\qquad\qquad
-
\EE_{\bbeta,t',\cS_{t'}(\hat\bbeta_{t'-1})}
\big(\bv_{t',i_{t'}}\big)\,
\EE_{\bbeta,t',\cS_{t'}(\hat\bbeta_{t'-1})}
\big(\bv_{t',i_{t'}}\big)^\top
\Big],
\label{eq: hessian}
\end{align}
where $\EE_{\bbeta,t',\cS_{t'}(\bbeta')}(\cdot)$ denotes expectation with respect to the item draw, conditional on the contextual features $\bv_{t'}$ and the assortment $\cS_{t'}(\bbeta')$ selected under parameter $\bbeta'$:
\[
i_{t'} \sim \PP_{\bbeta,\bv_{t'}}(\cdot \mid \cS_{t'}(\bbeta')).
\]

Then, for each $t \ge 2$, we compute the $\ell_1$-penalized estimator
\begin{equation}\label{eq: theta hat}
\hat\bbeta_{t-1}
\in
\arg\min_{\|\bbeta-\hat\bbeta_0\|_1 \le 2\tau}
\Big\{\ell_{t-1}(\bbeta)+\lambda_{t-1}\|\bbeta\|_1\Big\},
\end{equation}
where $\lambda_{t-1}>0$ will be specified in Theorem~\ref{thm:lasso-rates}. This local maximum likelihood estimator updates the estimate of $\bbeta^*$ online as data are collected sequentially, while the $\ell_1$ penalty performs variable selection for the relevant contextual covariates.

Given $\hat\bbeta_{t-1}$, we select the offered assortment at time $t$ by maximizing the revenue evaluated at $\hat\bbeta_{t-1}$, namely,
$$
\cS_t \in \cS_t(\hat\bbeta_{t-1}).
$$
Any ties are resolved by a deterministic rule, such as by favoring assortments with larger item revenues. We denote by $\cS_t^* \in \cS_t(\bbeta^*)$ an oracle optimal assortment at time $t$ under the true coefficient vector $\bbeta^*$. In Section~\ref{sec:theory}, we show that this policy yields consistent estimation and effective support recovery for $\bbeta^*$, and in turn attains near-optimal regret.

\subsection{Inferential Procedure}\label{sec: inf proc}
As established in Section~\ref{sec:setup}, the general inferential objective is equivalent to testing whether the max revenue gap in \eqref{eq: equiv test} is nonnegative. Although this reformulation reduces the problem to the sign of a scalar functional, the resulting inference problem remains challenging in the post-online-learning setting. The difficulty is that the max revenue gap is still a nonconvex and nonsmooth functional of the unknown coefficient vector, since it is defined through the difference of two maximization operators over combinatorial classes of assortments. In particular, the active maximizing assortments are unknown and may change discontinuously under small perturbations of the parameter, especially near ties. Moreover, the data used for inference are collected adaptively through the online assortment policy, so the estimation error is coupled with both temporal dependence and the combinatorial dependence induced by repeated maximization. As a result, standard smooth-function arguments such as the delta method or Wald-type inference are not applicable, and a more careful post-learning procedure is required.  

To address this difficulty, we exploit the latent sparsity of the true coefficient vector and decompose the uncertainty in the estimator on the selected support into directional and magnitude components. Specifically, let $\hat\bbeta_{T-1}$ denote the estimator in \eqref{eq: theta hat}, and define the selected support by
\[
\cI := \{ j\in[p]:\ \hat\bbeta_{T-1,j}\neq 0\},
\]
with support size $\hat{s}:=|\cI|$. Since the penalization induces bias, we first debias the estimator on the selected support to facilitate subsequent valid inference. To this end, define the one-step debiased estimator $\tilde\bbeta^d$ by setting $\tilde\bbeta^d_j=0$ for $j\notin\cI$ and
\[
[\tilde\bbeta^{d}]_{\cI}
:=
[\hat\bbeta_{T-1}]_{\cI}
-
\Big([\nabla_{\bbeta}^2\ell_{T-1}(\hat\bbeta_{T-1})]_{\cI,\cI}\Big)^{-1}
\Big([\nabla_{\bbeta}\ell_{T-1}(\hat\bbeta_{T-1})]_{\cI}\Big).
\]

We test \eqref{eq: equiv test} by constructing a p-value based on a perturbation of the plug-in revenue-gap statistic evaluated at $\tilde\bbeta^d$. For each $t \in [T]$ and $\cS \in \bcS^K$, the gradient of the expected revenue of assortment $\cS$ with respect to $\bbeta$ is
\begin{equation}\label{eq: gradient revenue}
\nabla_{\bbeta} R(\cS \mid \bbeta,\bv_t,\br_t)
=
\sum_{j\in\cS}
\PP_{\bbeta,\bv_t}(j \mid \cS)\, r_{tj}\,\bv_{tj}
-
\sum_{j,j'\in\cS}
\PP_{\bbeta,\bv_t}(j \mid \cS)\PP_{\bbeta,\bv_t}(j' \mid \cS)\, r_{tj}\,\bv_{tj'},
\end{equation}
where $\PP_{\bbeta,\bv_t}(j \mid \cS)
:=
\frac{\exp(\bv_{tj}^\top\bbeta)}{1+\sum_{j'\in\cS}\exp(\bv_{tj'}^\top\bbeta)}$.

Given a directional accuracy parameter $\epsilon \in (0,1)$, fix $m \in \NN^+$. Let $\zeta_1,\ldots,\zeta_m$ be i.i.d.\ random vectors drawn uniformly from the unit sphere $S^{\hat{s}-1}$. We define the minimal perturbation radius needed to bring the max-difference statistic close to the null boundary along the sampled directions by
\begin{equation}\label{eq: inf lambda}
\begin{aligned}
\cU_T
:=
\inf\Bigg\{a\ge 0:\ 
&\max_{1\le i\le m}
\Big[
\max_{\cS\in\bcS_0}
\big(\hat R_{T,\cS}+a\,\hat g_{\cS}^\top \widehat\bTheta^{1/2}\zeta_i\big)\\
&\qquad\qquad
-
\max_{\cS\notin\bcS_0}
\big(\hat R_{T,\cS}+a\,\hat g_{\cS}^\top \widehat\bTheta^{1/2}\zeta_i\big)
\Big]
\ge -\kappa
\Bigg\},
\end{aligned}
\end{equation}
where $\kappa \ge 0$ is a tuning parameter whose choice will be specified in Theorem~\ref{thm: valid p}. Here
\begin{equation}\label{eq: rev and grad}
    \hat R_{T,\cS}:=R(\cS\mid \tilde\bbeta^{\,d},\bv_T,\br_T),
\qquad
\hat g_{\cS}:=[\nabla_{\bbeta}R(\cS\mid \hat\bbeta_{T-1},\bv_T,\br_T)]_{\cI},
\end{equation}
and
\begin{equation}\label{eq: est fisher mat}
    \widehat\bTheta
:=
\Big([\nabla_{\bbeta}^2\ell_{T-1}(\hat\bbeta_{T-1})]_{\cI,\cI}\Big)^{-1}.
\end{equation}
Thus, $\cU_T$ records the smallest perturbation magnitude for which the plug-in max-difference becomes approximately nonnegative along at least one of the sampled directions.

Based on the minimal perturbation radius $\cU_T$, we define the p-value by
\begin{equation}\label{eq: p value}
p_m
:=
\bigl(\bar{F}_{\chi^2_{\hat{s}}}(\cU_T^2)+\delta_m\bigr)\wedge 1,
\end{equation}
where $\bar{F}_{\chi^2_{\hat{s}}}(\cdot)$ denotes the upper-tail probability of a chi-square random variable with $\hat{s}$ degrees of freedom, and
\begin{equation}\label{eq: delta_m}
    \delta_m
:=
\exp\left(
-m\sqrt{\frac{\pi}{8 \hat{s}}}
\left(\frac{2\epsilon}{\pi}\right)^{\hat{s}-1}
\right)
\end{equation}
is the residual probability that the $m$ sampled directions fail to approximate the true perturbation direction within accuracy $\epsilon$. We will show in Theorem~\ref{thm: valid p} that $p_m$ is asymptotically valid.

Based on the constructed p-value, for a given significance level $\alpha$, we reject $H_0$ if and only if $p_m \le \alpha$.


\section{Theory}\label{sec:theory}

This section establishes theoretical guarantees for the proposed inferential framework. We begin with uniform convergence rates for the penalized estimator, followed by effective support recovery guarantees.
\subsection{Rates of the penalized estimator}

We impose the following assumption on the data-generating mechanism for $\{\bv_t,\br_t\}_{t=1}^T$ to facilitate the theoretical analysis of the temporal and combinatorial dependence induced by the adaptive data collection process.

\begin{assumption}\label{asp: abs cont of r}
The sequence $\{\bv_t,\br_t\}_{t=1}^T$ is i.i.d. across $t \in [T]$. For each $t \in [T]$, $\bv_t \perp\!\!\!\perp \br_t$.
\end{assumption}

Before introducing the remaining technical assumptions, we define some notation used throughout the analysis. For each $t \in [T]$, let
\begin{equation*}
\cH_{t-1}
:=
\big\{\bv_{t'},\br_{t'},i_{t'}\big\}_{t'=1}^{t-1}
\cup \{\hat\bbeta_0\}
\end{equation*}
denote the history available up to time $t-1$, with $\cH_0=\{\hat\bbeta_0\}$. Throughout the remainder of the paper, all probabilistic statements are understood to hold conditionally on $\cH_0$, unless stated otherwise.

For later use, define the conditional covariance matrix
\begin{equation}\label{eq: def Sig_t}
\bSigma_{t'}(\bbeta)
=
\EE_{\bbeta^*,t',\cS_{t'}(\bbeta)}
\big(
\bv_{t',i_{t'}}\bv_{t',i_{t'}}^\top
\big)
-
\EE_{\bbeta^*,t',\cS_{t'}(\bbeta)}
\big(
\bv_{t',i_{t'}}
\big)
\EE_{\bbeta^*,t',\cS_{t'}(\bbeta)}
\big(
\bv_{t',i_{t'}}
\big)^\top .
\end{equation}
This matrix is the covariance counterpart of the Hessian summand in \eqref{eq: hessian}, with the item draw evaluated under the true parameter $\bbeta^*$ and the offered assortment selected according to the plug-in parameter $\bbeta$. It therefore captures the local Hessian structure induced by the adaptive assortment-selection rule.

For a fixed plug-in value $\bbeta$, the conditional mean $\EE\big(\bSigma_{t'}(\bbeta)\mid \bbeta\big)$ averages over the randomness in the contextual information and revenues $(\bv_{t'},\br_{t'})$, while keeping the parameter used for assortment selection fixed at $\bbeta$. We further define the population Hessian matrix at the truth by
\[
\bSigma^*
=
\EE\big(\bSigma_{t'}(\bbeta^*)\mid \bbeta^*\big).
\]
Under Assumption~\ref{asp: abs cont of r}, this matrix does not depend on
\(t'\in[T]\). The following assumption ensures that \(\bSigma^*\) is
nonsingular, that the contextual features are uniformly bounded, and that the
MNL choice probabilities are locally comparable.

\begin{assumption}\label{asp: cov}
There exist constants \(\underline\lambda>0\), \(\nu>0\), and
\(\rho\ge1\) such that
\[
    \lambda_{\min}(\bSigma^*)\ge \underline\lambda,
    \qquad
    \max_{j\in[n],\,t\in[T]}\|\bv_{tj}\|_{\infty}\le \nu .
\]
Moreover, for all \(t\in[T]\), all \(\cS\in\bcS^K\), all
\(\bbeta\in\cB_1(\bbeta^*,3\tau)\), and all \(j,j'\in\cS_+\),
\begin{equation}\label{eq: rho rate}
    {
        \PP_{\bbeta,\bv_t}(j\mid\cS)
    /
        \PP_{\bbeta,\bv_t}(j'\mid\cS)
    }
    \le \rho .
\end{equation}
\end{assumption}

The following assumptions formalize the conditional independence structure of the adaptively collected choice outcomes and impose regularity conditions on the revenue distribution and candidate-assortment class. \begin{assumption}\label{asp: dpd structure}
For each $t\in[T]$, conditional on the current contextual information, revenues, and selected offer set, the customer choice outcome $i_t$ is independent of the past history; that is,
\[
i_t \perp\!\!\!\perp \mathcal H_{t-1}
\mid \bv_t,\br_t,\mathcal S_t .
\]
\end{assumption}
\begin{assumption} \label{asp: assortment const}
 Suppose that the revenue vectors $\{\br_t\}_{t=1}^T$ are i.i.d. across $t$, with $\br_t=(r_{t1},\ldots,r_{tn})^\top\sim N(\bmu_r,\sigma_r^2 \Ib_n)$, and denote $\bar\mu := \|\bmu_r\|_{\infty} \vee \sigma_r  \sqrt{2(\log 2n + 2\log  T)}$.
 Moreover, at least one of the following two conditions is satisfied:
    \begin{enumerate}
        \item $\bcS^K = \{\cS: \cS \subseteq [n], |\cS| = K\}$ and $\rho \le 2$.
        \item $\bcS^K \subseteq \{\cS: \cS \subseteq [n], |\cS| = K\}$ and for any $\cS, \cS' \in \bcS^K$, $|\cS \cap \cS'| \le (K /\rho^2 - 1)\vee 0$.
    \end{enumerate}
    Let $C_n = \bar\mu  K^3 \nu \sigma_r^{-1}  \sqrt{\log n}$ if (1) is satisfied, and $C_n = K \bar\mu   \nu \sigma_r^{-1} \sqrt{\log n}(K/\rho^2 \vee 1)$ if (2) is satisfied.
    \end{assumption}

\begin{remark}[Role of Assumption~\ref{asp: assortment const} and possible relaxations]
Assumption~\ref{asp: assortment const} is a primitive sufficient condition for controlling the dependence induced by adaptive assortment selection. It has two roles. First, it provides a high-probability revenue envelope, uniformly over \(t\in[T]\) and \(j\in[n]\). This part may be replaced by standard sub-Gaussian or bounded-moment tail conditions. Second, it gives local stability of the selection-induced Hessian. More precisely, the proofs only require a local continuity bound for the one-period conditional expected Hessian summand: for all \(\bbeta_1,\bbeta_2\in B_1(\bbeta^*,3\tau)\),
\[
    \left\|
        \mathbb E\{\bSigma_t(\bbeta_1)\mid \bbeta_1\}
        -
        \mathbb E\{\bSigma_t(\bbeta_2)\mid \bbeta_2\}
    \right\|_{\max}
    \le
    \mathfrak L_T\|\bbeta_1-\bbeta_2\|_1+\mathfrak r_T,
\]
where \(\mathfrak r_T\) is a negligible residual. Under Assumption~\ref{asp: assortment const}, Claim~\ref{claim: cont bd} shows that one may take
\[
    \mathfrak L_T\lesssim C_n\nu^2,
    \qquad
    \mathfrak r_T\lesssim \nu^2/T,
\]
which yields the \(C_n\)-dependent rates and conditions in the estimation, debiased-expansion, and p-value validity results below. Thus Assumption~\ref{asp: assortment const} should be viewed as an explicit and verifiable sufficient condition for local Hessian stability of the adaptive policy, rather than as a necessary condition for the perturbation p-value.

The Gaussian revenue assumption can also be relaxed. If the centered revenue vector admits a Stein kernel \(\tau(\cdot)\) with revenue covariance target \(\bLambda_r\) and
\[
    \Delta_r
    :=
    \mathbb E
    \max_{j,k}
    \left|
        \tau_{jk}\{\br_t-\mathbb E(\br_t)\}
        -
        [\bLambda_r]_{jk}
    \right|
\]
is small, then the Gaussian comparison step in Claim~\ref{claim: cont bd} can be replaced by the Stein-kernel comparison in Lemma~5.1 of \cite{belloni2024anticon}, at the cost of an additional max-difference approximation error of order
\(C_r\sqrt{\Delta_r K\log n}\), where \(C_r\) depends on the normalized expected maximum of the corresponding Gaussian assortment-level revenue process and on the structure of \(\bcS^K\).
\end{remark}
The next theorem establishes uniform convergence rates for the local penalized likelihood estimator.
\begin{theorem}[Rates for the penalized estimator]\label{thm:lasso-rates}
Suppose Assumptions~\ref{asp: abs cont of r}, \ref{asp: cov}--\ref{asp: assortment const} hold. Assume further that $\tau \le c\,\underline\lambda\,(C_n\nu^2 s)^{-1}$ for a sufficiently small absolute constant $c>0$, and that
\[
\lambda_{t-1} = C_{\lambda}\,\nu\Bigl\{\log(Tp)+\sqrt{t\log(Tp)}\Bigr\}
\]
for a sufficiently large constant $C_{\lambda}>0$ independent of $t$. Then there exist absolute constants $C,C'>0$ such that, with probability at least $1-O(T^{-1})$, the following bounds hold uniformly over all
\[
t \ge C'\,\nu^4 s^2 \underline\lambda^{-2}\log(Tp):
\]
\begin{align}
\|\hat\bbeta_{t-1}-\bbeta^*\|_2
&\le C\nu \sqrt{\frac{s\log(Tp)}{\underline\lambda^2\,t}}, \label{eq:l2-rate}\\
\|\hat\bbeta_{t-1}-\bbeta^*\|_1
&\le C\nu s \sqrt{\frac{\log(Tp)}{\underline\lambda^2\,t}}. \label{eq:l1-rate}
\end{align}
\end{theorem}
The proof of Theorem~\ref{thm:lasso-rates} is given in Section~\ref{sec: proof thm lasso rates}.

\begin{remark}\label{rmk: constraint removal}
When the time horizon is sufficiently large so that $\nu s \sqrt{\log(Tp)/(\underline\lambda^{2} t)} \le c \tau$ for a sufficiently small constant $c > 0$, the estimator \(\hat\bbeta_{t-1}\) lies in the constraint set
\(\cB_1(\hat\bbeta_0,2\tau)\) and therefore coincides with the unconstrained
\(\ell_1\)-penalized estimator. This observation allows us to use the local
penalized estimator for terminal-time inference.
\end{remark}

Now we define the true support of the coefficient vector
\[
    \cI_0 := \{j\in[p]: \bbeta_j^*\neq 0\}.
\]
To prepare for the effective-support recovery result, we introduce an effective signal-strength and weak-tail sparsity condition, together with a modified mutual incoherence condition.


\begin{assumption}[Effective support and weak-tail sparsity]
\label{asp: min signal}
There exists a decomposition
\[
    \mathcal I_0=\mathcal I_*\cup \mathcal I_{\rm wk},
    \qquad
    \mathcal I_*\cap\mathcal I_{\rm wk}=\emptyset,
\]
where \(\mathcal I_*\) is the effective support and
\(\mathcal I_{\rm wk}:=\mathcal I_0\setminus\mathcal I_*\) is the weak support. Let
\(s_*:=|\mathcal I_*|\). For notational simplicity, we assume throughout that
\(\mathcal I_*\neq\emptyset\).

There exists a sufficiently large constant \(C_\beta>0\) such that
\[
    \min_{j\in\mathcal I_*}|\beta_j^*|
    \ge
    C_\beta\nu\sqrt{\frac{s_*\log(Tp)}{\underline\lambda^2T}} .
\]

Moreover, the weak coefficients satisfy a rate condition of the form
\[
    \|[\boldsymbol\beta^*]_{\mathcal I_{\rm wk}}\|_1
    =
    \eta_T,
\]
where \(\eta_T\to 0\) as \(T\to\infty\), and may vary across different results; in particular, each theorem may impose additional restrictions on \(\eta_T\) as needed.

When \(\mathcal I_{\rm wk}=\emptyset\), the above condition is interpreted as holding automatically.
\end{assumption}
\begin{remark}
The case \(\mathcal I_*=\emptyset\) can be handled by the same arguments, with the beta-min condition omitted. In that case, all nonzero coefficients belong to the weak component and the preceding weak-tail condition implies that \(\boldsymbol\beta^*\) is asymptotically close to \(\mathbf 0\) in \(\ell_1\) norm. We exclude this case from the main statements only to avoid separate notation for an empty selected support.
\end{remark}
\begin{assumption}[Modified mutual incoherence]
\label{asp: mutual inch}
There exists a constant \(\gamma_0\in(0,1]\) such that
\[
    \left\|
        \boldsymbol\Sigma^*_{\mathcal I_*^c,\mathcal I_*}
        \left(\boldsymbol\Sigma^*_{\mathcal I_*}\right)^{-1}
    \right\|_{\infty}
    \le
    1-\gamma_0,
\]
where
\(\boldsymbol\Sigma_{\mathcal I_*}^*
:=
\boldsymbol\Sigma^*_{\mathcal I_*,\mathcal I_*}\) is the principal submatrix of
\(\boldsymbol\Sigma^*\) indexed by \(\mathcal I_*\times\mathcal I_*\), and
\(\mathcal I_*^c:=[p]\setminus\mathcal I_*\). 
\end{assumption}
Assumption~\ref{asp: min signal} is an effective signal-strength condition rather than a minimum signal-strength condition for exact recovery of the full support. This distinction is important in the MNL setting, where many contextual features may have small but nonzero effects on utilities. The set \(\mathcal I_*\) contains the coordinates whose effects are large enough to be statistically distinguishable at the estimation scale, while \(\mathcal I_{\rm wk}\) contains weak coordinates whose aggregate contribution is asymptotically negligible.
Assumption~\ref{asp: mutual inch} is the usual irrepresentability-type condition, but imposed relative to the effective support \(\mathcal I_*\) rather than the full nonzero support \(\mathcal I_0\). 

Under these additional conditions, we establish in the following corollary of Theorem~\ref{thm:lasso-rates} that, with high probability, a unique solution to \eqref{eq: theta hat} exists and recovers the effective support $\mathcal{I}_*$.
\begin{corollary}[Effective support recovery]
\label{col: unique solution decomp}
Suppose Assumptions~\ref{asp: abs cont of r}, \ref{asp: cov}--\ref{asp: mutual inch} hold. Assume that
\[
   C\nu s\sqrt{\frac{\log(Tp)}{\underline\lambda^2T}}
   \le \tau
   \le
   c\frac{\underline\lambda^2}{C_n\nu^4s^2},
   \qquad
   T\ge
   C\frac{\nu^8s^4\log(Tp)}{\underline\lambda^4}, \qquad \eta_T =  o\left(
        \frac{1}{\nu}\sqrt{\frac{\log(Tp)}{T}}
    \right),
\]
where \(C>0\) is sufficiently large and \(c>0\) is sufficiently small. Choose
\[
    \lambda_{T-1}
    =
    C_\lambda\nu\sqrt{T\log(Tp)}
\]
with \(C_\lambda>0\) sufficiently large. Then, with probability at least
\(1-O(T^{-1})\), the estimator in \eqref{eq: theta hat} admits a unique solution
\(\hat{\boldsymbol\beta}_{T-1}\) that recovers the effective support of
\(\boldsymbol\beta^*\); that is, $\mathcal I=\mathcal I_* $.

Moreover, with probability at least \(1-O(T^{-1})\),
\begin{equation}\label{eq: l1 rate improved}
    \|\hat{\boldsymbol\beta}_{T-1}-\boldsymbol\beta^*\|_1
    \le
    C\nu s_*\sqrt{\frac{\log(Tp)}{\underline\lambda^2T}} .
\end{equation}
This improves \eqref{eq:l1-rate} when \(s_*\ll s\).
\end{corollary}
The proof is given in Section~\ref{sec: proof col unique solution decomp}.
\begin{remark}
The lower bound on \(\tau\) should be read as a sample-size requirement: \(T\) must be large enough for the terminal estimator to improve over the pilot \(\hat\bbeta_0\), so that \(\hat\bbeta_{T-1}\in\cB_1(\hat\bbeta_0,2\tau)\) and hence coincides with the unconstrained \(\ell_1\)-penalized estimator.
\end{remark}

Another consequence of Theorem~\ref{thm:lasso-rates} is that it yields a revenue-loss guarantee for the dynamic policy. This justifies the earlier claim that the proposed procedure does not require sacrificing revenue through a purely exploratory data-collection phase in order to obtain valid inference. More formally, the expected regret of the dynamic policy is defined as
\[
\sum_{t=1}^{T-1}
\EE\Big\{
R(\cS_t^* \mid \bbeta^*,\bv_t,\br_t)
-
R(\cS_t \mid \bbeta^*,\bv_t,\br_t)
\Big\},
\]
which is the cumulative expected revenue loss incurred by offering \(\cS_t\) instead of the oracle assortment \(\cS_t^*\). The following corollary of Theorem~\ref{thm:lasso-rates} provides an upper bound on this regret, whose proof is deferred to Section~\ref{sec: proof col regret bound}.

\begin{corollary}\label{col: Regret bound}
Under the conditions of Theorem~\ref{thm:lasso-rates}, the dynamic data-collection policy in Section~\ref{sec: policy} satisfies
\[
\sum_{t=1}^{T-1}
\EE\Big\{
R(\cS_t^* \mid \bbeta^*,\bv_t,\br_t)
-
R(\cS_t \mid \bbeta^*,\bv_t,\br_t)
\Big\}
\lesssim
\frac{\nu^2\bar\mu s}{\underline\lambda}\sqrt{T\log(Tp)}.
\]
\end{corollary}
\begin{remark}
Corollary~\ref{col: Regret bound} shows that the dynamic policy incurs sublinear regret while collecting data for the terminal inference procedure. Ignoring logarithmic and problem-dependent constants, the leading dependence is \(s\sqrt T\). When \(s\asymp d\), this matches the \(\widetilde O(d\sqrt T)\) regret rate in \cite{Chen2020context} up to logarithmic factors, and is consistent with their information-theoretic lower bound \(\Omega(d\sqrt T/K)\) up to logarithmic and \(K\)-dependent factors. If \(\hat\bbeta_0\) is obtained from an initial online phase, this may induce additional regret; under the required \(\tau\)-accuracy, that contribution is dominated by the regret from the subsequent exploitative stage.
\end{remark}
With these convergence and support-recovery guarantees in place, we now turn to the inferential analysis. We begin with a non-asymptotic error decomposition for the debiased estimator \(\tilde\bbeta^d\).
\begin{corollary}\label{col: error decomp debias lasso}
Under the same conditions as Corollary~\ref{col: unique solution decomp}, with probability at least
\(1-O(T^{-1})\), the following statements hold:
\begin{equation}\label{eq: debiased-error-decomp}
\begin{aligned}
&[\tilde\bbeta^{\,d}]_{\cI_0^{c}}=[\bbeta^{*}]_{\cI_0^{c}}=\mathbf 0,\\
&\|[\tilde\bbeta^{\,d}]_{\cI_{\rm wk}}-[\bbeta^{*}]_{\cI_{\rm wk}}\|_1
=
\|[\boldsymbol\beta^*]_{\mathcal I_{\rm wk}}\|_1
=
\eta_T,\\
&[\tilde\bbeta^{\,d}]_{\cI_*}-[\bbeta^{*}]_{\cI_*}
=
-\frac{1}{T-1}\,(\bSigma_{\cI_*}^{*})^{-1}\,[\nabla_{\bbeta}\ell_{T-1}(\bbeta^{*})]_{\cI_*}
+\Rb .
\end{aligned}
\end{equation}
The remainder term satisfies
\[
\|\Rb\|_2
\lesssim
\frac{\nu^3s_*^{3/2}\log(Tp)}{T\underline\lambda^2}
\left\{
    1+\frac{C_n\nu s+\nu^2s_*}{\underline\lambda}
\right\}
+
\frac{\nu^2\sqrt{s_*}}{\underline\lambda}
\eta_T .
\]
\end{corollary}
The proof is given in Section~\ref{sec: proof col error decomp lasso}. Corollary~\ref{col: error decomp debias lasso} provides the linear expansion that underlies the subsequent inference, with a leading term given by a martingale score on the effective support. We use this expansion to calibrate the directional and magnitude components of the estimation uncertainty, and then establish the validity and power of the proposed test in the next section.
\subsection{Testing Validity and Power}

Since the terminal inference is conditional on the realized contextual information
\((\bv_T,\br_T)\), we first define the null and alternative parameter spaces for
\(\bbeta^*\) relative to this terminal context. For \((\bv_T,\br_T)\), define
\begin{equation}\label{eq: bTheta_0}
    \cM_0(\bv_T,\br_T)
    :=
    \{\bbeta^* \in \RR^p: \cS_T(\bbeta^*) \cap \bcS_0 \ne \emptyset\},
\end{equation}
and
\begin{equation}\label{eq: bTheta_1}
    \cM_1(\bv_T,\br_T)
    :=
    \{\bbeta^* \in \RR^p: \cS_T(\bbeta^*) \cap \bcS_0 = \emptyset\}.
\end{equation}

The proposed procedure does not calibrate uncertainty uniformly over all candidate assortments in \(\bcS^K\). Such a uniform calibration is typically conservative, because assortments that are far from maximizing the null or alternative revenue surfaces do not affect the sign of the max-difference statistic. Instead, the validity and power analysis can be localized to assortments that may become maximizers after perturbing the terminal revenue surface within the statistical error scale. To formalize this localization, define
\begin{align}
    \bar\bcS_{0}(\bv_T,\br_T)
    &=
    \left\{
    \cS \in \bcS_0:
    R(\cS \mid \bbeta^*,\bv_T,\br_T)
    \ge
    \max_{\cS' \in \bcS_0}
    R(\cS' \mid \bbeta^*,\bv_T,\br_T)
    -
    18\nu s_* \bar\mu
    \sqrt{\frac{\log T}{T\underline\lambda}}
    \right\}, \label{eq: S_0 max cand} \\
    \bar\bcS_{1}(\bv_T,\br_T)
    &=
    \left\{
    \cS \notin \bcS_0:
    R(\cS \mid \bbeta^*,\bv_T,\br_T)
    \ge
    \max_{\cS' \notin \bcS_0}
    R(\cS' \mid \bbeta^*,\bv_T,\br_T)
    -
    18\nu s_* \bar\mu
    \sqrt{\frac{\log T}{T\underline\lambda}}
    \right\}. \label{eq: S_1 max cand}
\end{align}
The sets \(\bar\bcS_0(\bv_T,\br_T)\) and \(\bar\bcS_1(\bv_T,\br_T)\) contain the assortments that can act as potential maximizers of
\(R(\cS\mid\tilde\bbeta^d,\bv_T,\br_T)\) over \(\bcS_0\) and \(\bcS^K\setminus\bcS_0\), respectively, with high probability. Hence, only these localized boundary candidates determine the leading uncertainty of the max-difference statistic.

The following theorem establishes the validity of the perturbation-based p-value under the null.

\begin{theorem}[Validity of the perturbation p-value]\label{thm: valid p}
For each \(\cS \in \bcS^K\), define $g_{\cS} :=    \bigl[\nabla_{\bbeta} R(\cS \mid \bbeta^*, \bv_T, \br_T)\bigr]_{\cI_*}$, and let
\[
    \sigma_{\bv_T,\br_T}
    :=
    \max_{\cS \in \bar\bcS_0(\bv_T,\br_T)\cup \bar\bcS_1(\bv_T,\br_T)}
    \|g_{\cS}\|_2.
\]
Assume that \(\sigma_{\bv_T,\br_T} > 0\) almost surely. Let the conditions of Corollary~\ref{col: unique solution decomp} hold. Fix any \(\alpha \in (0,1)\) and \(\epsilon \in (0,1)\).

Suppose that, for sufficiently large constants \(C>0\) and sufficiently small \(c>0\),
\[
T
\ge
C\frac{C_n^2 \nu^6 s^2 s_*^2 \log^2 (Tp)}{\underline\lambda^4}
\left[
\log (Tp)
\vee
\frac{\bar\mu^2 \nu^{4}s_*}
{\sigma_{\bv_T,\br_T}^2 \underline\lambda}
\right],\quad \text{and}\quad
\eta_T
\le
c\,
\frac{
    \sigma_{\bv_T,\br_T}\sqrt{\underline\lambda}
}{
    \bar\mu \nu^3 \sqrt{s_* T}
}.
\]

Let \(m \in \mathbb{N}\) and \(\kappa > 0\) satisfy
\[
m \ge
\log\!\left(\frac{2}{\alpha}\right)
\sqrt{\frac{8\hat{s}}{\pi}}
\left(\frac{\pi}{2\epsilon}\right)^{\hat{s}-1},
\qquad
\kappa
\asymp
\sigma_{\bv_T,\br_T}
\sqrt{\frac{s_*}{T\underline\lambda}} .
\]

Then, with probability at least \(1 - O(T^{-1})\) over \((\bv_T,\br_T)\),
\[
\sup_{\bbeta^* \in \cM_0(\bv_T,\br_T)}
\PP_{\bbeta^*}\!\left(p_m \le \alpha \mid \bv_T,\br_T\right)
\le \alpha + o(1).
\]
\end{theorem}

The proof is given in Section~\ref{sec: proof thm type I}.

\begin{remark}\label{rmk:kappa-choice}
Under Assumption~\ref{asp: abs cont of r}, \(\sigma_{\bv_T,\br_T}>0\) almost surely provided that there exists an \(\cS \in \bar\bcS_0(\bv_T,\br_T)\cup \bar\bcS_1(\bv_T,\br_T)\)
 with \([\bv_{Tj}]_{\cI_*}\neq 0\) for some \(j\in\cS\).
The threshold \(\kappa\) absorbs higher-order errors from the revenue expansion and the remainder in the decomposition of \(\tilde\bbeta^d\). Its theoretically sufficient order depends on the local scale \(\sigma_{\bv_T,\br_T}\) and the minimal eigenvalue of \(\bSigma^*\). When \(\bSigma^*\) is well conditioned and the localized candidate set is moderate, \(\sigma_{\bv_T,\br_T}^2\) is typically of order \(\underline\lambda\). In this regime, Theorem~\ref{thm: valid p} requires
\(\kappa \asymp \sqrt{\hat{s}/T}\) for fixed \(\epsilon\), which is directly
implementable once the selected support is obtained. Thus, \(\kappa\) should
be viewed as a conservative theoretical margin rather than a sharp calibration
parameter. In the simulation studies, we choose
\[
    \kappa
    =
    C_\kappa
    \sqrt{\frac{\hat{s}}{T}}\,
    \epsilon,
\]
and find that the performance is robust even for very small values of
\(C_\kappa\).

The number of sampled directions \(m\) depends on the selected support size \(\hat{s}=|\cI|\). Under the stated conditions, \(\hat{s}=s_*\) with high probability. Since \(s_*\) is small relative to the ambient dimension in sparse high-dimensional MNL models, the required number of directions remains tractable when \(s_*\) is bounded or grows slowly.
\end{remark}
 Then we have the following result on the asymptotic validity and power of the proposed test. The proof is in Section~\ref{sec: proof thm valid powerful test}.
\begin{theorem}\label{thm: valid powerful test}
Under the conditions of Theorem~\ref{thm: valid p}, for each fixed
\(\alpha\in(0,1)\),
\begin{equation}\label{eq: type I asymp}
\sup_{\bbeta^*\in\cM_0(\bv_T,\br_T)}
\PP_{\bbeta^*}\big(\text{reject }H_0\mid\bv_T,\br_T\big)
\le
\alpha+o(1)
\end{equation}
with probability at least \(1-O(T^{-1})\) over the randomness in
\((\bv_T,\br_T)\).

Moreover, with probability at least \(1-O(T^{-1})\) over the randomness in
\((\bv_T,\br_T)\), for any
\(\bbeta^*\in\cM_1(\bv_T,\br_T)\) satisfying
\begin{equation}\label{eq: signal strength}
\sigma_{\bv_T,\br_T}^{-1}
\left(
\max_{\cS\in\bcS_0} R_{T,\cS}^*
-
\max_{\cS\notin\bcS_0} R_{T,\cS}^*
\right)
\ll
-
\sqrt{\frac{s_*}{T\underline\lambda}},
\end{equation}
we have
\begin{equation}\label{eq: power}
\PP_{\bbeta^*}\big(\text{reject }H_0\mid\bv_T,\br_T\big)
\to1 .
\end{equation}
\end{theorem}
\begin{remark}\label{rmk:localized-power}
The power condition in Theorem~\ref{thm: valid powerful test} is localized through
\[
    \sigma_{\bv_T,\br_T}
    :=
    \max_{\cS\in\bar\bcS_0(\bv_T,\br_T)\cup\bar\bcS_1(\bv_T,\br_T)}
    \|g_{\cS}\|_2 ,
\]
which is computed only over near-optimal assortments close to the null--alternative boundary. The resulting signal scale is
\[
    \sigma_{\bv_T,\br_T}
    \sqrt{\frac{s_*}{T\underline\lambda}} .
\]
By contrast, a uniform error bound approach would typically calibrate the leading revenue error over all \(\cS\in\bcS^K\), leading to a scale of the form
\[
    \sigma_{\rm glob}
    \sqrt{\frac{\log|\bcS^K|}{T\underline\lambda}},
    \qquad
    \sigma_{\rm glob}:=\max_{\cS\in\bcS^K}\|g_{\cS}\|_2 .
\]
Thus the power gain is largest when
\[
    \sigma_{\bv_T,\br_T}\sqrt{s_*}
    \ll
    \sigma_{\rm glob}\sqrt{\log|\bcS^K|},
\]
which can occur either because the effective support dimension \(s_*\) is small relative to the combinatorial complexity \(\log|\bcS^K|\), or because the localized gradient scale \(\sigma_{\bv_T,\br_T}\) is much smaller than the global scale \(\sigma_{\rm glob}\). For the full \(K\)-subset class, \(\log|\bcS^K|\asymp K\log(en/K)\), so this comparison remains relevant when \(K\) grows.
\end{remark}


\section{Numerical Results}\label{sec:numerical}

We compare the proposed perturbation test with a uniform error bound (UEB) baseline for the three testing problems in Examples~\ref{exm: subset test}--\ref{exm: feature test} of Section~\ref{sec:setup}. The reported size and power are computed from a held-out set of Monte Carlo replications that is not used for tuning-parameter selection; additional diagnostics are reported in Appendix~\ref{app:numerical}. Throughout this section, the nominal level is \(\alpha=0.05\). We use \(n=20\) products, ambient dimension \(p=500\), sparsity level \(s=s_*\in\{3,4,5\}\), and cardinality constraint \(K=3\). For simplicity of the numerical setup, we impose \(\mathcal I_0=\mathcal I_*\).
\subsection{Simulation setup}\label{sec:sim-setup}
 We first describe the common data-generating mechanism and optimization class used across the three examples. The feasible class is the full collection of \(K\)-subsets,
\[
\bcS^K=\{\cS\subseteq[n]:|\cS|=K\},
\]
so that \(|\bcS^K|=\binom{20}{3}=1140\) candidate assortments. The active support of \(\bbeta^*\) is set to \(\cI_0=[s]\), with alternating-sign coefficients of decreasing magnitude,
\[
    [\bbeta^*]_{\cI_0}
    =
    1.5 \cdot\,[(1,-0.8,0.6,-0.5,0.4)^{\top}]_{1:s}.
\]
At each round \(t\), the contextual feature vectors \(\bv_{tj}\) are generated independently from \(\cN(0,\Ib_p/3)\), with entries truncated to \([-1,1]\). Item revenues are generated as \(r_{tj}\overset{\rm i.i.d.}{\sim}\cN(6.5,1)\), truncated below at \(0.01\). The penalty parameter for \(\hat\bbeta_{t-1}\) in~\eqref{eq: theta hat} and the thresholding parameter in the perturbation test are chosen as
\[
    \lambda_{t-1}
    =
    C_\lambda\,\{\log(Tp)+\sqrt{t\log(Tp)}\}, \quad
   \kappa
    =
    C_\kappa
    \sqrt{\frac{\hat{s}}{T}}\,
    \epsilon,
\]
where \(C_\lambda\) and \(C_\kappa\) are tuning constants and \(\hat{s}=|\cI|\) is the
selected support size.
Here the threshold is written in the directly implementable form: the constant
\(C_\kappa\) absorbs the local-scale factor \(\sigma_{\bv_T,\br_T}/\sqrt{\underline\lambda}\)
that enters the theoretical order of \(\kappa\) in Section~\ref{sec:theory},
because \(\sigma_{\bv_T,\br_T}\) and \(\underline\lambda\) are not available in
practice and are therefore calibrated jointly through \(C_\kappa\) rather than estimated separately.

The four tuning parameters \((C_\lambda,\epsilon,C_\kappa,\delta_m)\) are selected using Monte Carlo replications disjoint from those used for the final evaluation. Specifically, we use a three-stage procedure consisting of an initial calibration grid with \(N=50\) replications, a confirmation step with \(N=200\) replications applied to the leading candidate configurations, and a final evaluation set with \(N=500\) replications reserved exclusively for reporting. The selected configuration is
\begin{equation}\label{eq:selected-params}
(C_\lambda, \epsilon, C_\kappa, \delta_m)
\;=\;
(0.40,\ 0.50,\ 10^{-4},\ 0.002).
\end{equation}
The number of random directions \(m\) is determined from \(\delta_m\) through the definition in~\eqref{eq: delta_m}, and the resulting value is capped at \(50{,}000\) in the numerical implementation.
We use the noninformative pilot \(\hat\bbeta_0=\boldsymbol{0}\). The estimator \(\hat\bbeta_{t-1}\) is updated at every time \(t\), and the selected support \(\cI\) used in the debiasing and perturbation steps is obtained from the terminal estimator.

We use the UEB procedure as the benchmark. This baseline is adapted from the maximal-perturbation approach of \cite{shen2023combinfassort} and calibrates the Gaussian analogue of the maximal revenue error over the full candidate assortment class. For comparability, the baseline uses the same adaptive history, terminal context \((\bv_T,\br_T)\), and selected support \(\cI\) as our proposed method. It computes the plug-in revenues at the unregularized maximum likelihood estimator refitted on \(\cI\), and approximates the maximal perturbation statistic
\begin{equation}\label{eq:ueb-stat}
W_b
\;=\;
\max_{\cS \in \bcS^K}
\bigl|\hat g_\cS^{\!\top} \widehat\bTheta^{1/2} Z_b\bigr|,
\qquad Z_b \overset{\rm i.i.d.}{\sim} \cN(0, \Ib_{\hat{s}}),
\end{equation}
where \(\hat g_{\cS}\) and \(\widehat\bTheta\) are defined in~\eqref{eq: rev and grad} and~\eqref{eq: est fisher mat}, respectively. In all reported evaluations, this Gaussian maximum is simulated using \(B=1000\) draws per replication.
Let \(\hat\bbeta_{\cI}^{\rm mle}\) denote the unregularized maximum likelihood
estimator refitted on the selected support \(\cI\), with coordinates outside
\(\cI\) set to zero, and define
\[
\widehat R_{T,\cS}^{\rm U}
:=
R(\cS\mid \hat\bbeta_{\cI}^{\rm mle},\bv_T,\br_T),
\]
the plug-in revenue used by the UEB confidence set, in contrast to the debiased
plug-in revenue \(\hat R_{T,\cS}\) in~\eqref{eq: rev and grad} used by the
proposed test.
Let \(C_W\) denote the empirical \((1-\alpha)\)-quantile of \(\{W_b\}_{b=1}^B\). The UEB baseline constructs the uniform revenue confidence set
\begin{equation}\label{eq:ueb-band}
\cC_\alpha =
\left\{
\cS \in \bcS^K :
\widehat R_{T,\cS}^{\rm U}
\ge
\max_{\cS' \in \bcS^K}\widehat R_{T,\cS'}^{\rm U}
-\frac{2C_W}{\sqrt{T-1}}
\right\}.
\end{equation}
It rejects \(H_0:\cS_T^*\cap\bcS_0\neq\emptyset\) if and only if \(\cC_\alpha\cap\bcS_0=\emptyset\). For each example, the UEB rejection rule is evaluated on the same structural target \(\bcS_0\) and the same terminal contexts as the proposed test.

The empirical rejection rate under \(H_0\) depends on the revenue gap
\[
\Delta^*
:=
\max_{\cS \in \bcS_0} R^*_{T,\cS}
-
\max_{\cS \notin \bcS_0} R^*_{T,\cS},
\]
which is nonnegative under the null hypothesis in~\eqref{eq: hypo test}. We therefore evaluate size at the least favorable boundary \(\Delta^*=0\). Specifically, we draw a terminal context satisfying \(\cS_T^*(\bbeta^*)\cap\bcS_0\neq\emptyset\) and adjust one terminal revenue coordinate by bisection until \(\Delta^*=0\). Under \(H_1\), we use unmodified terminal contexts satisfying \(\cS_T^*(\bbeta^*)\cap\bcS_0=\emptyset\). Each cell uses \(N=500\) independent replications over horizons \(T\in\{200,300,400,500,700,1000,1500,2000\}\). 

We report empirical size and power for Examples~\ref{exm: subset test}--\ref{exm: feature test} of Section~\ref{sec:setup}. In the size panels, the vertical axis is truncated at \(0.10\) to focus on the nominal region; values exceeding this range occur only at the smallest horizons. Across the \(s\in\{3,4,5\}\) configurations, the finite-sample size distortion at very small horizons decreases rapidly as \(T\) increases, consistent with the asymptotic nature of the calibration.
The UEB baseline is substantially more conservative across configurations, making the power comparison especially pronounced. We summarize the results for each example below.

\subsection{Example~\ref{exm: subset test}: Product inclusion test}\label{sec:results-ex1}

We set \(\cA=\{1,2\}\) and test whether both products belong to a revenue-maximizing assortment. Figure~\ref{fig:ex1} compares the empirical size and power of the proposed test with those of the UEB baseline. The proposed test is oversized at the smallest horizons, but its empirical size decreases rapidly and falls below the nominal level as \(T\) grows.

The power advantage over the UEB baseline is substantial. Across \(s\in\{3,4,5\}\), the proposed test has high power already at small to moderate horizons and approaches one as \(T\) increases. In contrast, the UEB baseline remains much less powerful throughout the horizon range while being highly conservative under the null. This contrast is consistent with the theoretical motivation: the UEB procedure controls revenue error uniformly over the full candidate class, whereas the proposed perturbation test targets the max-difference boundary relevant to the subset-inclusion hypothesis.

\begin{figure}[htbp]
\centering
\includegraphics[width=\textwidth]{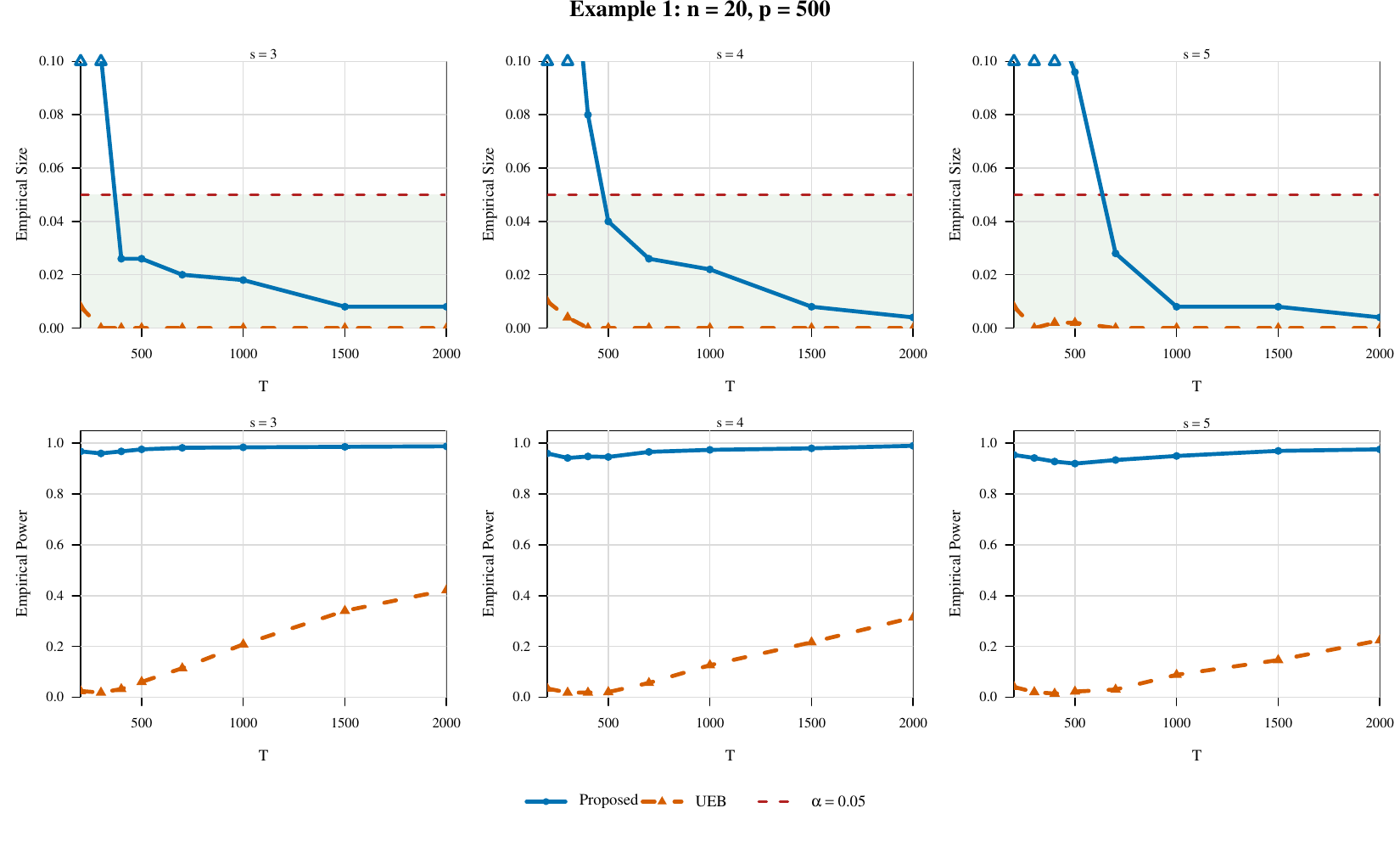}
\caption{Example~\ref{exm: subset test}. Top row: empirical size at the least-favorable boundary \(\Delta^*=0\). Bottom row: empirical power. The proposed test falls below the nominal level as \(T\) grows and attains substantially higher power than the UEB baseline across all sparsity levels. The size panels are truncated at \(0.10\) to emphasize the nominal region; values exceeding this range occur only at the smallest horizons.}
\label{fig:ex1}
\end{figure}

\subsection{Example~\ref{exm: category prop test}: Category proportion test}\label{sec:results-ex2}

We partition products into a focal category
\(\cA=\{1,\ldots,\lceil n/2\rceil\}\) and its complement, and test
whether the optimal assortment contains a strict majority drawn from
\(\cA\). With \(K=3\), the null requires at least two of the three selected
products to lie in \(\cA\), yielding \(|\bcS_0|=570\) out of \(|\bcS^K|=1140\)
when \(n=20\). This balanced partition produces small revenue gaps \(\Delta^*\) between the best null-feasible and alternative-feasible assortments, making this the statistically hardest of the three examples.

Figure~\ref{fig:ex2} shows that the proposed test again controls size after the initial small-horizon regime, with empirical rejection rates falling below the nominal level as \(T\) increases. The power comparison is especially pronounced in this example. Across all sparsity levels \(s\in\{3,4,5\}\), the proposed test gains power steadily with \(T\), reaching high power by \(T=2000\). In contrast, the UEB baseline is nearly powerless throughout the horizon range; its power curves are visually close to zero in all panels. This example is therefore the most informative comparison between the directional perturbation construction and uniform error control: when the structural null divides the candidate class into two balanced parts and the relevant revenue gap is local, the proposed test delivers a large power improvement while remaining empirically conservative under \(H_0\).

\begin{figure}[htbp]
\centering
\includegraphics[width=\textwidth]{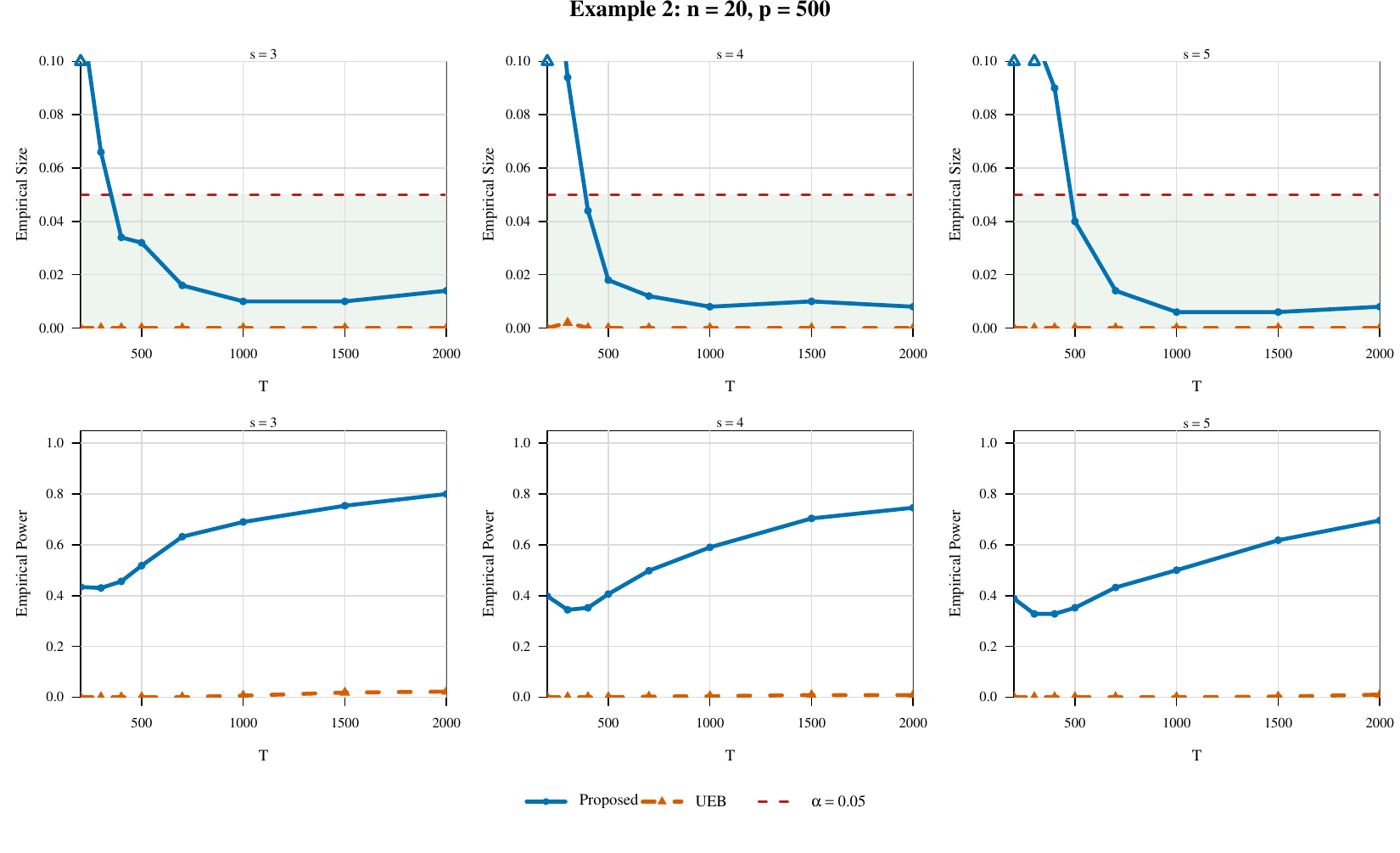}
\caption{Example~\ref{exm: category prop test}. Layout and line styles are as in Figure~\ref{fig:ex1}. The proposed test falls below the nominal size as \(T\) grows and maintains substantial power in this near-boundary setting. The UEB baseline is highly conservative and has power close to zero across all sparsity levels.}
\label{fig:ex2}
\end{figure}

\subsection{Example~\ref{exm: feature test}: Feature test}\label{sec:results-ex3}

Example~\ref{exm: feature test} uses a fixed operational feature screen. We pre-specify the feature set
\[
\cV
=
\left\{
x\in\RR^{p+1}:
\max_{k\in\{6,7\}} |x_k|\le 0.65
\right\},
\]
and keep it fixed across all \(s\in\{3,4,5\}\) configurations. The corresponding null family is
\[
\bcS_0
=
\left\{
\cS\in\bcS^K:
\max_{k\in\{6,7\}} |[v_{Tj}]_k|\le 0.65,\ \forall j\in\cS
\right\}.
\]
Thus, the test asks whether there exists a terminal revenue-maximizing assortment whose items all have moderate values on this fixed group of operational attributes. Unlike Examples~\ref{exm: subset test}--\ref{exm: category prop test}, the null class \(\bcS_0\) is context-dependent and varies across replications. 

Figure~\ref{fig:ex3} reports the size and power diagnostics. Both the proposed procedure and the UEB baseline are evaluated on the same terminal contexts and the same fixed-feature null family. Because \(\bcS_0\) is context-dependent and varies across replications, this example provides the most demanding assessment of calibration among the three, as it probes a null class that is itself random in the terminal context. The proposed test shows some size distortion at the smallest horizons, but this distortion decreases with \(T\) and its calibration improves at larger horizons. Across \(s\in\{3,4,5\}\), its power remains substantially higher than that of the UEB baseline, which is again more conservative under the null and has noticeably lower power, especially at moderate horizons.
\begin{figure}[htbp]
\centering
\includegraphics[width=\textwidth]{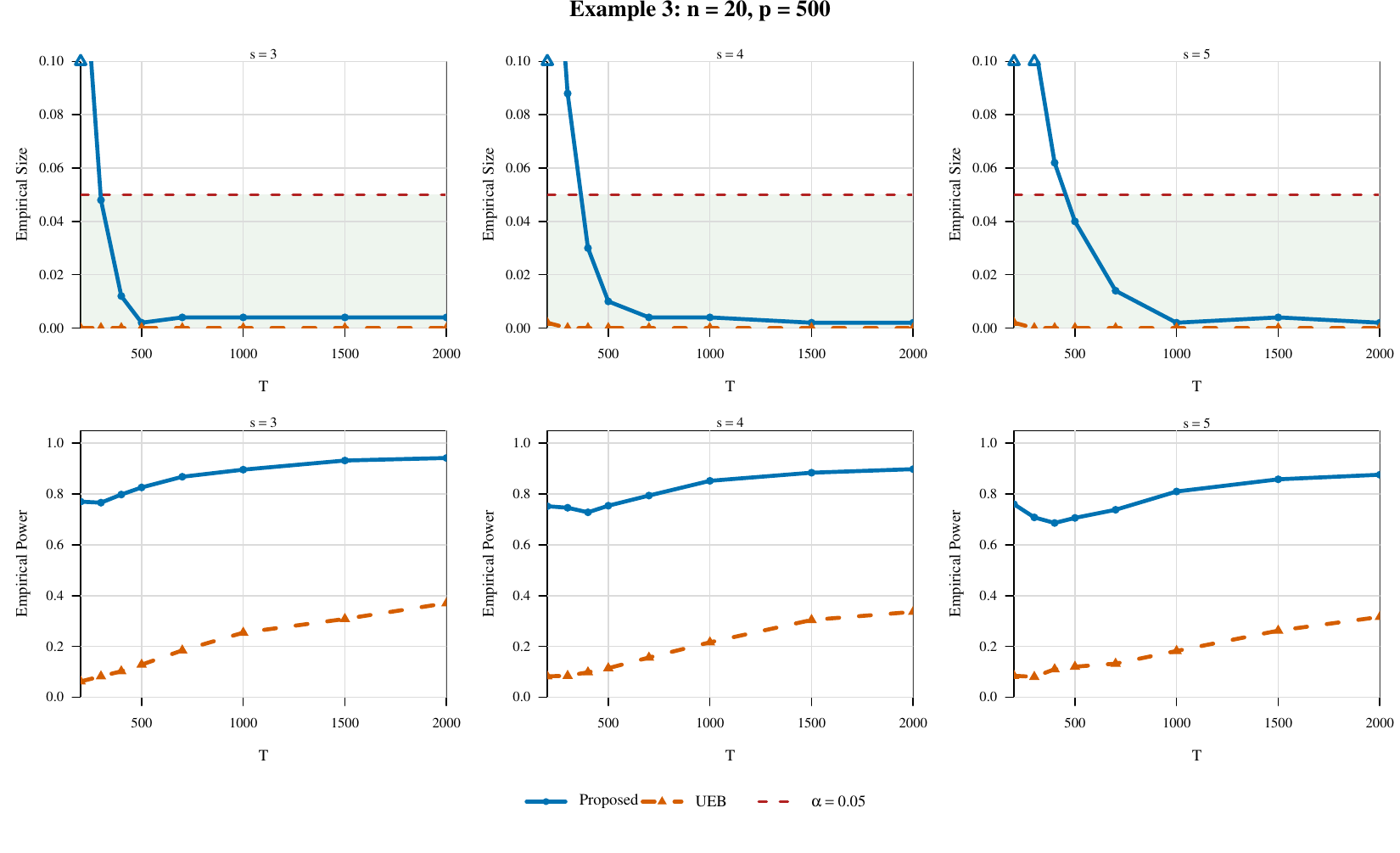}
\caption{Example~\ref{exm: feature test}. Layout and conventions are as in Figure~\ref{fig:ex1}. Here the null class is context-dependent and varies across replications. As \(T\) grows, the size distortion of the proposed test decreases and its power is substantially higher than that of the UEB baseline across all sparsity levels. The UEB baseline remains conservative under the null and has noticeably lower power, especially at moderate horizons.}
\label{fig:ex3}
\end{figure}

\subsection{Cumulative regret of the adaptive policy}
\label{sec:results-regret}

We also examine the regret performance of the adaptive data-collection policy. This experiment provides a finite-sample diagnostic for the sublinear regret guarantee in Corollary~\ref{col: Regret bound}. We use the same simulation setting as in the main experiments, with \(n=20\), \(p=500\), \(K=3\), and sparsity levels \(s\in\{3,4,5\}\). The tuning parameters are fixed at the values selected by the three-stage protocol, and the results are based on \(N=100\) independent replications.

Figure~\ref{fig:regret} reports the median cumulative regret as a function of the horizon \(T\). For all three sparsity levels, cumulative regret increases slowly with \(T\), and the curves shift upward as \(s\) increases. This pattern is consistent with the sparsity dependence and the sublinear regret growth predicted by Corollary~\ref{col: Regret bound}.

\begin{figure}[!htbp]
\centering
\includegraphics[width=0.75\textwidth]{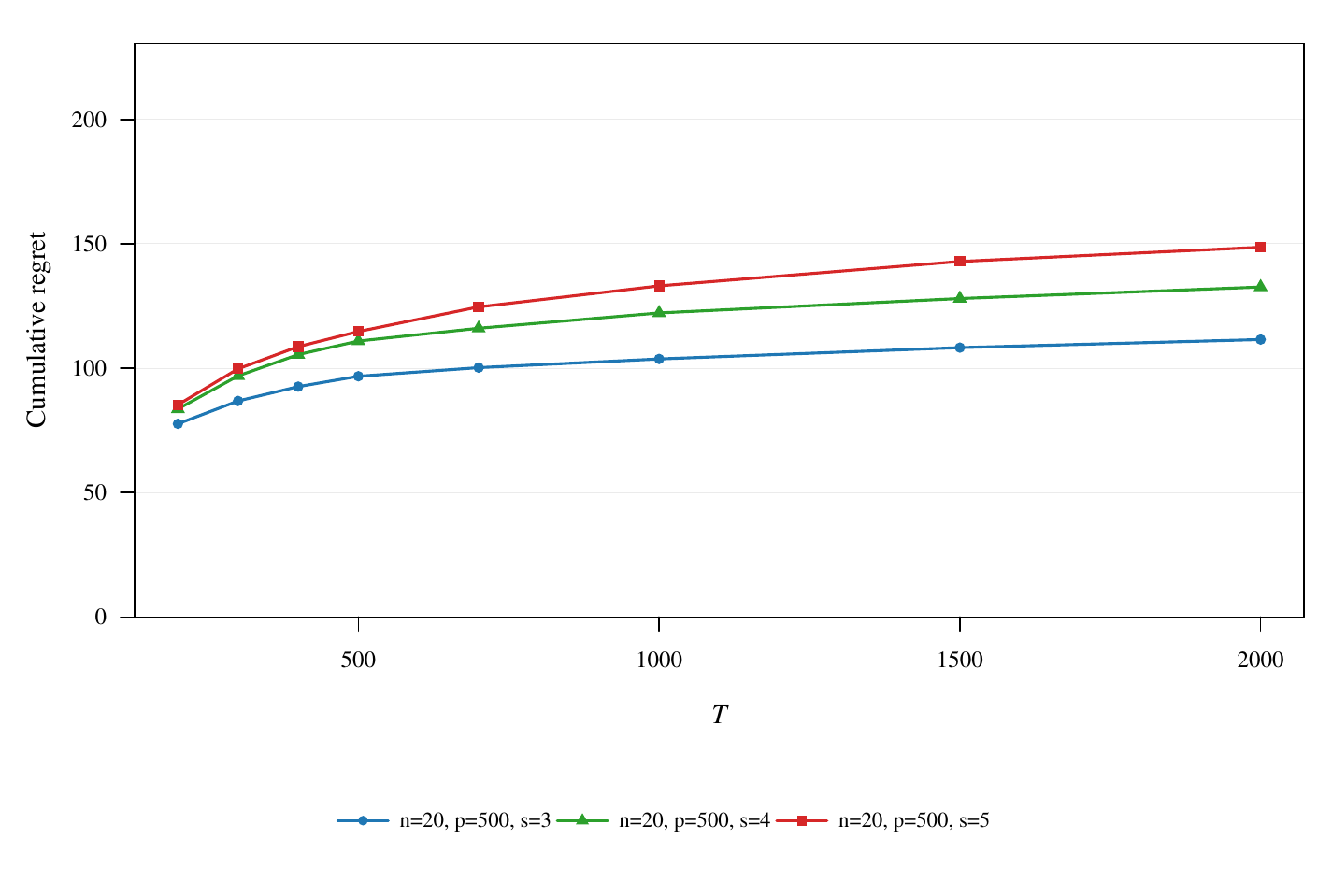}
\caption{Cumulative regret of the greedy adaptive policy with \(n=20\), \(p=500\), and \(s\in\{3,4,5\}\), based on \(N=100\) independent replications. The curves report the median cumulative regret as a function of \(T\), showing slow growth over the horizon.}
\label{fig:regret}
\end{figure}

To assess the predicted scaling more directly, Figure~\ref{fig:regret-normalized} reports the median normalized regret \(\mathrm{Regret}(T)/\sqrt{T\log(Tp)}\) computed from the same replications. The normalized curves remain bounded and slowly varying across the range of \(T\), with the same ordering by sparsity level. This behavior supports the \(\sqrt{T\log(Tp)}\)-type scaling in Corollary~\ref{col: Regret bound} and provides additional evidence that the adaptive policy operates in a sublinear-regret regime.

\begin{figure}[!htbp]
\centering
\includegraphics[width=0.75\textwidth]{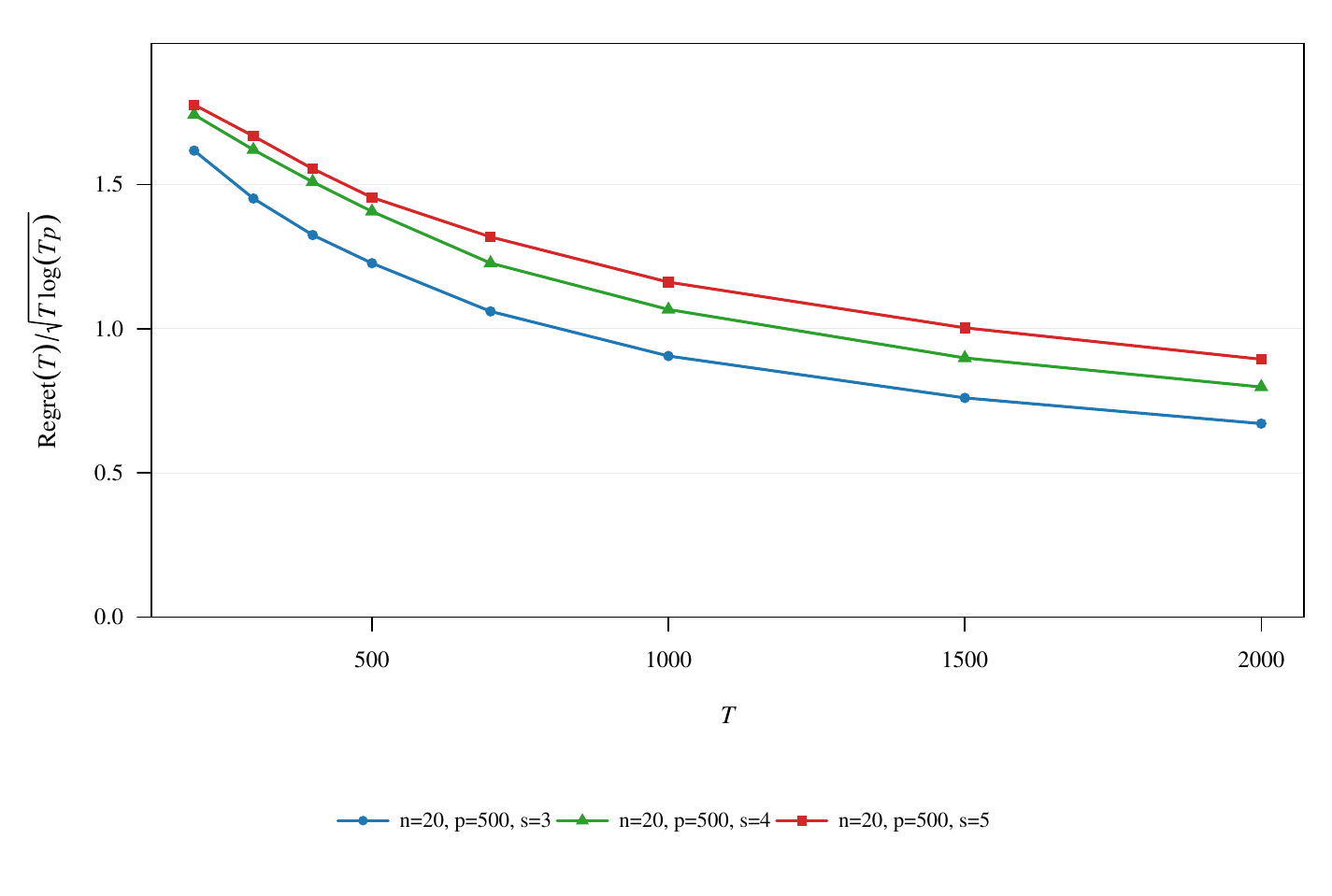}
\caption{Normalized cumulative regret of the greedy adaptive policy with \(n=20\), \(p=500\), and \(s\in\{3,4,5\}\), based on \(N=100\) independent replications. The curves report the median normalized regret \(\mathrm{Regret}(T)/\sqrt{T\log(Tp)}\) as a function of \(T\). Their bounded and slowly varying behavior supports the expected \(\sqrt{T\log(Tp)}\) scaling and the sublinear-regret guarantee in Corollary~\ref{col: Regret bound}.}
\label{fig:regret-normalized}
\end{figure}

\section{Conclusion}\label{sec: conclusion}

This paper develops a perturbation-based inferential framework for structural properties of post-learning combinatorial optimizers. In the high-dimensional contextual MNL assortment model, the inferential target is not the latent utility parameter, but whether the terminal oracle optimizer satisfies a prescribed discrete property. We formulate this problem as inference on the sign of a nonsmooth max-difference revenue functional and construct a p-value through the minimal radius of localized directional perturbations. The resulting procedure separates directional and magnitude uncertainty and avoids the conservativeness of uniform error bound methods that calibrate over the full candidate class.

On the theoretical side, we establish uniform estimation rates, effective support recovery, a debiased expansion on the selected support, and martingale Gaussian coupling for the adaptive score process. A novel use of anti-concentration for differences of Gaussian maxima controls the selection-induced Hessian variation under adaptive assortment selection. These tools yield asymptotic validity and power of the proposed p-value under a localized signal condition. The regret bound further shows that the data-collection policy supports terminal inference without relying on a separate purely exploratory phase.

The numerical experiments support the theory. Across three structurally distinct testing problems, the proposed test controls size at large horizons and delivers substantially higher power than a UEB benchmark, especially in near-boundary regimes where uniform calibration is overly conservative. These results suggest that localized perturbation of max-difference functionals provides a useful route to valid and powerful inference for irregular post-learning decisions.

\clearpage
{
\appendix
\numberwithin{equation}{section}
\numberwithin{figure}{section}
\numberwithin{table}{section}
\renewcommand{\theHsection}{supp.\arabic{section}}
\renewcommand{\theHsubsection}{supp.\arabic{section}.\arabic{subsection}}
\renewcommand{\theHequation}{supp.\arabic{section}.\arabic{equation}}
\renewcommand{\theHfigure}{supp.\arabic{section}.\arabic{figure}}
\renewcommand{\theHtable}{supp.\arabic{section}.\arabic{table}}
\renewcommand{\theHtheorem}{supp.\arabic{section}.\arabic{theorem}}
\renewcommand{\theHassumption}{supp.\arabic{section}.\arabic{assumption}}
\renewcommand{\theHdefinition}{supp.\arabic{section}.\arabic{definition}}
\renewcommand{\theHremark}{supp.\arabic{section}.\arabic{remark}}
\renewcommand{\thesection}{S.\arabic{section}}
\renewcommand{\theequation}{S.\arabic{section}.\arabic{equation}}
\renewcommand{\thefigure}{S.\arabic{section}.\arabic{figure}}
\renewcommand{\thetable}{S.\arabic{section}.\arabic{table}}
\renewcommand{\thetheorem}{S.\arabic{section}.\arabic{theorem}}
\renewcommand{\theassumption}{S.\arabic{section}.\arabic{assumption}}
\renewcommand{\thedefinition}{S.\arabic{section}.\arabic{definition}}
\renewcommand{\theremark}{S.\arabic{section}.\arabic{remark}}
\setcounter{section}{0}
\setcounter{equation}{0}
\setcounter{figure}{0}
\setcounter{table}{0}

\section{Additional Numerical Results}\label{app:numerical}

\subsection{Revenue-gap distribution under the alternative}
\label{app:h1-gap}
The power comparison in Section~\ref{sec:numerical} is most informative when the alternative is close to the max-difference boundary. To quantify this difficulty, we examine the distribution of the population revenue gap \(|\Delta^*|\) under the unadjusted \(H_1\) contexts. Throughout this subsection we fix \(n=20\) and \(p=500\), and report the gap distribution for the sparsity levels \(s = s_* \in\{3,4,5\}\) across Examples~\ref{exm: subset test}--\ref{exm: feature test}.

For each simulation cell, Figure~\ref{fig:app-h1-gap} reports the \(5\%\), \(25\%\), \(50\%\), \(75\%\), and \(95\%\) quantiles of \(|\Delta^*|\). The category-proportion test in Example~\ref{exm: category prop test} has the smallest gaps among the three examples: its quantiles are concentrated near the small-gap reference scale marked in Figure~\ref{fig:app-h1-gap}, and its lower quantiles often lie substantially closer to zero. This confirms that Example~\ref{exm: category prop test} is the near-boundary regime in which uniform calibration over the full class \(\bcS^K\) is especially conservative. The proposed perturbation test retains power in this regime because it localizes the uncertainty to the null--alternative revenue boundary rather than controlling the maximal revenue error over all candidate assortments.

By contrast, the product-inclusion and fixed-feature tests have larger typical gaps. The fixed-feature alternative in Example~\ref{exm: feature test} therefore lies at a moderate separation scale, whereas the category-proportion alternative in Example~\ref{exm: category prop test} represents the most challenging small-gap setting.

\begin{figure}[htbp]
\centering
\includegraphics[width=\textwidth]{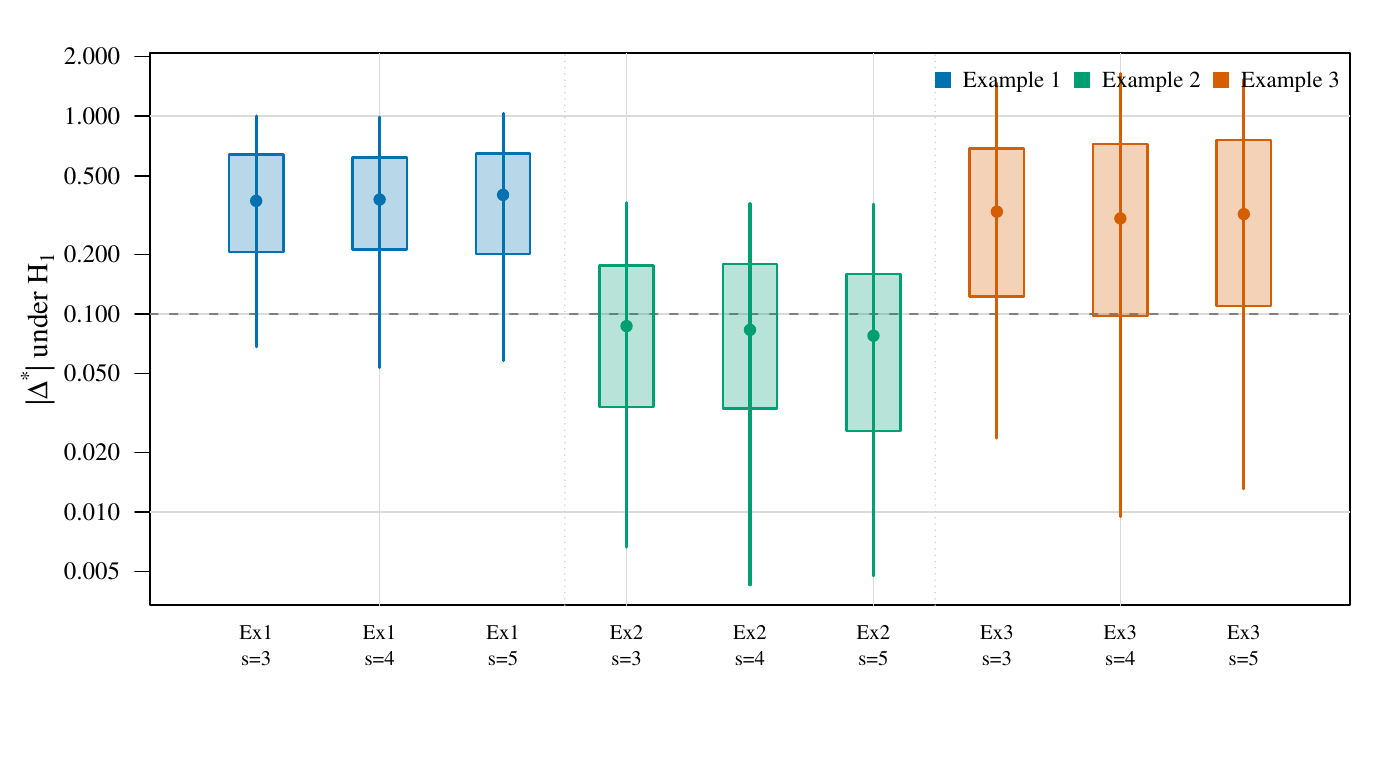}
\caption{Per-cell distribution of \(|\Delta^*|\) on \(H_1\)
contexts at \(n=20\), \(p=500\), for \(s\in\{3,4,5\}\) across the three
examples. Each box summarizes the
\(5\)/\(25\)/\(50\)/\(75\)/\(95\)\% quantiles for one configuration:
the box spans the interquartile range (\(25\)th--\(75\)th percentiles),
the filled dot marks the median, and the vertical whiskers extend to the
\(5\)th and \(95\)th percentiles.
The dashed gray line at \(0.10\) marks a small-gap reference scale near the
null--alternative boundary; the vertical axis is on the log scale.}
\label{fig:app-h1-gap}
\end{figure}

\subsection{Empirical estimation rate}
\label{app:estimation-rate}

As a finite-sample rate diagnostic for Theorem~\ref{thm:lasso-rates},
we record the terminal estimation error
\(\|\hat\bbeta_{T-1}-\bbeta^*\|_2\) at \(n=20\) and \(p=500\) for the
sparsity levels \(s\in\{3,4,5\}\). The experiment uses
\(N=100\) independent replications and the same online estimation
pipeline as the main simulation study. Figure~\ref{fig:app-estimation-rate}
plots the median error against \(T\) on a log-log scale, with
interquartile bands. The log-log reference line has slope \(-1/2\),
matching the theoretical rate established in Theorem~\ref{thm:lasso-rates}.
This diagnostic supports the
finite-sample behavior of the estimator-rate component used in the
inference procedure.

\begin{figure}[t]
\centering
\includegraphics[width=0.75\textwidth]{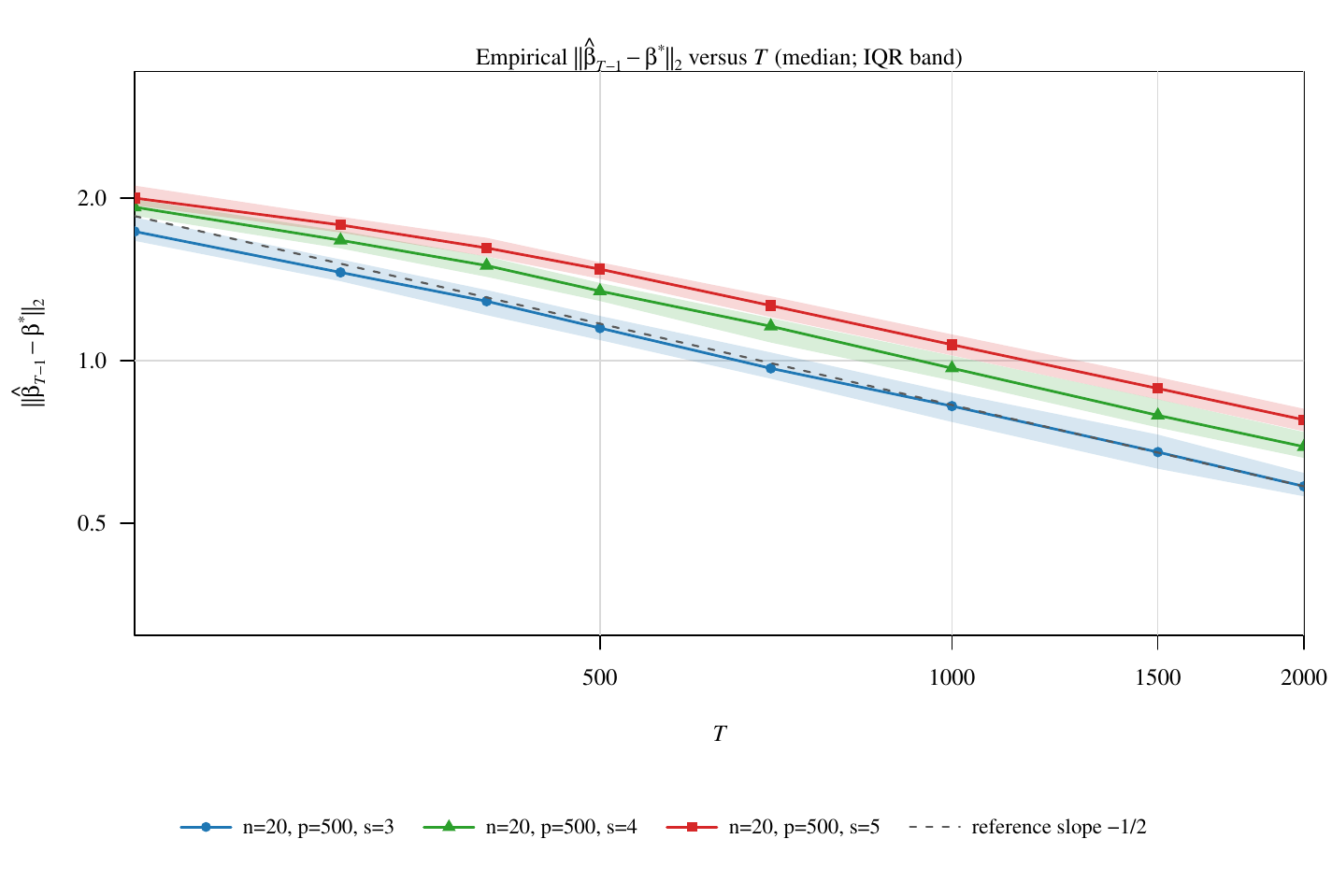}
\caption{Empirical estimation-rate diagnostic for the online penalized
estimator at \(n=20\), \(p=500\), for \(s\in\{3,4,5\}\), based on
\(N=100\) independent replications. Solid curves show the median
\(\|\hat\bbeta_{T-1}-\bbeta^*\|_2\), and shaded bands show the
interquartile range. The gray dashed reference line has slope \(-1/2\)
on the log-log scale.}
\label{fig:app-estimation-rate}
\end{figure}

\section{Proofs of Main Theorems}\label{app: proof main thms}
\subsection{Proof of Theorem~\ref{thm:lasso-rates} (Rates for the penalized estimator)}\label{sec: proof thm lasso rates}
We first invoke the following lemma, which provides uniform entrywise bounds for the gradient and Hessian of the negative log-likelihood.  
\begin{lemma}\label{lem: grad hessian rates}
    Under Assumption~\ref{asp: assortment const}, with probability at least $1 - O(T^{-1})$, the following two bounds hold for all $t \in [T]$ and all $\bbeta \in \cB_1(\bbeta^*,3\tau)$,
    \begin{align}
        \|  \nabla_{\bbeta} \ell_{t-1}(\bbeta^*)\|_{\infty} & \le C \nu\cdot \left( \log (T p) +   \sqrt{t\log (Tp)}\right)\label{eq: grad rate},\\
        \|\nabla_{\bbeta}^2 \ell_{t-1}(\bbeta) - (t-1) \cdot \bSigma^*\|_{\max} &\le C\nu^2\cdot\left(C_nt\tau+\log(Tp) + \sqrt{t \log(Tp)} \right),\label{eq: hes rate}
    \end{align}
    where $C > 0$ is a large enough constant, and the constant \(C_n\) admits the following forms under the respective conditions of Assumption~\ref{asp: assortment const}:
\begin{enumerate}
    \item If \(\bcS^K=\{\cS\subseteq[n]: |\cS|=K\}\) and \(\rho \le 2\), then
    \[
    C_n \;=\; \bar\mu  K^3 \nu \sigma_r^{-1}  \sqrt{\log n}.
    \]
    \item If \(\bcS^K \subseteq \{\cS\subseteq[n]: |\cS|=K\}\) and, for any \(\cS,\cS'\in \bcS^K\), \(|\cS\cap \cS'|\le (K/\rho^2-1)\vee 0\), then
    \[
    C_n \;=\; K \bar\mu   \nu \sigma_r^{-1} \sqrt{\log n}(K/\rho^2 \vee 1).
    \]
\end{enumerate}
\end{lemma}
The proof of Lemma~\ref{lem: grad hessian rates} is deferred to Section~\ref{sec: proof lem grad hessian rates}. Given a $t \in [T]$, let $\Delta_t := \hat\bbeta_{t-1}-\bbeta^*$. By the optimality condition for \eqref{eq: theta hat},
\begin{align*}
    \ell_{t-1}(\hat\bbeta_{t-1}) + \lambda_{t-1} \|\hat\bbeta_{t-1}\|_1 \le  \ell_{t-1}(\bbeta^*) + \lambda_{t-1} \|\bbeta^*\|_1.
\end{align*}
Rearranging terms and applying a Taylor expansion yields
\begin{equation}\label{eq: lasso decomp}
    \begin{aligned}
    \frac{1}{2} \Delta_t^{\top} \nabla^2_{\bbeta} \ell_{t-1}(\tilde\bbeta) \Delta_t &\le \lambda_{t-1} \|\bbeta^*\|_1 - \lambda_{t-1} \|\hat\bbeta_{t-1}\|_1 - \nabla_{\bbeta} \ell_{t-1}(\bbeta^*)^{\top} \Delta_t\\
    & \le \|\nabla_{\bbeta} \ell_{t-1}(\bbeta^*)\|_{\infty} \cdot\|\Delta_t\|_1 + \lambda_{t-1} \|[\Delta_t]_{\cI_0}\|_1 - \lambda_{t-1} \|[\Delta_t]_{\cI^c_0}\|_1\\
    & \le \frac{\lambda_{t-1}}{2} ( 3 \|[\Delta_t]_{\cI_0}\|_1 -  \|[\Delta_t]_{\cI^c_0}\|_1),
 \end{aligned}
\end{equation}
where the last inequality holds with probability at least $1-O(T^{-1})$ for all $t \in [T]$ by \eqref{eq: grad rate} in Lemma~\ref{lem: grad hessian rates}, provided that
\[
\lambda_{t-1} \ge C\nu\Bigl\{\log(Tp)+\sqrt{t\log(Tp)}\Bigr\}
\]
for a sufficiently large absolute constant $C>0$. The last inequality in \eqref{eq: lasso decomp} also gives the cone condition $\|[\Delta_t]_{\cI_0^c}\|_1 \le 3 \|[\Delta_t]_{\cI_0}\|_1$ because $\frac{1}{2} \Delta_t^{\top} \nabla^2_{\bbeta} \ell_{t-1}(\tilde\bbeta) \Delta_t \ge 0$, and that 
\begin{equation}\label{eq: delta lower bd}
    \Delta_t^{\top} \nabla^2_{\bbeta} \ell_{t-1}(\tilde\bbeta) \Delta_t \le {\lambda_{t-1}} ( 3 \|[\Delta_t]_{\cI_0}\|_1 -  \|[\Delta_t]_{\cI^c_0}\|_1) \le {3\lambda_{t-1} \sqrt{s} \|[\Delta_t]_{\cI^c_0}\|_2} \le {3\lambda_{t-1} \sqrt{s} \|\Delta_t\|_2}.
\end{equation}
We now study the term $\Delta_t^{\top} \nabla^2_{\bbeta} \ell_{t-1}(\tilde\bbeta) \Delta_t$. Under Assumption~\ref{asp: cov}, with probability at least $1 - O(T^{-1})$, uniformly over $t \in [T]$, we have 
\begin{align*}
    \Delta_t^{\top} \nabla^2_{\bbeta} \ell_{t-1}(\tilde\bbeta) \Delta_t & = \Delta_t^{\top} \left( \nabla^2_{\bbeta} \ell_{t-1}(\tilde\bbeta) - (t-1) \cdot \bSigma^* + (t-1) \cdot \bSigma^* \right) \Delta_t \\
    & \ge (t-1)\cdot \underline\lambda \|\Delta_t\|_2^2 -  \|\nabla_{\bbeta}^2 \ell_{t-1}(\bbeta) - (t-1) \cdot \bSigma^*\|_{\max} \|\Delta_t\|_1^2\\
    & \ge (t-1)\cdot \underline\lambda \|\Delta_t\|_2^2 - 16 C\nu^2 s\cdot\left(C_nt\tau+\log(Tp) + \sqrt{t \log(Tp)} \right) \|\Delta_t\|_2^2,
\end{align*}
where the last inequality follows from \eqref{eq: hes rate} in Lemma~\ref{lem: grad hessian rates} and the inequality 
$$
\|\Delta_t\|_1 =  \|[\Delta_t]_{\cI_0}\|_1 + \|[\Delta_t]_{\cI^c_0}\|_1 \le 4  \|[\Delta_t]_{\cI_0}\|_1 \le 4 \sqrt{s} \|[\Delta_t]_{\cI_0}\|_2 \le 4 \sqrt{s} \|\Delta_t\|_2. 
$$
For all $t\in[T]$ satisfying $t \ge C'\nu^4 s^2\underline\lambda^{-2}\log(Tp)$ for a sufficiently large constant $C'>0$, and assuming $\tau \le c\,\underline\lambda\,(C_n\nu^2 s)^{-1}$ for a sufficiently small constant $c>0$, it follows that
$$ \Delta_t^{\top} \nabla^2_{\bbeta} \ell_{t-1}(\tilde\bbeta) \Delta_t \ge (t-1)\cdot \underline\lambda \|\Delta_t\|_2^2 - \frac{t-1}{2} \cdot \underline\lambda \|\Delta_t\|_2^2 = \frac{t-1}{2}  \underline\lambda \|\Delta_t\|_2^2. $$
Combining this with \eqref{eq: delta lower bd} yields
\[
\|\Delta_t\|_2 \;\lesssim\; \frac{\lambda_{t-1}\sqrt{s}}{\underline\lambda t}
\;\le\; C \nu \sqrt{\frac{s\log(Tp)}{\underline\lambda^2 t}},
\qquad
\|\Delta_t\|_1 \;\le\; 4\sqrt{s}\,\|\Delta_t\|_2
\;\le\; C \nu s \sqrt{\frac{\log(Tp)}{\underline\lambda^2 t}},
\]
for a sufficiently large constant $C>0$, provided also that
\[
\lambda_{t-1} \le C'' \nu\Bigl\{\log(Tp)+\sqrt{t\log(Tp)}\Bigr\}
\]
for some absolute constant $C''>0$ independent of $t$.
\subsection{Proof of Corollary~\ref{col: unique solution decomp} (Effective support recovery)}
\label{sec: proof col unique solution decomp}
We begin by establishing uniqueness of the linear predictors along the sample path. Namely, for each
\(t\in[T-1]\), define
\[
u_t(\bbeta):=\big(v_{tj}^\top\bbeta\big)_{j\in S_t}\in\RR^{|S_t|},
\qquad
u(\bbeta):=\big(u_1(\bbeta)^\top,\ldots,u_{T-1}(\bbeta)^\top\big)^\top .
\]
Then, for any minimizer
\begin{equation}\label{eq: lasso T}
    \hat\bbeta_{T-1}\in \arg\min_{\|\bbeta-\hat\bbeta_0\|_1\le 2\tau}\Big\{\ell_{T-1}(\bbeta)+\lambda_{T-1}\|\bbeta\|_1\Big\},
\end{equation}
the vector \(u(\hat\bbeta_{T-1})\) is uniquely determined.

To see this, note that \(\ell_{T-1}(\bbeta)\) depends on \(\bbeta\) only through \(u(\bbeta)\) and can be written as
\(\ell_{T-1}(\bbeta)=h(u(\bbeta))\), where
\[
h(u)
:=
-\sum_{t=1}^{T-1}\log\Bigg\{\frac{\exp(u_{t,i_t})}{1+\sum_{k\in S_t}\exp(u_{tk})}\Bigg\}.
\]
Moreover, \(h(\cdot)\) is twice continuously differentiable and its Hessian is block diagonal:
\[
\nabla_u^2 h(u)=\operatorname{diag}\big(H_1(u),\ldots,H_{T-1}(u)\big),
\qquad
H_t(u)=\diag\!\big(p_t(u)\big)-p_t(u)p_t(u)^\top,
\]
where \(p_t(u)=(p_{tj}(u))_{j\in S_t}\) with
\begin{equation}\label{eq: p_tj}
    p_{tj}(u):=\frac{\exp(u_{tj})}{1+\sum_{k\in S_t}\exp(u_{tk})},
\qquad
p_{t0}(u):=\frac{1}{1+\sum_{k\in S_t}\exp(u_{tk})}.
\end{equation}
By Gershgorin's circle theorem, for each \(t\in[T-1]\),
\begin{align*}
  \lambda_{\min}\big(H_t(u)\big)
& \ge \min_{j\in S_t}\Big\{p_{tj}(u)-\sum_{k\in S_t\setminus\{j\}}p_{tj}(u)p_{tk}(u)\Big\}
=\min_{j\in S_t} p_{tj}(u)\,p_{t0}(u)\\
& =\min_{j\in S_t}\frac{\exp(u_{tj})}{\Big(1+\sum_{k\in S_t}\exp(u_{tk})\Big)^2}
>0.
\end{align*}
Hence \(h(\cdot)\) is strictly convex in \(u\). Since the feasible set \(\{\bbeta:\|\bbeta-\hat\bbeta_0\|_1\le 2\tau\}\)
is convex and \(u(\bbeta)\) is affine in \(\bbeta\), strict convexity of \(h\) implies that any two minimizers
\(\hat\bbeta_{T-1}^{(1)},\hat\bbeta_{T-1}^{(2)}\) must satisfy \(u(\hat\bbeta_{T-1}^{(1)})=u(\hat\bbeta_{T-1}^{(2)})\), proving
the claim.

Now let \(\hat\bbeta_{T-1}\) be any minimizer of~\eqref{eq: lasso T}. As noted in
Remark~\ref{rmk: constraint removal}, under the condition
\(\nu s\sqrt{\log(Tp)/(\underline\lambda^{2}T)} \le C^{-1}\tau\) for a sufficiently large \(C>0\), the estimator \(\hat\bbeta_{T-1}\) lies in the feasible set \(\cB_1(\hat\bbeta_0,2\tau)\) and hence coincides with the unconstrained \(\ell_1\)-penalized estimator. Therefore, the Karush--Kuhn--Tucker (KKT) conditions imply that
\begin{equation}\label{eq:kkt_global}
-\nabla_{\bbeta}\ell_{T-1}(\hat\bbeta_{T-1})
=
\sum_{t=1}^{T-1}
\left\{
\bv_{t,i_t}
-
\sum_{j\in \cS_t\cup\{0\}} p_{tj}\big(u(\hat\bbeta_{T-1})\big)\,\bv_{tj}
\right\}
=
\lambda_{T-1}\,\sgn(\hat\bbeta_{T-1}),
\end{equation}
where \(p_{tj}(\cdot)\) is defined in~\eqref{eq: p_tj} and \(\sgn(\hat\bbeta_{T-1})\) denotes a
subgradient of \(\|\bbeta\|_1\) at \(\hat\bbeta_{T-1}=(\hat\bbeta_{T-1,1},\ldots,\hat\bbeta_{T-1,p})^\top\), that is,
\[
\sgn(\hat\bbeta_{T-1,j})\in
\begin{cases}
\{\sign(\hat\bbeta_{T-1,j})\}, & \text{if }\hat\bbeta_{T-1,j}\neq 0,\\[2pt]
[-1,1], & \text{if }\hat\bbeta_{T-1,j}=0,
\end{cases}
\qquad j=1,\ldots,p.
\]
Moreover, since \(u(\hat\bbeta_{T-1})\) is unique, the left-hand side of~\eqref{eq:kkt_global} is uniquely determined, and hence so is the associated KKT subgradient.

Consequently, suppose that \([\nabla^2_{\bbeta}\ell_{T-1}(\bbeta)]_{\cI_*,\cI_*}\) is positive definite for all
\(\bbeta\in\cB_1(\hat\bbeta_0,2\tau)\), and that there exists a feasible point
\(\hat\bbeta_{T-1}\in\cB_1(\hat\bbeta_0,2\tau)\) satisfying
\(\sgn([\hat\bbeta_{T-1}]_{\cI_*})=\sgn([\bbeta^*]_{\cI_*})\),
\([\hat\bbeta_{T-1}]_{\cI_*^c}=\mathbf 0\), and the KKT conditions
\begin{equation}\label{eq: kkt support}
-\big[\nabla_{\bbeta}\ell_{T-1}(\hat\bbeta_{T-1})\big]_{\cI_*}
=
\lambda_{T-1}\,\big[\sgn(\bbeta^*)\big]_{\cI_*},
\end{equation}
and
\begin{equation}\label{eq: kkt off support}
\big|\big[\nabla_{\bbeta}\ell_{T-1}(\hat\bbeta_{T-1})\big]_j\big|<\lambda_{T-1},
\qquad \forall j\notin \cI_* .
\end{equation}
Then \(\hat\bbeta_{T-1}\) is the unique solution to~\eqref{eq: lasso T} and it recovers the effective support of \(\bbeta^*\).

We first verify that \([\nabla_{\bbeta}^2\ell_{T-1}(\bbeta)]_{\cI_*,\cI_*}\) is positive definite uniformly over \(\cB_1(\hat\bbeta_0,2\tau)\) with high probability. Recall
\(\bSigma_{\cI_*}^*:=\bSigma^*_{\cI_*,\cI_*}\). Then
\[
\lambda_{\min}(\bSigma^*)\le \lambda_{\min}(\bSigma_{\cI_*}^*)
\le \lambda_{\max}(\bSigma_{\cI_*}^*)\le \lambda_{\max}(\bSigma^*).
\]
By~\eqref{eq: hes rate} in Lemma~\ref{lem: grad hessian rates}, with probability at least
\(1-O(T^{-1})\), for all \(\bbeta\in \cB_1(\hat\bbeta_0,2\tau)\subseteq \cB_1(\bbeta^*,3\tau)\),
\begin{align}
\big\|[\nabla_{\bbeta}^2 \ell_{T-1}(\bbeta)]_{\cI_*,\cI_*}-(T-1)\bSigma_{\cI_*}^*\big\|_{2}
&\le s_*\,\big\|\nabla_{\bbeta}^2 \ell_{T-1}(\bbeta)-(T-1)\bSigma^*\big\|_{\max} \notag\\
&\le C\nu^2 s_*\Big(C_nT\tau+\log(Tp)+\sqrt{T\log(Tp)}\Big).\label{eq: sig support bound}
\end{align}
Under Assumption~\ref{asp: cov}, we have
\[
\nu^2s_*\ge s_*\|\bSigma_{\cI_*}^*\|_{\max}\ge \lambda_{\max}(\bSigma_{\cI_*}^*)\ge \lambda_{\min}(\bSigma_{\cI_*}^*)\ge \lambda_{\min}(\bSigma^*)\ge \underline\lambda.
\]
Consequently, if \(T\ge C\nu^4s^2\underline\lambda^{-2}\log(Tp)\) and
\(\tau \le c\underline\lambda(C_n\nu^2s)^{-1}\) for sufficiently large \(C>0\) and sufficiently small \(c>0\), then for all \(\bbeta\in \cB_1(\hat\bbeta_0,2\tau)\),
\begin{align}
\lambda_{\min}\!\left([\nabla_{\bbeta}^2 \ell_{T-1}(\bbeta)]_{\cI_*,\cI_*}\right)
&\ge (T-1)\lambda_{\min}(\bSigma_{\cI_*}^*)
-\big\|[\nabla_{\bbeta}^2 \ell_{T-1}(\bbeta)]_{\cI_*,\cI_*}-(T-1)\bSigma_{\cI_*}^*\big\|_{2} \notag \\
&\ge (T-1)\underline\lambda
- C\nu^2 s_*\Big(C_nT\tau+\log(Tp)+\sqrt{T\log(Tp)}\Big) \notag\\
&\ge (T-1)\underline\lambda/2 \;>\;0, \label{eq: min eigen sig support}
\end{align}
and hence \([\nabla_{\bbeta}^2\ell_{T-1}(\bbeta)]_{\cI_*,\cI_*}\) is positive definite uniformly over \(\cB_1(\hat\bbeta_0,2\tau)\) with high probability.

We next construct an estimator \(\hat\bbeta_{T-1}\in \cB_1(\hat\bbeta_0,2\tau)\) that satisfies
\(\sgn([\hat\bbeta_{T-1}]_{\cI_*})=\sgn([\bbeta^*]_{\cI_*})\),
\([\hat\bbeta_{T-1}]_{\cI_*^c}=\mathbf 0\), and the KKT conditions
\eqref{eq: kkt support} and \eqref{eq: kkt off support}. Set
\([\hat\bbeta_{T-1}]_{\cI_*^c}=\mathbf{0}\) and define
\([\hat\bbeta_{T-1}]_{\cI_*}\) as a minimizer of the restricted problem
\begin{equation}\label{eq:restricted_lasso}
[\hat\bbeta_{T-1}]_{\cI_*}\in
\arg\min_{\bbeta\in\RR^{s_*}:\ \|\bbeta-[\bbeta^*]_{\cI_*}\|_1\le \tau}
\Big\{\tilde\ell_{T-1}(\bbeta)+\lambda_{T-1}\|\bbeta\|_1\Big\},
\end{equation}
where
\[
\tilde\ell_{T-1}(\bbeta)
=
-\sum_{t'=1}^{T-1}\log\left\{
\frac{\exp\{[\bv_{t',\,i_{t'}}]_{\cI_*}^{\top}\bbeta\}}
{1+\sum_{k\in \cS_{t'}}\exp\{[\bv_{t'k}]_{\cI_*}^{\top}\bbeta\}}
\right\}
\]
is the negative log-likelihood restricted to the effective support. Denote
\(\hat\bbeta_{\cI_*}:=[\hat\bbeta_{T-1}]_{\cI_*}\) and
\(\bbeta_{\cI_*}:=[\bbeta^*]_{\cI_*}\), and let
\(\Delta_T:=\hat\bbeta_{\cI_*}-\bbeta_{\cI_*}\). By the optimality condition for \eqref{eq:restricted_lasso},
\[
    \tilde\ell_{T-1}(\hat\bbeta_{\cI_*})+\lambda_{T-1}\|\hat\bbeta_{\cI_*}\|_1
    \le
    \tilde\ell_{T-1}(\bbeta_{\cI_*})+\lambda_{T-1}\|\bbeta_{\cI_*}\|_1 .
\]
Rearranging terms and applying a Taylor expansion yields
\begin{equation}\label{eq: lasso decomp 1}
    \begin{aligned}
    \frac{1}{2}\Delta_T^{\top}\tilde\nabla^2_{\bbeta}\ell_{T-1}(\tilde\bbeta)\Delta_T
    &\le
    \lambda_{T-1}\|\bbeta_{\cI_*}\|_1-\lambda_{T-1}\|\hat\bbeta_{\cI_*}\|_1
    -
    \nabla_{\bbeta}\tilde\ell_{T-1}(\bbeta_{\cI_*})^{\top}\Delta_T\\
    &\le
    \|\nabla_{\bbeta}\tilde\ell_{T-1}(\bbeta_{\cI_*})\|_{\infty}\|\Delta_T\|_1
    +
    \lambda_{T-1}\|\Delta_T\|_1 .
    \end{aligned}
\end{equation}
The restricted gradient satisfies
\[
    \nabla_{\bbeta}\tilde\ell_{T-1}(\bbeta_{\cI_*})
    =
    [\nabla_{\bbeta}\ell_{T-1}(\bbeta^\circ)]_{\cI_*},
\]
where \(\bbeta^\circ\) agrees with \(\bbeta^*\) on \(\cI_*\) and is zero on \(\cI_*^c\). Hence, by a Taylor expansion between \(\bbeta^\circ\) and \(\bbeta^*\),
\[
\begin{aligned}
    \|\nabla_{\bbeta}\tilde\ell_{T-1}(\bbeta_{\cI_*})
    -[\nabla_{\bbeta}\ell_{T-1}(\bbeta^*)]_{\cI_*}\|_{\infty}
    &\le
    \sup_{\bbeta\in\cB_1(\bbeta^*,3\tau)}
    \|\nabla_{\bbeta}^2\ell_{T-1}(\bbeta)\|_{\max}
    \|[\bbeta^*]_{\cI_{\rm wk}}\|_1  \\
    &\le
    C T\nu^2\|[\bbeta^*]_{\cI_{\rm wk}}\|_1
    =
    o\{\nu\sqrt{T\log(Tp)}\}.
\end{aligned}
\]
Together with \eqref{eq: grad rate} in Lemma~\ref{lem: grad hessian rates}, this gives
\[
    \|\nabla_{\bbeta}\tilde\ell_{T-1}(\bbeta_{\cI_*})\|_{\infty}
    \le
    \lambda_{T-1}/2
\]
with probability at least \(1-O(T^{-1})\), provided \(C_\lambda\) is sufficiently large. Therefore \eqref{eq: lasso decomp 1} gives
\[
    \frac{1}{2}\Delta_T^{\top}\tilde\nabla^2_{\bbeta}\ell_{T-1}(\tilde\bbeta)\Delta_T
    \le
    \frac{3\lambda_{T-1}}{2}\|\Delta_T\|_1 .
\]
Using \eqref{eq: min eigen sig support} and \(\|\Delta_T\|_1\le \sqrt{s_*}\|\Delta_T\|_2\), we obtain
\[
    \frac{(T-1)\underline\lambda}{4}\|\Delta_T\|_2^2
    \le
    \frac{3\lambda_{T-1}}{2}\sqrt{s_*}\|\Delta_T\|_2 .
\]
Thus, with probability at least \(1-O(T^{-1})\),
\begin{align}
\big\|[\hat\bbeta_{T-1}]_{\cI_*}-[\bbeta^*]_{\cI_*}\big\|_2
&\le C\nu\sqrt{\frac{s_*\log(Tp)}{\underline\lambda^2T}},
\label{eq: support l2-rate}\\
\big\|[\hat\bbeta_{T-1}]_{\cI_*}-[\bbeta^*]_{\cI_*}\big\|_1
&\le C\nu s_*\sqrt{\frac{\log(Tp)}{\underline\lambda^2T}}.
\label{eq: support l1-rate}
\end{align}
In particular, \([\hat\bbeta_{T-1}]_{\cI_*}\) belongs to the feasible set
\(\{\bbeta\in\RR^{s_*}:\|\bbeta-[\bbeta^*]_{\cI_*}\|_1\le\tau\}\) when
\(\nu s_*\sqrt{\log(Tp)/(\underline\lambda^{2}T)}
\le
\nu s\sqrt{\log(Tp)/(\underline\lambda^{2}T)}
\le c\tau\) for small enough \(c>0\), and thus coincides with the unconstrained
\(\ell_1\)-penalized estimator for \eqref{eq:restricted_lasso}. Moreover, under Assumption~\ref{asp: min signal}, \eqref{eq: support l2-rate} implies that for any \(j\in\cI_*\),
\[
|\beta_j^*-\hat\bbeta_{T-1,j}|
\le
C\nu\sqrt{\frac{s_*\log(Tp)}{\underline\lambda^2T}}
\le
\frac{|\beta_j^*|}{2},
\qquad
\sgn(\hat\bbeta_{T-1,j})=\sgn(\beta_j^*)\neq 0.
\]
Since the restricted solution lies in the interior of the restricted feasible set, the restricted KKT condition gives
\[
    -[\nabla_{\bbeta}\ell_{T-1}(\hat\bbeta_{T-1})]_{\cI_*}
    =
    \lambda_{T-1}\sgn([\hat\bbeta_{T-1}]_{\cI_*}).
\]
The sign consistency just proved then yields \eqref{eq: kkt support}. By the triangle inequality,
\[
\|\hat\bbeta_{T-1}-\hat\bbeta_0\|_1
\le
\|[\hat\bbeta_{T-1}]_{\cI_*}-[\bbeta^*]_{\cI_*}\|_1
+
\|[\bbeta^*]_{\cI_{\rm wk}}\|_1
+
\|\hat\bbeta_0-\bbeta^*\|_1
<2\tau,
\]
and hence \(\hat\bbeta_{T-1}\) lies in the feasible set \(\cB_1(\hat\bbeta_0,2\tau)\). It remains to verify that the off-support condition~\eqref{eq: kkt off support} also holds for the constructed \(\hat\bbeta_{T-1}\).

Define the integrated Hessian
\begin{equation}\label{eq: int hess}
    \bar H
    :=
    \int_0^1
    \nabla_{\bbeta}^2\ell_{T-1}
    \left(
        \bbeta^*+a(\hat\bbeta_{T-1}-\bbeta^*)
    \right)
    \,da .
\end{equation}
Since both \(\bbeta^*\) and \(\hat\bbeta_{T-1}\) lie in
\(\cB_1(\hat\bbeta_0,2\tau)\subseteq\cB_1(\bbeta^*,3\tau)\), all uniform Hessian bounds above apply to \(\bar H\). By the fundamental theorem of calculus,
\[
    \nabla_{\bbeta}\ell_{T-1}(\hat\bbeta_{T-1})
    =
    \nabla_{\bbeta}\ell_{T-1}(\bbeta^*)
    +
    \bar H(\hat\bbeta_{T-1}-\bbeta^*) .
\]
Applying this identity to \eqref{eq: kkt support}, we obtain
\begin{align*}
    & \lambda_{T-1}[\sgn(\bbeta^*)]_{\cI_*} =\\
    & \qquad -\big[\nabla_{\bbeta}\ell_{T-1}(\bbeta^*)\big]_{\cI_*}
-
\bar H_{\cI_*,\cI_*}
[\hat\bbeta_{T-1}-\bbeta^*]_{\cI_*}
+
\bar H_{\cI_*,\cI_{\rm wk}}
[\bbeta^*]_{\cI_{\rm wk}},
\end{align*}
and hence
\begin{equation}\label{eq: taylor beta support}
\begin{aligned}
[\hat\bbeta_{T-1}-\bbeta^*]_{\cI_*}
&=
-\bar H_{\cI_*,\cI_*}^{-1}\\
&\quad\times
\left(
[\nabla_{\bbeta}\ell_{T-1}(\bbeta^*)]_{\cI_*}
+\lambda_{T-1}[\sgn(\bbeta^*)]_{\cI_*}
-
\bar H_{\cI_*,\cI_{\rm wk}}
[\bbeta^*]_{\cI_{\rm wk}}
\right).
\end{aligned}
\end{equation}
Substituting \eqref{eq: taylor beta support} into the left-hand side of \eqref{eq: kkt off support} yields
\[
\begin{aligned}
[\nabla_{\bbeta}\ell_{T-1}(\hat\bbeta_{T-1})]_{\cI_*^c}
&=
[\nabla_{\bbeta}\ell_{T-1}(\bbeta^*)]_{\cI_*^c}
+
\bar H_{\cI_*^c,\cI_*}
[\hat\bbeta_{T-1}-\bbeta^*]_{\cI_*}\\
&\quad
-
\bar H_{\cI_*^c,\cI_{\rm wk}}
[\bbeta^*]_{\cI_{\rm wk}}\\
&=
[\nabla_{\bbeta}\ell_{T-1}(\bbeta^*)]_{\cI_*^c}\\
&\quad
-
\bar H_{\cI_*^c,\cI_*}
\bar H_{\cI_*,\cI_*}^{-1}
[\nabla_{\bbeta}\ell_{T-1}(\bbeta^*)]_{\cI_*}\\
&\quad
-
\lambda_{T-1}
\bar H_{\cI_*^c,\cI_*}
\bar H_{\cI_*,\cI_*}^{-1}
[\sgn(\bbeta^*)]_{\cI_*}\\
&\quad
+
\bar H_{\cI_*^c,\cI_*}
\bar H_{\cI_*,\cI_*}^{-1}
\bar H_{\cI_*,\cI_{\rm wk}}
[\bbeta^*]_{\cI_{\rm wk}}\\
&\quad
-
\bar H_{\cI_*^c,\cI_{\rm wk}}
[\bbeta^*]_{\cI_{\rm wk}} .
\end{aligned}
\]
Therefore,
\[
\begin{aligned}
\left\|[\nabla_{\bbeta}\ell_{T-1}(\hat\bbeta_{T-1})]_{\cI_*^c}\right\|_{\infty}
&\le {\rm (I)}+\lambda_{T-1}{\rm (II)}+{\rm (III)},
\end{aligned}
\]
where
\[
\begin{aligned}
{\rm (I)}
&:=
\left\|
[\nabla_{\bbeta}\ell_{T-1}(\bbeta^*)]_{\cI_*^c}
-
\bar H_{\cI_*^c,\cI_*}
\bar H_{\cI_*,\cI_*}^{-1}
[\nabla_{\bbeta}\ell_{T-1}(\bbeta^*)]_{\cI_*}
\right\|_{\infty},\\
{\rm (II)}
&:=
\left\|
\bar H_{\cI_*^c,\cI_*}
\bar H_{\cI_*,\cI_*}^{-1}
[\sgn(\bbeta^*)]_{\cI_*}
\right\|_{\infty},
\end{aligned}
\]
and
\[
\begin{aligned}
{\rm (III)}
&:=
\left\|
\left\{
\bar H_{\cI_*^c,\cI_*}
\bar H_{\cI_*,\cI_*}^{-1}
\bar H_{\cI_*,\cI_{\rm wk}}
-
\bar H_{\cI_*^c,\cI_{\rm wk}}
\right\}
[\bbeta^*]_{\cI_{\rm wk}}
\right\|_{\infty}.
\end{aligned}
\]

We now bound terms~(I), (II), and (III). First, we control
\(\|\bar H_{\cI_*^c,\cI_*}
\bar H_{\cI_*,\cI_*}^{-1}\|_\infty\). For any
\(\ell\in\mathcal I_*^c\), by \eqref{eq: hes rate} in Lemma~\ref{lem: grad hessian rates}, together with \eqref{eq: sig support bound} and \eqref{eq: min eigen sig support}, it holds with probability at least \(1-O(T^{-1})\) that
\begin{align}
    &\|\bar H_{\cI_*,\cI_*}^{-1}  - [(T-1)\bSigma^*_{\cI_*}]^{-1}\|_2 \notag\\
    &\quad\le
    \|\bar H_{\cI_*,\cI_*}^{-1}\|_2
    \|[(T-1)\bSigma_{\cI_*}^*]^{-1}\|_2
    \|\bar H_{\cI_*,\cI_*}-(T-1)\bSigma^*_{\cI_*}\|_2 \notag\\
    &\quad\lesssim
    \frac{1}{T^2\underline\lambda^2}\nu^2s_*\Big(C_nT\tau+\log(Tp)+\sqrt{T\log(Tp)}\Big),\label{eq: inv bd}\\
    &\|\bar H_{\cI_*,\ell} - [(T-1)\bSigma^*]_{\cI_*,\ell}\|_2
    \lesssim
    \sqrt{s_*}\nu^2\Big(C_nT\tau+\log(Tp)+\sqrt{T\log(Tp)}\Big).\notag
\end{align}
Consequently,
\begin{align*}
&\left\|\bar H_{\cI_*,\cI_*}^{-1}
\bar H_{\cI_*,\ell}\right\|_1\\
&\quad\le
\left\|(\bSigma^*_{\cI_*})^{-1}[\bSigma^*]_{\cI_*,\ell}\right\|_1\\
&\qquad+
\sqrt{s_*}\|\bar H_{\cI_*,\cI_*}^{-1}-[(T-1)\bSigma^*_{\cI_*}]^{-1}\|_2
\|\bar H_{\cI_*,\ell}-[(T-1)\bSigma^*]_{\cI_*,\ell}\|_2\\
&\qquad+
\sqrt{s_*}\|[(T-1)\bSigma^*_{\cI_*}]^{-1}\|_2
\|\bar H_{\cI_*,\ell}-[(T-1)\bSigma^*]_{\cI_*,\ell}\|_2\\
&\qquad+
\sqrt{s_*}\|\bar H_{\cI_*,\cI_*}^{-1}-[(T-1)\bSigma^*_{\cI_*}]^{-1}\|_2
\|[(T-1)\bSigma^*]_{\cI_*,\ell}\|_2\\
&\quad\lesssim
\|\bSigma^*_{\cI_*^c,\cI_*}(\bSigma^*_{\cI_*})^{-1}\|_{\infty}
+
\frac{\nu^4s_*^2}{T\underline\lambda^2}
\Big(C_nT\tau+\log(Tp)+\sqrt{T\log(Tp)}\Big).
\end{align*}
 Under the conditions that \(\tau \le c\underline\lambda^2/(C_n\nu^4s_*^2)\) and \(T\ge C\nu^8s_*^4\log(Tp)/\underline\lambda^4\) for sufficiently large \(C > 0\) and sufficiently small \(c >0\), the second term in the last inequality is no larger than \(\gamma_0 / 2\). Therefore, by Assumption~\ref{asp: mutual inch},
\begin{equation}\label{eq: approx mutual inch}
\begin{aligned}
\|\bar H_{\cI_*^c,\cI_*}
\bar H_{\cI_*,\cI_*}^{-1}\|_{\infty}
&\le 1-\gamma_0+\gamma_0/2
\le 1-\gamma_0/2 .
\end{aligned}
\end{equation}
Applying \eqref{eq: approx mutual inch}, we obtain
\begin{align*}
{\rm (I)}
&\le
\left(
1+
\|\bar H_{\cI_*^c,\cI_*}
\bar H_{\cI_*,\cI_*}^{-1}\|_{\infty}
\right)
\|\nabla_{\bbeta}\ell_{T-1}(\bbeta^*)\|_{\infty}\\
&\le
2C\nu\{\log(Tp)+\sqrt{T\log(Tp)}\},
\end{align*}
and
\[
{\rm (II)}
\le
\|\bar H_{\cI_*^c,\cI_*}
\bar H_{\cI_*,\cI_*}^{-1}\|_{\infty}
\|[\sgn(\bbeta^*)]_{\cI_*}\|_{\infty}
\le
1-\gamma_0/2.
\]

It remains to bound the weak-tail term. Since the Hessian summands are uniformly bounded entrywise by \(C\nu^2\), and using \eqref{eq: approx mutual inch},
\[
\begin{aligned}
{\rm (III)}
&\le
\left(
1+
\|\bar H_{\cI_*^c,\cI_*}
\bar H_{\cI_*,\cI_*}^{-1}\|_{\infty}
\right)
\sup_{\bbeta\in\cB_1(\bbeta^*,3\tau)}
\|\nabla_{\bbeta}^2\ell_{T-1}(\bbeta)\|_{\max}
\|[\bbeta^*]_{\cI_{\rm wk}}\|_1\\
&\le
C T\nu^2\|[\bbeta^*]_{\cI_{\rm wk}}\|_1
=
o\{\nu\sqrt{T\log(Tp)}\}
=
o(\lambda_{T-1}).
\end{aligned}
\]
When \(\lambda_{T-1}=C_\lambda\nu\sqrt{T\log(Tp)}\) for a sufficiently large \(C_\lambda>0\), the preceding displays imply
\begin{align*}
\left\|[\nabla_{\bbeta}\ell_{T-1}(\hat\bbeta_{T-1})]_{\cI_*^c}\right\|_{\infty}
&\le
\gamma_0\lambda_{T-1}/4+(1-\gamma_0/2)\lambda_{T-1}+o(\lambda_{T-1})\\
&\le
(1-\gamma_0/4)\lambda_{T-1}
<
\lambda_{T-1}.
\end{align*}
Thus the KKT condition for the off-support set \eqref{eq: kkt off support} is satisfied for \(\hat\bbeta_{T-1}\). This proves that the constructed point is the unique global solution to \eqref{eq: lasso T} and that \(\mathcal I=\mathcal I_*\).

Moreover, since \([\hat\bbeta_{T-1}]_{\cI_*^c}=0\), we have
\[
\begin{aligned}
\|\hat\bbeta_{T-1}-\bbeta^*\|_1
&=
\|[\hat\bbeta_{T-1}]_{\cI_*}-[\bbeta^*]_{\cI_*}\|_1
+
\|[\bbeta^*]_{\cI_{\rm wk}}\|_1  \\
&\le
C\nu s_*\sqrt{\frac{\log(Tp)}{\underline\lambda^2T}}
+
\|[\bbeta^*]_{\cI_{\rm wk}}\|_1 .
\end{aligned}
\]
By Assumption~\ref{asp: min signal} and the condition on $\eta_T$,
\[
    \|[\bbeta^*]_{\cI_{\rm wk}}\|_1
    =
    o\left(
        \frac{1}{\nu}\sqrt{\frac{\log(Tp)}{T}}
    \right).
\]
Since \(\cI_*\ne\emptyset\), Assumption~\ref{asp: cov} gives
\(\underline\lambda\le\nu^2s_*\). Therefore,
\[
    \|[\bbeta^*]_{\cI_{\rm wk}}\|_1
    =
    o\left(
        \nu s_*\sqrt{\frac{\log(Tp)}{\underline\lambda^2T}}
    \right).
\]
Absorbing the weak-tail term into the leading rate yields \eqref{eq: l1 rate improved}, completing the proof.
\subsection{Proof of Corollary~\ref{col: Regret bound} (Regret bound)}\label{sec: proof col regret bound}

For each $t \in [T-1]$, we have
\begin{align*}
    R(\cS_t^* \mid \bbeta^*,\bv_t,\br_t)
    -
    R(\cS_t \mid \bbeta^*,\bv_t,\br_t)
    &= R(\cS_t^* \mid \bbeta^*,\bv_t,\br_t)
    -
    R(\cS_t^* \mid \hat\bbeta_{t-1},\bv_t,\br_t) \\
    &\quad +
    R(\cS_t^* \mid \hat\bbeta_{t-1},\bv_t,\br_t)
    -
    R(\cS_t \mid \hat\bbeta_{t-1},\bv_t,\br_t) \\
    &\quad +
    R(\cS_t \mid \hat\bbeta_{t-1},\bv_t,\br_t)
    -
    R(\cS_t \mid \bbeta^*,\bv_t,\br_t) \\
    &\le
    2 \max_{\cS \in \bcS^K}
    \big|R(\cS \mid \hat\bbeta_{t-1},\bv_t,\br_t) - R(\cS \mid \bbeta^*,\bv_t,\br_t)\big| \\
    &\le
    2 \max_{\bbeta \in \RR^p}\max_{\cS \in \bcS^K}
    \|\nabla_{\bbeta} R(\cS \mid \bbeta,\bv_t,\br_t)\|_{\infty}
    \cdot
    \|\hat\bbeta_{t-1} - \bbeta^*\|_1 \\
    &\lesssim
    \nu \|\hat\bbeta_{t-1} - \bbeta^*\|_1 \cdot \max_{j \in [n]} |r_{tj}|,
\end{align*}
where the first inequality follows from the fact that
$$
R(\cS_t^* \mid \hat\bbeta_{t-1},\bv_t,\br_t)
-
R(\cS_t \mid \hat\bbeta_{t-1},\bv_t,\br_t)
\le 0,
$$
since $\cS_t$ is the maximizer of $R(\cS \mid \hat\bbeta_{t-1},\bv_t,\br_t)$ by definition, and the last inequality follows from \eqref{eq: gradient revenue}.

We next have the decomposition
\begin{align*}
    &\big|R(\cS \mid \hat\bbeta_{t-1},\bv_t,\br_t)
    -
    R(\cS \mid \bbeta^*,\bv_t,\br_t)\big| \\
    &\quad =
    \big|R(\cS \mid \hat\bbeta_{t-1},\bv_t,\br_t)
    -
    R(\cS \mid \bbeta^*,\bv_t,\br_t)\big|
    \cdot
    \big(
    \II\{\text{\eqref{eq:l1-rate} holds}\}
    +
    \II\{\text{\eqref{eq:l1-rate} fails}\}
    \big) \\
    &\quad \le
    \max_{j \in [n]} |r_{tj}|
    \frac{\nu^2 s}{\underline\lambda}
    \sqrt{\frac{\log(Tp)}{t}}
    \cdot
    \II\{\text{\eqref{eq:l1-rate} holds}\}
    +
    \max_{j \in [n]} |r_{tj}|
    \cdot
    \II\{\text{\eqref{eq:l1-rate} fails}\}.
\end{align*}
Hence,
\begin{align*}
    &\EE\Big\{
    R(\cS_t^* \mid \bbeta^*,\bv_t,\br_t)
    -
    R(\cS_t \mid \bbeta^*,\bv_t,\br_t)
    \Big\} \\
    &\quad \lesssim
    \frac{\nu^2 s}{\underline\lambda}
    \sqrt{\frac{\log(Tp)}{t}}
    \EE\Big[\max_{j \in [n]} |r_{tj}|\Big]
    +
    \frac{1}{T}
    \Big(
    \EE\Big[\max_{j \in [n]} |r_{tj}|\Big]^2
    \Big)^{1/2} \\
    &\quad \lesssim
    \frac{\nu^2 \bar\mu s}{\underline\lambda}
    \sqrt{\frac{\log(Tp)}{t}}
    +
    \frac{\bar\mu}{T}
    \lesssim
    \frac{\nu^2 \bar\mu s}{\underline\lambda}
    \sqrt{\frac{\log(Tp)}{t}},
\end{align*}
where the penultimate inequality follows from the maximal inequality for sub-Gaussian random variables and the log-Sobolev inequality for the maximum of non-centered Gaussian random variables \cite{belloni2024anticon}, together with Assumption~\ref{asp: assortment const}.

Combining the above with the fact that
$$
\sum_{t=1}^{T-1} \frac{1}{\sqrt{t}}
\lesssim
\sum_{t=1}^{T-1} \frac{1}{\sqrt{t}+\sqrt{t+1}}
=
\sum_{t=1}^{T-1} (\sqrt{t+1}-\sqrt{t})
=
\sqrt{T},
$$
the claim follows.
\subsection{Proof of Corollary~\ref{col: error decomp debias lasso} (Debiased-estimator error decomposition)}
\label{sec: proof col error decomp lasso}

By Corollary~\ref{col: unique solution decomp}, with probability at least \(1-O(T^{-1})\), the minimizer
\(\hat\bbeta_{T-1}\) of \eqref{eq: theta hat} is unique and satisfies \(\cI=\cI_*\). On this event,
\([\tilde\bbeta^{\,d}]_{\cI_*^c}=\mathbf 0\). Hence
\[
    [\tilde\bbeta^{\,d}]_{\cI_0^c}=[\bbeta^*]_{\cI_0^c}=\mathbf 0,
    \qquad
    [\tilde\bbeta^{\,d}]_{\cI_{\rm wk}}=\mathbf 0,
\]
and therefore
\[
\|[\tilde\bbeta^{\,d}]_{\cI_{\rm wk}}-[\bbeta^*]_{\cI_{\rm wk}}\|_1
=
\|[\bbeta^*]_{\cI_{\rm wk}}\|_1
=
\eta_T
\]
by Assumption~\ref{asp: min signal}.

It remains to prove the expansion on \(\cI_*\). By the definition of the debiased estimator,
\[
[\tilde\bbeta^{\,d}]_{\cI_*}
=
[\hat\bbeta_{T-1}]_{\cI_*}
-
\Big([\nabla_{\bbeta}^2\ell_{T-1}(\hat\bbeta_{T-1})]_{\cI_*,\cI_*}\Big)^{-1}
[\nabla_{\bbeta}\ell_{T-1}(\hat\bbeta_{T-1})]_{\cI_*}.
\]
Recall $\bar H$ defined in \eqref{eq: int hess}, 
\[
    \bar H
    :=
    \int_0^1
    \nabla_{\bbeta}^2\ell_{T-1}
    \left\{
        \bbeta^*+a(\hat\bbeta_{T-1}-\bbeta^*)
    \right\}
    da ,
\]
and that because both \(\bbeta^*\) and \(\hat\bbeta_{T-1}\) lie in
\(\cB_1(\bbeta^*,3\tau)\), all uniform Hessian bounds apply to \(\bar H\). By the fundamental theorem of calculus,
\[
    \nabla_{\bbeta}\ell_{T-1}(\hat\bbeta_{T-1})
    -
    \nabla_{\bbeta}\ell_{T-1}(\bbeta^*)
    =
    \bar H(\hat\bbeta_{T-1}-\bbeta^*).
\]
Since \([\hat\bbeta_{T-1}]_{\cI_*^c}=\mathbf 0\), we have
\[
    [\hat\bbeta_{T-1}-\bbeta^*]_{\cI_{\rm wk}}
    =
    -[\bbeta^*]_{\cI_{\rm wk}},
    \qquad
    [\hat\bbeta_{T-1}-\bbeta^*]_{\cI_0^c}
    =
    \mathbf 0.
\]
Therefore,
\[
\begin{aligned}
[\nabla_{\bbeta}\ell_{T-1}(\hat\bbeta_{T-1})]_{\cI_*}
&=
[\nabla_{\bbeta}\ell_{T-1}(\bbeta^*)]_{\cI_*}
+
\bar H_{\cI_*,\cI_*}
\{[\hat\bbeta_{T-1}]_{\cI_*}-[\bbeta^*]_{\cI_*}\}\\
&\quad
-
\bar H_{\cI_*,\cI_{\rm wk}}[\bbeta^*]_{\cI_{\rm wk}}.
\end{aligned}
\]
It follows that
\[
\begin{aligned}
&[\tilde\bbeta^{\,d}]_{\cI_*}-[\bbeta^*]_{\cI_*}
+
\frac{1}{T-1}(\bSigma_{\cI_*}^*)^{-1}
[\nabla_{\bbeta}\ell_{T-1}(\bbeta^*)]_{\cI_*}\\
&=
-\underbrace{
\left\{
\Big([\nabla_{\bbeta}^2\ell_{T-1}(\hat\bbeta_{T-1})]_{\cI_*,\cI_*}\Big)^{-1}
-
\frac{1}{T-1}(\bSigma_{\cI_*}^*)^{-1}
\right\}
[\nabla_{\bbeta}\ell_{T-1}(\hat\bbeta_{T-1})]_{\cI_*}
}_{{\rm (I)}}\\
&\quad
-\underbrace{
\left\{
\frac{1}{T-1}(\bSigma_{\cI_*}^*)^{-1}
\bar H_{\cI_*,\cI_*}
-
\Ib_{s_*}
\right\}
\{[\hat\bbeta_{T-1}]_{\cI_*}-[\bbeta^*]_{\cI_*}\}
}_{{\rm (II)}}\\
&\quad
+\underbrace{
\frac{1}{T-1}(\bSigma_{\cI_*}^*)^{-1}
\bar H_{\cI_*,\cI_{\rm wk}}
[\bbeta^*]_{\cI_{\rm wk}}
}_{{\rm (III)}}.
\end{aligned}
\]
We bound the three terms separately.

First, we record a refined Hessian bound at \(\hat\bbeta_{T-1}\) and along the segment between \(\bbeta^*\) and \(\hat\bbeta_{T-1}\). Using the decomposition
\[
\begin{aligned}
\|\nabla_{\bbeta}^2\ell_{T-1}(\bbeta)-(T-1)\bSigma^*\|_{\max}
&\le
\underbrace{\|\nabla_{\bbeta}^2\ell_{T-1}(\bbeta)-\nabla_{\bbeta}^2\ell_{T-1}(\bbeta^*)\|_{\max}}_{\rm I}\\
&\quad+
\underbrace{\left\|\nabla_{\bbeta}^2\ell_{T-1}(\bbeta^*)-\sum_{t'=1}^{T-1}\EE\{\bSigma_{t'}(\hat\bbeta_{t'-1})\mid\hat\bbeta_{t'-1}\}\right\|_{\max}}_{\rm II}\\
&\quad+
\underbrace{\left\|\sum_{t'=1}^{T-1}\EE\{\bSigma_{t'}(\hat\bbeta_{t'-1})\mid\hat\bbeta_{t'-1}\}-(T-1)\bSigma^*\right\|_{\max}}_{\rm III},
\end{aligned}
\]
we improve the first and third terms by using the estimation rates. For the first term, by the Hessian Lipschitz bound shown in \eqref{eq: hess liptz} and \eqref{eq: l1 rate improved},
\[
    {\rm I}
    \lesssim
    T\nu^3\|\hat\bbeta_{T-1}-\bbeta^*\|_1
    \lesssim
    \frac{\nu^4s_*\sqrt{T\log(Tp)}}{\underline\lambda}.
\]
By \eqref{eq: term II bd}, the martingale term satisfies
\[
    {\rm II}
    \lesssim
    \nu^2\{\log(Tp)+\sqrt{T\log(Tp)}\}.
\]
For the third term, the proof of Claim~\ref{claim: cont bd} gives the sharper pointwise form
\[
    \left\|
        \EE\{\bSigma_{t'}(\hat\bbeta_{t'-1})\mid\hat\bbeta_{t'-1}\}
        -
        \bSigma^*
    \right\|_{\max}
    \lesssim
    C_n\nu^2\|\hat\bbeta_{t'-1}-\bbeta^*\|_1+\nu^2/T .
\]
Combining this with the uniform rate in Theorem~\ref{thm:lasso-rates} gives
\[
\begin{aligned}
    {\rm III}
    &\lesssim
    C_n\nu^2
    \sum_{t'=1}^{T-1}
    \|\hat\bbeta_{t'-1}-\bbeta^*\|_1
    +
    \nu^2\\
    &\lesssim
    \frac{C_n\nu^3s\sqrt{T\log(Tp)}}{\underline\lambda}
    +
    \nu^2\log(Tp),
\end{aligned}
\]
where the initial time indices below the range of Theorem~\ref{thm:lasso-rates} are absorbed into the logarithmic term under the stated lower bound on \(T\) and upper bound on \(\tau\). Therefore,
\begin{align}
&\left\|[\nabla_{\bbeta}^2\ell_{T-1}(\hat\bbeta_{T-1})]_{\cI_*,\cI_*}
-(T-1)\bSigma_{\cI_*}^*\right\|_2
\notag\\
&\qquad\lesssim
s_*
\left[
    \nu^2\{\log(Tp)+\sqrt{T\log(Tp)}\}
    +
    \frac{C_n\nu^3s\sqrt{T\log(Tp)}}{\underline\lambda}
    +
    \frac{\nu^4s_*\sqrt{T\log(Tp)}}{\underline\lambda}
\right],
\label{eq: refined hessian support bound}
\end{align}
and the same bound holds with
\([\nabla_{\bbeta}^2\ell_{T-1}(\hat\bbeta_{T-1})]_{\cI_*,\cI_*}\) replaced by
\(\bar H_{\cI_*,\cI_*}\).

By \eqref{eq: refined hessian support bound} and \eqref{eq: min eigen sig support},
\[
\begin{aligned}
&\left\|
\Big([\nabla_{\bbeta}^2\ell_{T-1}(\hat\bbeta_{T-1})]_{\cI_*,\cI_*}\Big)^{-1}
-
\frac{1}{T-1}(\bSigma_{\cI_*}^*)^{-1}
\right\|_2\\
&\qquad\lesssim
\frac{s_*}{T^2\underline\lambda^2}
\left[
    \nu^2\{\log(Tp)+\sqrt{T\log(Tp)}\}
    +
    \frac{C_n\nu^3s\sqrt{T\log(Tp)}}{\underline\lambda}
    +
    \frac{\nu^4s_*\sqrt{T\log(Tp)}}{\underline\lambda}
\right].
\end{aligned}
\]
By the KKT condition on \(\cI_*\),
\[
    \|[\nabla_{\bbeta}\ell_{T-1}(\hat\bbeta_{T-1})]_{\cI_*}\|_2
    \le
    \lambda_{T-1}\sqrt{s_*}
    \lesssim
    \nu\sqrt{s_*T\log(Tp)}.
\]
Thus
\[
\|{\rm (I)}\|_2
\lesssim
\frac{\nu^3s_*^{3/2}\log(Tp)}{T\underline\lambda^2}
\left\{
    1+\frac{C_n\nu s+\nu^2s_*}{\underline\lambda}
\right\}.
\]

Next, using \eqref{eq: refined hessian support bound} again,
\[
\begin{aligned}
&\left\|
\frac{1}{T-1}(\bSigma_{\cI_*}^*)^{-1}
\bar H_{\cI_*,\cI_*}
-\Ib_{s_*}
\right\|_2\\
&\quad \lesssim
\frac{s_*}{T\underline\lambda}
\left[
    \nu^2\{\log(Tp)+\sqrt{T\log(Tp)}\}
    +
    \frac{C_n\nu^3s\sqrt{T\log(Tp)}}{\underline\lambda}
    +
    \frac{\nu^4s_*\sqrt{T\log(Tp)}}{\underline\lambda}
\right].
\end{aligned}
\]
Together with \eqref{eq: support l2-rate},
\[
\|[\hat\bbeta_{T-1}]_{\cI_*}-[\bbeta^*]_{\cI_*}\|_2
\lesssim
\nu\sqrt{\frac{s_*\log(Tp)}{\underline\lambda^2T}},
\]
this yields
\[
\|{\rm (II)}\|_2
\lesssim
\frac{\nu^3s_*^{3/2}\log(Tp)}{T\underline\lambda^2}
\left\{
    1+\frac{C_n\nu s+\nu^2s_*}{\underline\lambda}
\right\}.
\]

Finally, since the Hessian summands are uniformly bounded entrywise by \(C\nu^2\),
\[
\begin{aligned}
\|{\rm (III)}\|_2
&\le
\frac{1}{T-1}
\|(\bSigma_{\cI_*}^*)^{-1}\|_2
\|\bar H_{\cI_*,\cI_{\rm wk}}[\bbeta^*]_{\cI_{\rm wk}}\|_2\\
&\lesssim
\frac{1}{T\underline\lambda}
\sqrt{s_*}\,
T\nu^2
\|[\bbeta^*]_{\cI_{\rm wk}}\|_1\\
&=
\frac{\nu^2\sqrt{s_*}}{\underline\lambda}
\|[\bbeta^*]_{\cI_{\rm wk}}\|_1 .
\end{aligned}
\]
Combining the bounds for \({\rm (I)}\), \({\rm (II)}\), and \({\rm (III)}\) gives the stated bound on \(\|\Rb\|_2\), and completes the proof.
\subsection{Proof of Theorem~\ref{thm: valid p} (Type-I validity)}
\label{sec: proof thm type I}
We begin with three auxiliary lemmas.
\begin{lemma}\label{lm: MTG coupling}
Suppose Assumptions~\ref{asp: abs cont of r}, \ref{asp: cov}--\ref{asp: min signal} hold. Assume that
\[
    C \nu s \sqrt{\frac{\log(Tp)}{\underline\lambda^2 T}}
    \le \tau
    \le
    c\frac{\underline\lambda^2}{C_n \nu^4 s^2},
    \qquad
    T \ge
    C\frac{\nu^8 s^4 \log (Tp)}{\underline\lambda^4},
\]
where \(C>0\) is sufficiently large and \(c>0\) is sufficiently small. If
\[
    \lambda_{T-1}
    =
    C_{\lambda}\nu\sqrt{T\log(Tp)}
\]
with \(C_{\lambda}>0\) sufficiently large, then, on an event with probability at least
\(1-O(T^{-1})\), for any \(\eta>0\) there exists a random vector
\(\bxi\mid \hat\bbeta_0\sim\cN(0,\Ib_{s_*})\) such that
\begin{equation}\label{eq: MTG coupling}
\begin{aligned}
\PP\big(\|\boldsymbol{S}-\bxi\|_2>\eta \mid \hat\bbeta_0 \big)
&\lesssim
\left(
    \frac{s_*^{5/2}\nu^3}
    {T^{1/2}\underline\lambda^{3/2}\eta^3}
\right)^{1/3} +
\left(
    \frac{
        C_n\nu^3s\,s_*^2\sqrt{\log(Tp)}
    }{
        \underline\lambda^2\sqrt T\,\eta^2
    }
\right)^{1/3},
\end{aligned}
\end{equation}
where
\[
\boldsymbol{S}
:=
-
(T-1)^{-1/2}\sum_{t=1}^{T-1}
(\bSigma_{\cI_*}^*)^{-1/2}
\Big\{
[\bv_{t,i_t}]_{\cI_*}
-
\EE_{\bbeta^*,t,\cS_t}\big([\bv_{t,i_t}]_{\cI_*}\big)
\Big\}.
\]
\end{lemma}
\begin{proof}
    See Section~\ref{sec: proof lm MTG coupling}.
\end{proof}
Under Assumption~\ref{asp: assortment const}, with probability at least \(1-O(T^{-1})\) over
\((\bv_T,\br_T)\), the event
\begin{equation}\label{eq: cE r bd T}
    \cE
    :=
    \left\{\max_{j\in[n]}|r_{Tj}|\le 2\bar\mu\right\}
\end{equation}
holds. Conditional on \(\cE\), the next lemma provides uniform derivative bounds together
with a linear expansion of the revenue function.

\begin{lemma}\label{lm: rev decomp}
Under the same conditions as Corollary~\ref{col: error decomp debias lasso}, conditional on
\(\cE\),
\begin{align}
&\max_{\cS\in\bcS^K}\ \sup_{\bbeta\in\R^p}
\big\|\nabla_{\bbeta} R(\cS\mid \bbeta,\bv_T,\br_T)\big\|_{\infty}
\le 4\nu\bar\mu, \label{eq: grad T bd}\\
&\max_{\cS\in\bcS^K}\ \sup_{\bbeta\in\R^p}
\big\|\nabla_{\bbeta}^2 R(\cS\mid \bbeta,\bv_T,\br_T)\big\|_{\max}
\lesssim \nu^2\bar\mu. \label{eq: hes T bd}
\end{align}
Moreover, on the event \(\cE\), with probability at least \(1-O(T^{-1})\) over
\(\cH_{T-1}\mid\hat\bbeta_0\),
\begin{equation}\label{eq: rev decomp}
\begin{aligned}
&\max_{\cS\in\bcS^K}
\Bigg|
\hat R_{T,\cS}
-
R^*_{T,\cS}
+
\frac{1}{\sqrt{T-1}}\,
g_{\cS}^{\top}
(\bSigma_{\cI_*}^{*})^{-1/2}\boldsymbol S
\Bigg|\\
&\qquad\lesssim
\bar\mu
\left[
\frac{\nu^4s_*^2\log(Tp)}{T\underline\lambda^2}
\left\{
    1+\frac{C_n\nu s+\nu^2s_*}{\underline\lambda}
\right\}
+
\frac{\nu^3s_*}{\underline\lambda}
\|[\boldsymbol\beta^*]_{\cI_{\rm wk}}\|_1
\right],
\end{aligned}
\end{equation}
where $ g_{\cS}:= \big[\nabla_{\bbeta}R(\cS\mid \bbeta^*,\bv_T,\br_T)\big]_{\cI_*}$.
\end{lemma}
\begin{proof}
    See Section~\ref{sec: proof lm rev decomp}.
\end{proof}
\begin{lemma}\label{lm: sphere coverage}
Let \(s_*\ge2\), and let \(\zeta_1,\ldots,\zeta_m\) be i.i.d. uniformly distributed on the unit sphere \(S^{s_*-1}\). Fix \(0<\eps\le1\) and any \(\eb\in S^{s_*-1}\). Then, with probability at least
\[
1-\exp\left\{-m\sqrt{\frac{\pi}{8s_*}}\left(\frac{2\eps}{\pi}\right)^{s_*-1}\right\},
\]
there exists \(i\in[m]\) such that \(\|\eb-\zeta_i\|_2\le\eps\).
\end{lemma}
\begin{proof}
    See Section~\ref{sec: proof lm sphere coverage}.
\end{proof}
Note that if \(s_*=1\), then \(S^0=\{-1,1\}\), and the deterministic choice \(\zeta_1=1\) and \(\zeta_2=-1\) covers \(S^0\) exactly. Hence the direction-sampling argument
is trivial in this case. For simplicity, throughout the remainder of the proof we assume \(s_*\ge2\).

We work on the event on which Lemmas~\ref{lm: MTG coupling} and
\ref{lm: rev decomp}, Corollary~\ref{col: unique solution decomp}, and the event
\(\cE\) in \eqref{eq: cE r bd T} all hold. This event has probability at least
\(1-O(T^{-1})\). By Corollary~\ref{col: unique solution decomp}, on this event
\[
    \hat s=s_*,
    \qquad
    \cI=\cI_* .
\]

When \(p_m=1\), the result holds trivially. We therefore focus on the case
\(p_m<1\). We first show that the quantity in \eqref{eq: inf lambda} admits the equivalent representation
\begin{equation}\label{eq: inf lambda 1}
\begin{aligned}
\tilde\cU_T
&:=
\inf\Bigg\{a\ge 0:\ 
\sup_{b \in [0,a]}\max_{1\le i\le m}
\Big[
\max_{\cS\in\bcS_0}\big(\hat R_{T,\cS}+b\,\hat g_{\cS}^\top \widehat\bTheta^{1/2}\zeta_i\big)
-
\Big.\\
&\quad \Big.
\max_{\cS\notin\bcS_0}\big(\hat R_{T,\cS}+b\,\hat g_{\cS}^\top \widehat\bTheta^{1/2}\zeta_i\big)
\Big]
\ge -\kappa
\Bigg\}.
\end{aligned}
\end{equation}
For any $\delta > 0$, by construction, 
\begin{align*}
    \cU_T + \delta &\in \Bigg\{a\ge 0:\ 
\begin{aligned}[t]
&\max_{1\le i\le m}
\Big[
\max_{\cS\in\bcS_0}\big(\hat R_{T,\cS}+a\,\hat g_{\cS}^\top \widehat\bTheta^{1/2}\zeta_i\big)
-
\Big.\\
&\quad \Big.
\max_{\cS\notin\bcS_0}\big(\hat R_{T,\cS}+a\,\hat g_{\cS}^\top \widehat\bTheta^{1/2}\zeta_i\big)
\Big]
\ge -\kappa \Bigg\}
\end{aligned}
\\
    &\subseteq \Bigg\{a\ge 0:\ 
\begin{aligned}[t]
&\sup_{b \in [0,a]} \max_{1\le i\le m}
\Big[
\max_{\cS\in\bcS_0}\big(\hat R_{T,\cS}+b\,\hat g_{\cS}^\top \widehat\bTheta^{1/2}\zeta_i\big)
-
\Big.\\
&\quad \Big.
\max_{\cS\notin\bcS_0}\big(\hat R_{T,\cS}+b\,\hat g_{\cS}^\top \widehat\bTheta^{1/2}\zeta_i\big)
\Big]
\ge -\kappa \Bigg\}
\end{aligned}
,
\end{align*}
which gives $\cU_T + \delta \ge \tilde\cU_T$. 

Similarly, since $\cU_T - \delta < \cU_T$, if there exists a $b \in [0,\cU_T - \delta]$ such that 
$$
\max_{1\le i\le m}
\Big[
\max_{\cS\in\bcS_0}\big(\hat R_{T,\cS}+b\,\hat g_{\cS}^\top \widehat\bTheta^{1/2}\zeta_i\big)
-
\max_{\cS\notin\bcS_0}\big(\hat R_{T,\cS}+b\,\hat g_{\cS}^\top \widehat\bTheta^{1/2}\zeta_i\big)
\Big]
\ge -\kappa,
$$
we will have 
$$
b \in \Bigg\{a\ge 0:\ 
\max_{1\le i\le m}
\Big[
\max_{\cS\in\bcS_0}\big(\hat R_{T,\cS}+a\,\hat g_{\cS}^\top \widehat\bTheta^{1/2}\zeta_i\big)
-
\max_{\cS\notin\bcS_0}\big(\hat R_{T,\cS}+a\,\hat g_{\cS}^\top \widehat\bTheta^{1/2}\zeta_i\big)
\Big]
\ge -\kappa
\Bigg\}
$$
and $\cU_T \le b \le \cU_T - \delta$, which reaches a contradiction. Therefore, we have
$$
\sup_{b \in [0,\cU_T - \delta]} \max_{1\le i\le m}
\Big[
\max_{\cS\in\bcS_0}\big(\hat R_{T,\cS}+b\,\hat g_{\cS}^\top \widehat\bTheta^{1/2}\zeta_i\big)
-
\max_{\cS\notin\bcS_0}\big(\hat R_{T,\cS}+b\,\hat g_{\cS}^\top \widehat\bTheta^{1/2}\zeta_i\big)
\Big]
< -\kappa,
$$
which yields $\cU_T - \delta \le \tilde\cU_T$. Taking $\delta \rightarrow 0$ gives $\tilde\cU_T = \cU_T$. Accordingly, it suffices to establish the theorem with $\tilde\cU_T$ in place of $\cU_T$.

On the event \(\cE\) in \eqref{eq: cE r bd T}, with probability at least
\(1-O(T^{-1})\) over \(\cH_{T-1}\mid\hat\bbeta_0\), we have
\begin{align}
&\max_{\cS\in\bcS^K}
\left\|
\hat\bTheta^{1/2}\hat g_{\cS}
-
\frac{1}{\sqrt{T-1}}\,
(\bSigma_{\cI_*}^{*})^{-1/2}
g_{\cS}
\right\|_2 \notag\\
&\le
\max_{\cS\in\bcS^K}
\left\|
\hat g_{\cS}-g_{\cS}
\right\|_2
\left\|
[\nabla_{\bbeta}^2\ell_{T-1}(\hat\bbeta_{T-1})]_{\cI_*,\cI_*}^{-1/2}
\right\|_2 \notag\\
&\qquad+
\max_{\cS\in\bcS^K}
\|g_{\cS}\|_2
\left\|
[\nabla_{\bbeta}^2\ell_{T-1}(\hat\bbeta_{T-1})]_{\cI_*,\cI_*}^{-1/2}
-
[(T-1)\bSigma_{\cI_*}^*]^{-1/2}
\right\|_2 .
\label{eq: hes grad decomp revised}
\end{align}
The first inequality follows from the triangle inequality and exact effective support recovery by Corollary~\ref{col: unique solution decomp}. We now bound the two terms separately.

For the first term in \eqref{eq: hes grad decomp revised}, by \eqref{eq: hes T bd} and \eqref{eq: l1 rate improved},
\[
\begin{aligned}
\max_{\cS\in\bcS^K}
\|\hat g_{\cS}-g_{\cS}\|_2
&\le
\sqrt{s_*}
\max_{\cS\in\bcS^K}
\|\hat g_{\cS}-g_{\cS}\|_\infty\\
&\le
\sqrt{s_*}
\max_{\cS\in\bcS^K}
\sup_{\bbeta\in\R^p}
\|\nabla_{\bbeta}^2R(\cS\mid\bbeta,\bv_T,\br_T)\|_{\max}
\|\hat\bbeta_{T-1}-\bbeta^*\|_1\\
&\lesssim \frac{\bar\mu\nu^3s_*^{3/2}\sqrt{\log(Tp)}}{\underline\lambda\sqrt T}.
\end{aligned}
\]
By \eqref{eq: min eigen sig support},
\[
\left\|
[\nabla_{\bbeta}^2\ell_{T-1}(\hat\bbeta_{T-1})]_{\cI_*,\cI_*}^{-1/2}
\right\|_2
\lesssim
\frac{1}{\sqrt{T\underline\lambda}}.
\]
Therefore,
\begin{equation}\label{eq: hes grad first term revised}
\max_{\cS\in\bcS^K}
\|\hat g_{\cS}-g_{\cS}\|_2
\left\|
[\nabla_{\bbeta}^2\ell_{T-1}(\hat\bbeta_{T-1})]_{\cI_*,\cI_*}^{-1/2}
\right\|_2
\lesssim
\frac{\bar\mu\nu^3s_*^{3/2}\sqrt{\log(Tp)}}
{T\underline\lambda^{3/2}} .
\end{equation}

We next bound the second term in \eqref{eq: hes grad decomp revised}. Write
\[
    H_T:=
    [\nabla_{\bbeta}^2\ell_{T-1}(\hat\bbeta_{T-1})]_{\cI_*,\cI_*},
    \qquad
    H_T^0:=(T-1)\bSigma_{\cI_*}^* .
\]
By the refined Hessian bound used in Corollary~\ref{col: error decomp debias lasso},
\begin{align}
\|H_T-H_T^0\|_2
&\lesssim
s_*
\left[
    \nu^2\{\log(Tp)+\sqrt{T\log(Tp)}\}
    +
    \frac{C_n\nu^3s\sqrt{T\log(Tp)}}{\underline\lambda}
    +
    \frac{\nu^4s_*\sqrt{T\log(Tp)}}{\underline\lambda}
\right] \notag\\
&\lesssim
s_*\sqrt{T\log(Tp)}
\left[
    \nu^2
    +
    \frac{C_n\nu^3s+\nu^4s_*}{\underline\lambda}
\right],
\label{eq: refined hessian perturb revised}
\end{align}
where the \(\nu^2\log(Tp)\) term is absorbed using
\(C_n\gtrsim\nu\), \(\nu^2s\gtrsim\underline\lambda\), and
\(T\ge C\nu^8s^4\log(Tp)/\underline\lambda^4\) with \(C>0\) sufficiently large. By
\eqref{eq: min eigen sig support}, both \(H_T\) and \(H_T^0\) have minimal eigenvalues bounded below by a constant multiple of \(T\underline\lambda\).

We use the matrix square-root perturbation bound: if
\(\Ab\succeq\mu_1\Ib\) and \(\Bb\succeq\mu_2\Ib\), then
\[
    \|\Ab^{1/2}-\Bb^{1/2}\|_2
    \le
    \frac{\|\Ab-\Bb\|_2}{\sqrt{\mu_1}+\sqrt{\mu_2}},
\]
see \cite{schmitt1992perturbation}. Applying this bound and the identity
\[
    H_T^{-1/2}-(H_T^0)^{-1/2}
    =
    H_T^{-1/2}\big\{(H_T^0)^{1/2}-H_T^{1/2}\big\}(H_T^0)^{-1/2},
\]
we obtain
\begin{align}
\left\|
H_T^{-1/2}-(H_T^0)^{-1/2}
\right\|_2
&\le
\|H_T^{-1/2}\|_2
\|(H_T^0)^{-1/2}\|_2
\left\|H_T^{1/2}-(H_T^0)^{1/2}\right\|_2 \notag\\
&\lesssim
\frac{1}{T\underline\lambda}
\cdot
\frac{\|H_T-H_T^0\|_2}{\sqrt{T\underline\lambda}} \notag\\
&\lesssim
\frac{s_*\sqrt{\log(Tp)}}{T\underline\lambda^{3/2}}
\left[
    \nu^2
    +
    \frac{C_n\nu^3s+\nu^4s_*}{\underline\lambda}
\right].
\label{eq: inv sqrt perturb revised}
\end{align}
Since \eqref{eq: grad T bd} gives
\[
    \max_{\cS\in\bcS^K}\|g_{\cS}\|_2
    \le
    4\nu\bar\mu\sqrt{s_*},
\]
combining \eqref{eq: inv sqrt perturb revised} yields
\begin{align}
&\max_{\cS\in\bcS^K}
\|g_{\cS}\|_2
\left\|
H_T^{-1/2}-(H_T^0)^{-1/2}
\right\|_2 \notag\\
&\quad\lesssim
\frac{\bar\mu\nu^3s_*^{3/2}\sqrt{\log(Tp)}}{T\underline\lambda^{3/2}}
\left\{
    1+\frac{C_n\nu s+\nu^2s_*}{\underline\lambda}
\right\}.
\label{eq: hes grad second term revised}
\end{align}
Combining \eqref{eq: hes grad first term revised} and
\eqref{eq: hes grad second term revised}, we obtain
\begin{equation}\label{eq: hes grad bd}
\begin{aligned}
&\max_{\cS\in\bcS^K}
\left\|
\hat\bTheta^{1/2}\hat g_{\cS}
-
\frac{1}{\sqrt{T-1}}\,
(\bSigma_{\cI_*}^{*})^{-1/2}
g_{\cS}
\right\|_2\\
&\quad\lesssim
\frac{\bar\mu\nu^3s_*^{3/2}\sqrt{\log(Tp)}}{T\underline\lambda^{3/2}}
\left\{
    1+\frac{C_n\nu s+\nu^2s_*}{\underline\lambda}
\right\}.
\end{aligned}
\end{equation}

The localized version improves the second term by using
\[
    \max_{\cS\in\bar\bcS_0(\bv_T,\br_T)\cup\bar\bcS_1(\bv_T,\br_T)}
    \|g_{\cS}\|_2
    =
    \sigma_{\bv_T,\br_T}
\]
instead of \(4\nu\bar\mu\sqrt{s_*}\). Therefore,
\begin{equation}\label{eq: hes grad bd subset}
\begin{aligned}
&\max_{\cS\in\bar\bcS_0(\bv_T,\br_T)\cup\bar\bcS_1(\bv_T,\br_T)}
\left\|
\hat\bTheta^{1/2}\hat g_{\cS}
-
\frac{1}{\sqrt{T-1}}\,
(\bSigma_{\cI_*}^{*})^{-1/2}
g_{\cS}
\right\|_2\\
&\quad\lesssim
\frac{\bar\mu\nu^3s_*^{3/2}\sqrt{\log(Tp)}}{T\underline\lambda^{3/2}}
+
\frac{\sigma_{\bv_T,\br_T}\nu^2s_*\sqrt{\log(Tp)}}{T\underline\lambda^{3/2}}
\left\{
    1+\frac{C_n\nu s+\nu^2s_*}{\underline\lambda}
\right\}.
\end{aligned}
\end{equation}
Assume
\begin{equation}\label{eq: cond coeff dist}
     \frac{\nu^2s_*\sqrt{\log(Tp)}}{\sqrt T\,\underline\lambda}
    \left\{
        1+\frac{C_n\nu s+\nu^2s_*}{\underline\lambda}
    \right\}
    +
    \frac{\bar\mu\nu^3s_*^{3/2}\sqrt{\log(Tp)}}
    {\sigma_{\bv_T,\br_T}\sqrt T\,\underline\lambda}
    \le
    c,
\end{equation}
for small enough $c > 0$. Together with
\(\kappa\asymp\sigma_{\bv_T,\br_T}\sqrt{s_*/(T\underline\lambda)}\),
\eqref{eq: hes grad bd subset} implies
\begin{equation}\label{eq: kappa cond 1}
\sqrt{s_*}
\max_{\cS\in\bar\bcS_0(\bv_T,\br_T)\cup\bar\bcS_1(\bv_T,\br_T)}
\left\|
\hat\bTheta^{1/2}\hat g_{\cS}
-
\frac{1}{\sqrt{T-1}}\,
(\bSigma_{\cI_*}^{*})^{-1/2}
g_{\cS}
\right\|_2
\le
c_1\kappa
\end{equation}
for a sufficiently small constant \(c_1>0\). Moreover, \eqref{eq: hes grad bd} gives
\begin{equation}\label{eq: theta g global bd revised}
\max_{\cS\in\bcS^K}
\|\widehat\bTheta^{1/2}\hat g_{\cS}\|_2
\lesssim
\nu\bar\mu\sqrt{\frac{s_*}{T\underline\lambda}},
\end{equation}
because \eqref{eq: cond coeff dist} implies that the perturbation term in \eqref{eq: hes grad bd} is dominated by the leading scale
\(\nu\bar\mu\sqrt{s_*/(T\underline\lambda)}\).

Applying Lemma~\ref{lm: MTG coupling} with
\begin{equation}\label{eq: eta value}
    \eta
    =
    \sqrt{\log(Tp)}
    \left\{
    \left(
        \frac{s_*^{5/2}\nu^3}
        {T^{1/2}\underline\lambda^{3/2}}
    \right)^{1/3}
    \vee
    \left(
        \frac{C_n\nu^3s\,s_*^2\sqrt{\log(Tp)}}
        {\underline\lambda^2\sqrt T}
    \right)^{1/2}
    \right\},
\end{equation}
there exists a random vector
\(\bxi\mid\hat\bbeta_0\sim\cN(0,\Ib_{s_*})\) such that
\begin{equation}\label{eq: S xi dis}
\PP\big(\|\boldsymbol S-\bxi\|_2>\eta\mid\hat\bbeta_0\big)
=o(1).
\end{equation}
Furthermore, assume that
\begin{equation}\label{eq: cond eta}
     \sqrt{\log(Tp)}
    \left\{
    \left(
        \frac{s_*^{5/2}\nu^3}
        {T^{1/2}\underline\lambda^{3/2}}
    \right)^{1/3}
    \vee
    \left(
        \frac{C_n\nu^3s\,s_*^2\sqrt{\log(Tp)}}
        {\underline\lambda^2\sqrt T}
    \right)^{1/2}
    \right\}
    \le
    c\sqrt{s_*}.
\end{equation}
Then, for \(\eta\) defined in \eqref{eq: eta value},
\begin{equation}\label{eq: eta kappa cond revised}
    \frac{\sigma_{\bv_T,\br_T}}{\sqrt{T\underline\lambda}}\,\eta
    \le
    c_2\kappa
\end{equation}
for a sufficiently small constant \(c_2>0\), after choosing \(c>0\) sufficiently small.

In addition, Corollary~\ref{col: unique solution decomp} implies that
\(\hat s=s_*\) with probability at least \(1-O(T^{-1})\). Consequently, for any fixed
\(\alpha\in(0,1)\), if
\[
m
\ge
\log(2/\alpha)\sqrt{\frac{8\hat s}{\pi}}
\left(\frac{\pi}{2\epsilon}\right)^{\hat s-1},
\]
then, on the event \(\hat s=s_*\),
\[
\delta_m
:=
\exp\left\{
    -m\sqrt{\frac{\pi}{8s_*}}
    \left(\frac{2\epsilon}{\pi}\right)^{s_*-1}
\right\}
\le
\alpha/2.
\]

Hence, for any \(\delta\in(0,1-\alpha)\), given \((\bv_T,\br_T)\) such that \(\cE\) holds, and for any \(\bbeta^*\in\cM_0(\bv_T,\br_T)\), we have
\begin{equation}\label{eq: p value ini bd}
\begin{aligned}
&\PP_{\bbeta^*}(p_m\le\alpha\mid\bv_T,\br_T,\hat\bbeta_0)\\
&\le
\PP_{\bbeta^*}
\left(
\tilde\cU_T^2
\ge
F^{-1}_{\chi_{s_*}^2}(1-\alpha+\delta_m)
\mid\bv_T,\br_T,\hat\bbeta_0
\right)
+O(T^{-1})\\
&\le
\PP_{\bbeta^*}\Bigg\{
\sup_{b\in\left[0,\sqrt{F^{-1}_{\chi_{s_*}^2}(1-\alpha+\delta_m-\delta)}\right]}
\max_{1\le i\le m}
\Big[
\max_{\cS\in\bcS_0}
\Big(\hat R_{T,\cS}+b\,\hat g_{\cS}^{\top}\widehat\bTheta^{1/2}\zeta_i\Big)\\
&\hspace{8em}
-
\max_{\cS\notin\bcS_0}
\Big(\hat R_{T,\cS}+b\,\hat g_{\cS}^{\top}\widehat\bTheta^{1/2}\zeta_i\Big)
\Big]
<-\kappa
\mid\bv_T,\br_T,\hat\bbeta_0
\Bigg\}
+O(T^{-1}).
\end{aligned}
\end{equation}

Now suppose that
\[
\|\bxi\|_2^2
\le
F^{-1}_{\chi_{s_*}^2}(1-\alpha+\delta_m-\delta)
\le
F^{-1}_{\chi_{s_*}^2}(1-\alpha/2-\delta)
=
O(s_*),
\]
that \(\|\bxi/\|\bxi\|_2-\zeta_{i^*}\|_2\le\epsilon\) for some \(i^*\in[m]\), and that
\eqref{eq: hes grad bd} holds. We have
\begin{align*}
&\sup_{b\in\left[0,\sqrt{F^{-1}_{\chi_{s_*}^2}(1-\alpha+\delta_m-\delta)}\right]}
\max_{1\le i\le m}
\Big[
\max_{\cS\in\bcS_0}
\Big(\hat R_{T,\cS}+b\,\hat g_{\cS}^{\top}\widehat\bTheta^{1/2}\zeta_i\Big)
-
\max_{\cS\notin\bcS_0}
\Big(\hat R_{T,\cS}+b\,\hat g_{\cS}^{\top}\widehat\bTheta^{1/2}\zeta_i\Big)
\Big]\\
&\quad\ge
\max_{\cS\in\bcS_0}
\Big(\hat R_{T,\cS}+\|\bxi\|_2\,\hat g_{\cS}^{\top}\widehat\bTheta^{1/2}\zeta_{i^*}\Big)
-
\max_{\cS\notin\bcS_0}
\Big(\hat R_{T,\cS}+\|\bxi\|_2\,\hat g_{\cS}^{\top}\widehat\bTheta^{1/2}\zeta_{i^*}\Big).
\end{align*}

On the event \eqref{eq: rev decomp}, using \eqref{eq: theta g global bd revised},
\(\|\bxi\|_2=O(\sqrt{s_*})\), and the bound
\begin{equation}\label{eq: cond beta weak tail}
     \bar\mu
    \left[
    \frac{\nu^4s_*^2\log(Tp)}{T\underline\lambda^2}
    \left\{
        1+\frac{C_n\nu s+\nu^2s_*}{\underline\lambda}
    \right\}
    +
    \frac{\nu^3s_*}{\underline\lambda}
    \|[\boldsymbol\beta^*]_{\cI_{\rm wk}}\|_1
    \right]
    \le
    c\,\sigma_{\bv_T,\br_T}
    \sqrt{\frac{s_*}{T\underline\lambda}} ,
\end{equation}
we obtain
\begin{align*}
&\max_{\cS\in\bcS^K}
\left|
\hat R_{T,\cS}
+
\|\bxi\|_2\,\hat g_{\cS}^{\top}\widehat\bTheta^{1/2}\zeta_{i^*}
-
R^*_{T,\cS}
\right|\\
&\quad\lesssim
\|\bxi\|_2
\max_{\cS\in\bcS^K}
\left(
\|\widehat\bTheta^{1/2}\hat g_{\cS}\|_2
+
\frac{1}{\sqrt{T-1}}\|(\bSigma_{\cI_*}^*)^{-1/2}g_{\cS}\|_2
\right)\\
&\qquad+
\max_{\cS\in\bcS^K}
\Bigg|
\hat R_{T,\cS}
-
R^*_{T,\cS}
+
\frac{1}{\sqrt{T-1}}\,
g_{\cS}^{\top}
(\bSigma_{\cI_*}^{*})^{-1/2}\boldsymbol S
\Bigg|\\
&\quad\lesssim
\frac{\nu\bar\mu s_*}{\sqrt{T\underline\lambda}}
+
\bar\mu
\left[
\frac{\nu^4s_*^2\log(Tp)}{T\underline\lambda^2}
\left\{
    1+\frac{C_n\nu s+\nu^2s_*}{\underline\lambda}
\right\}
+
\frac{\nu^3s_*}{\underline\lambda}
\|[\boldsymbol\beta^*]_{\cI_{\rm wk}}\|_1
\right]\\
&\quad\le
9\nu s_*\bar\mu\sqrt{\frac{\log T}{T\underline\lambda}},
\end{align*}
where the last inequality follows from \(\log T\ge1\), \(\epsilon\in(0,1)\),
\(\sigma_{\bv_T,\br_T}\le 4\nu\bar\mu\sqrt{s_*}\), and \eqref{eq: cond beta weak tail}. Therefore,
\begin{align*}
&\max_{\cS\in\bcS_0\setminus\bar\bcS_0(\bv_T,\br_T)}
\left\{
\hat R_{T,\cS}
+
\|\bxi\|_2\hat g_{\cS}^{\top}\widehat\bTheta^{1/2}\zeta_{i^*}
\right\}\\
&\quad\le
\max_{\cS\in\bcS_0\setminus\bar\bcS_0(\bv_T,\br_T)}
R^*_{T,\cS}
+
9\nu s_*\bar\mu\sqrt{\frac{\log T}{T\underline\lambda}}\\
&\quad\le
\max_{\cS\in\bar\bcS_0(\bv_T,\br_T)}
R^*_{T,\cS}
-
9\nu s_*\bar\mu\sqrt{\frac{\log T}{T\underline\lambda}}\\
&\quad\le
\max_{\cS\in\bar\bcS_0(\bv_T,\br_T)}
\left\{
\hat R_{T,\cS}
+
\|\bxi\|_2\hat g_{\cS}^{\top}\widehat\bTheta^{1/2}\zeta_{i^*}
\right\}.
\end{align*}
This gives
\begin{equation}\label{eq: max in bar S_0}
\max_{\cS\in\bcS_0}
\left\{
\hat R_{T,\cS}
+
\|\bxi\|_2\hat g_{\cS}^{\top}\widehat\bTheta^{1/2}\zeta_{i^*}
\right\}
=
\max_{\cS\in\bar\bcS_0(\bv_T,\br_T)}
\left\{
\hat R_{T,\cS}
+
\|\bxi\|_2\hat g_{\cS}^{\top}\widehat\bTheta^{1/2}\zeta_{i^*}
\right\}.
\end{equation}
Similarly,
\begin{equation}\label{eq: max in bar S_1}
\max_{\cS\notin\bcS_0}
\left\{
\hat R_{T,\cS}
+
\|\bxi\|_2\hat g_{\cS}^{\top}\widehat\bTheta^{1/2}\zeta_{i^*}
\right\}
=
\max_{\cS\in\bar\bcS_1(\bv_T,\br_T)}
\left\{
\hat R_{T,\cS}
+
\|\bxi\|_2\hat g_{\cS}^{\top}\widehat\bTheta^{1/2}\zeta_{i^*}
\right\}.
\end{equation}

When events \eqref{eq: max in bar S_0} and \eqref{eq: max in bar S_1} both hold, using \eqref{eq: kappa cond 1}, we further have
\begin{align*}
&\max_{\cS\in\bcS_0}
\Big(\hat R_{T,\cS}+\|\bxi\|_2\hat g_{\cS}^{\top}\widehat\bTheta^{1/2}\zeta_{i^*}\Big)
-
\max_{\cS\notin\bcS_0}
\Big(\hat R_{T,\cS}+\|\bxi\|_2\hat g_{\cS}^{\top}\widehat\bTheta^{1/2}\zeta_{i^*}\Big)\\
&=
\max_{\cS\in\bar\bcS_0(\bv_T,\br_T)}
\Big(\hat R_{T,\cS}+\|\bxi\|_2\hat g_{\cS}^{\top}\widehat\bTheta^{1/2}\zeta_{i^*}\Big)
-
\max_{\cS\in\bar\bcS_1(\bv_T,\br_T)}
\Big(\hat R_{T,\cS}+\|\bxi\|_2\hat g_{\cS}^{\top}\widehat\bTheta^{1/2}\zeta_{i^*}\Big)\\
&\ge
\max_{\cS\in\bar\bcS_0(\bv_T,\br_T)}
\Big(\hat R_{T,\cS}+\hat g_{\cS}^{\top}\widehat\bTheta^{1/2}\bxi\Big)
-
\max_{\cS\in\bar\bcS_1(\bv_T,\br_T)}
\Big(\hat R_{T,\cS}+\hat g_{\cS}^{\top}\widehat\bTheta^{1/2}\bxi\Big)\\
&\quad-
O\left(
\sqrt{s_*}\epsilon
\max_{\cS\in\bar\bcS_0(\bv_T,\br_T)\cup\bar\bcS_1(\bv_T,\br_T)}
\|\hat\bTheta^{1/2}\hat g_{\cS}\|_2
\right)\\
&\ge
\max_{\cS\in\bar\bcS_0(\bv_T,\br_T)}
\Big(\hat R_{T,\cS}
+\frac{1}{\sqrt{T-1}}g_{\cS}^{\top}(\bSigma_{\cI_*}^{*})^{-1/2}\bxi\Big)\\
&\quad-
\max_{\cS\in\bar\bcS_1(\bv_T,\br_T)}
\Big(\hat R_{T,\cS}
+\frac{1}{\sqrt{T-1}}g_{\cS}^{\top}(\bSigma_{\cI_*}^{*})^{-1/2}\bxi\Big)
-\frac{\kappa}{2}.
\end{align*}
Furthermore, by Lemmas~\ref{lm: MTG coupling} and \ref{lm: rev decomp}, on the events \eqref{eq: rev decomp} and \(\|\boldsymbol S-\bxi\|_2\le\eta\), using \eqref{eq: eta kappa cond revised} and \eqref{eq: cond beta weak tail},
\begin{align*}
&\max_{\cS\in\bar\bcS_0(\bv_T,\br_T)}
\Big(\hat R_{T,\cS}
+\frac{1}{\sqrt{T-1}}g_{\cS}^{\top}(\bSigma_{\cI_*}^{*})^{-1/2}\bxi\Big)\\
&\quad-
\max_{\cS\in\bar\bcS_1(\bv_T,\br_T)}
\Big(\hat R_{T,\cS}
+\frac{1}{\sqrt{T-1}}g_{\cS}^{\top}(\bSigma_{\cI_*}^{*})^{-1/2}\bxi\Big)
-\frac{\kappa}{2}\\
&\ge
\max_{\cS\in\bar\bcS_0(\bv_T,\br_T)}
\left(
R_{T,\cS}^*
-
\frac{1}{\sqrt{T-1}}g_{\cS}^{\top}(\bSigma_{\cI_*}^{*})^{-1/2}(\bxi-\boldsymbol S)
\right)\\
&\quad-
\max_{\cS\in\bar\bcS_1(\bv_T,\br_T)}
\left(
R_{T,\cS}^*
-
\frac{1}{\sqrt{T-1}}g_{\cS}^{\top}(\bSigma_{\cI_*}^{*})^{-1/2}(\bxi-\boldsymbol S)
\right)\\
&\quad-
\frac{\kappa}{2}
-
2\max_{\cS\in\bcS^K}
\Bigg|
\hat R_{T,\cS}
-
R^*_{T,\cS}
+
\frac{1}{\sqrt{T-1}}g_{\cS}^{\top}
(\bSigma_{\cI_*}^{*})^{-1/2}\boldsymbol S
\Bigg|\\
&\ge
\max_{\cS\in\bar\bcS_0(\bv_T,\br_T)}R^*_{T,\cS}
-
\max_{\cS\in\bar\bcS_1(\bv_T,\br_T)}R^*_{T,\cS}
-
\kappa.
\end{align*}
Since \(\bbeta^*\in\cM_0(\bv_T,\br_T)\), the last display is at least \(-\kappa\). Substituting the preceding bounds into \eqref{eq: p value ini bd}, we obtain
\begin{align*}
\PP_{\bbeta^*}(p_m\le\alpha\mid\bv_T,\br_T,\hat\bbeta_0)
&\le
O(T^{-1})
+
\PP\left(
\|\bxi\|_2^2>
F^{-1}_{\chi_{s_*}^2}(1-\alpha+\delta_m-\delta)
\mid\hat\bbeta_0
\right)\\
&\quad+
\PP\left(
\|\bxi/\|\bxi\|_2-\zeta_i\|_2>\epsilon,\ \forall i\in[m]
\mid\hat\bbeta_0
\right)\\
&\quad+
\PP(\|\boldsymbol S-\bxi\|_2>\eta\mid\hat\bbeta_0)\\
&\quad+
\PP\{\text{event \eqref{eq: rev decomp} fails}\mid\bv_T,\br_T,\hat\bbeta_0\}\\
&\quad+
\PP\{\text{event \eqref{eq: hes grad bd} fails}\mid\bv_T,\br_T,\hat\bbeta_0\}\\
&\quad+
\PP\{\text{event \eqref{eq: max in bar S_0} fails}\mid\bv_T,\br_T,\hat\bbeta_0\}\\
&\quad+
\PP\{\text{event \eqref{eq: max in bar S_1} fails}\mid\bv_T,\br_T,\hat\bbeta_0\}\\
&\le
\alpha-\delta_m+\delta+\delta_m+o(1)+O(T^{-1})\\
&=
\alpha+o(1)+\delta,
\end{align*}
where we apply Lemma~\ref{lm: sphere coverage} and the event \(\hat s=s_*\) to obtain
\[
\PP\big(\|\bxi/\|\bxi\|_2-\zeta_i\|_2>\epsilon,\ \forall i\in[m]\mid\hat\bbeta_0\big)
\le
\exp\left\{
    -m\sqrt{\frac{\pi}{8s_*}}
    \left(\frac{2\epsilon}{\pi}\right)^{s_*-1}
\right\}
=
\delta_m .
\]
Letting \(\delta\downarrow0\) gives the desired claim.

Finally, under the condition that 
$$
T
\ge
C\frac{C_n^2 \nu^6 s^2 s_*^2 \log^2 (Tp)}{\underline\lambda^4}\left[
\log (Tp)
\vee
\frac{\bar\mu^2 \nu^{4}s_*}
{\sigma_{\bv_T,\br_T}^2\underline\lambda}
\right], 
$$
and 
$$
 \|[\boldsymbol\beta^*]_{\cI_{\rm wk}}\|_1
    \le
    c\,
    \frac{
        \sigma_{\bv_T,\br_T}\sqrt{\underline\lambda}
    }{
        \bar\mu\nu^3\sqrt{s_*T}
    },
$$
for sufficiently large $C >0$ and sufficiently small $c > 0$, along with the fact that $\underline\lambda \le \nu^2$, we have that the conditions \eqref{eq: cond coeff dist}, \eqref{eq: cond eta} and \eqref{eq: cond beta weak tail} hold. This completes the proof.
\subsection{Proof of Theorem~\ref{thm: valid powerful test} (Validity and power)}
\label{sec: proof thm valid powerful test}

The Type~I error bound in \eqref{eq: type I asymp} follows directly from
Theorem~\ref{thm: valid p}. It remains to prove \eqref{eq: power}.

Under the specification
\[
m
\ge
\log\left(\frac{2}{\alpha}\right)
\sqrt{\frac{8\hat s}{\pi}}
\left(\frac{\pi}{2\eps}\right)^{\hat s-1},
\]
we have \(\delta_m\le\alpha/2\). Besides, by
Corollary~\ref{col: unique solution decomp}, we have
\(\hat s=s_*\) with probability at least \(1-O(T^{-1})\). By
Assumption~\ref{asp: assortment const}, the event \(\cE\) in
\eqref{eq: cE r bd T} holds with probability at least \(1-O(T^{-1})\) over
the randomness in \((\bv_T,\br_T)\).

On the event \(\cE\), for any fixed \(\alpha\in(0,1)\) and any
\(\bbeta^*\in\cM_1(\bv_T,\br_T)\), using the equivalent representation
\eqref{eq: inf lambda 1}, we have
\begin{align*}
\PP_{\bbeta^*}(p_m>\alpha\mid\bv_T,\br_T,\hat\bbeta_0)
&\le
\PP_{\bbeta^*}\left(
F^{-1}_{\chi^2_{s_*}}(1-\alpha+\delta_m)>\tilde\cU_T^2
\,\middle|\,
\bv_T,\br_T,\hat\bbeta_0
\right)
+
O(T^{-1})\\
&\le
\PP_{\bbeta^*}\left(
F^{-1}_{\chi^2_{s_*}}(1-\alpha/2)>\tilde\cU_T^2
\,\middle|\,
\bv_T,\br_T,\hat\bbeta_0
\right)
+
O(T^{-1})\\
&\le
\PP_{\bbeta^*}\Bigg(
\sup_{b\in[0,\tilde a]}
\max_{1\le i\le m}
\Big[
\max_{\cS\in\bcS_0}
\big(\hat R_{T,\cS}+b\,\hat g_{\cS}^{\top}\widehat\bTheta^{1/2}\zeta_i\big)\\
&\hspace{8em}
-
\max_{\cS\notin\bcS_0}
\big(\hat R_{T,\cS}+b\,\hat g_{\cS}^{\top}\widehat\bTheta^{1/2}\zeta_i\big)
\Big]
\ge -\kappa
\ \Big|\ \bv_T,\br_T,\hat\bbeta_0
\Bigg)\\
& \quad +
O(T^{-1}),
\end{align*}
where
\[
\tilde a
:=
\sqrt{F^{-1}_{\chi^2_{s_*}}(1-\alpha/2)}
=
O(\sqrt{s_*}) .
\]

On the event \(\cE\), by \eqref{eq: theta g global bd revised}, with
probability at least \(1-O(T^{-1})\) over \(\cH_{T-1}\mid\hat\bbeta_0\), we
have
\[
\max_{\cS\in\bcS^K}
\|\widehat\bTheta^{1/2}\hat g_{\cS}\|_2
\lesssim
\nu\bar\mu\sqrt{\frac{s_*}{T\underline\lambda}} .
\]
Consequently,
\begin{align*}
\sup_{b\in[0,\tilde a]}
\max_{i\in[m]}\max_{\cS\in\bcS^K}
\left|
b\,\hat g_{\cS}^{\top}\widehat\bTheta^{1/2}\zeta_i
\right|
&\le
\tilde a
\max_{\cS\in\bcS^K}
\|\widehat\bTheta^{1/2}\hat g_{\cS}\|_2 \lesssim
\nu\bar\mu
\frac{s_*}{\sqrt{T\underline\lambda}} \le
\nu s_*\bar\mu\sqrt{\frac{\log T}{T\underline\lambda}}.
\end{align*}

Taking \(\eta\) to be the value in \eqref{eq: eta value}, when
\eqref{eq: MTG coupling} and \eqref{eq: rev decomp} hold, there exists
\(\bxi\mid\hat\bbeta_0\sim\cN(0,\Ib_{s_*})\) such that
\[
    \|\boldsymbol S-\bxi\|_2\le\eta
\]
with conditional probability \(1-o(1)\). Moreover,
\begin{equation}\label{eq: bxi bd}
\begin{aligned}
\PP\big(\|\bxi\|_2\ge 2\sqrt{s_*\log T}\mid\hat\bbeta_0\big)
&\le
\PP\big(\|\bxi\|_{\infty}\ge 2\sqrt{\log T}\mid\hat\bbeta_0\big)\\
&\le
s_*\,\PP\big(|\xi_1|\ge 2\sqrt{\log T}\mid\hat\bbeta_0\big)
\le
\frac{2s_*}{T^2}
\le
\frac{2}{T}.
\end{aligned}
\end{equation}
On the event \(\|\bxi\|_2\le 2\sqrt{s_*\log T}\), we have
\begin{align*}
&\max_{\cS\in\bcS^K}
\frac{1}{\sqrt{T-1}}\,
\left|
g_{\cS}^{\top}
(\bSigma_{\cI_*}^{*})^{-1/2}\bxi
\right|\\
&\quad\le
\max_{\cS\in\bcS^K}
\frac{\|g_{\cS}\|_2\|\bxi\|_2}
{\sqrt{(T-1)\underline\lambda}}
\le
\frac{4\sqrt{s_*}\nu\bar\mu\cdot2\sqrt{s_*\log T}}
{\sqrt{(T-1)\underline\lambda}}
\le
8\nu s_*\bar\mu
\sqrt{\frac{\log T}{(T-1)\underline\lambda}} .
\end{align*}
Together with \eqref{eq: rev decomp} and \(\|\boldsymbol S-\bxi\|_2\le\eta\),
and using \eqref{eq: eta kappa cond revised} and \eqref{eq: cond beta weak tail} implied by the conditions on $T$ and $\eta_T$, this implies
\[
\max_{b\in[0,\tilde a]}
\max_{i\in[m]}
\max_{\cS\in\bcS^K}
\left|
\hat R_{T,\cS}
+
b\,\hat g_{\cS}^{\top}\widehat\bTheta^{1/2}\zeta_i
-
R_{T,\cS}^*
\right|
\le
9\nu s_*\bar\mu
\sqrt{\frac{\log T}{T\underline\lambda}}
\]
with probability at least \(1-O(T^{-1})-o(1)\) over
\(\cH_{T-1}\mid\hat\bbeta_0\).

Therefore, with probability at least \(1-O(T^{-1})-o(1)\) over
\(\cH_{T-1}\mid\hat\bbeta_0\), for all \(i\in[m]\) and all
\(b\in[0,\tilde a]\), we have
\begin{align*}
&\max_{\cS\in\bcS_0\setminus\bar\bcS_0(\bv_T,\br_T)}
\left\{
\hat R_{T,\cS}
+
b\,\hat g_{\cS}^{\top}\widehat\bTheta^{1/2}\zeta_i
\right\}\\
&\quad\le
\max_{\cS\in\bcS_0\setminus\bar\bcS_0(\bv_T,\br_T)}
R_{T,\cS}^*
+
9\nu s_* \bar\mu
\sqrt{\frac{\log T}{T\underline\lambda}}\\
&\quad\le
\max_{\cS\in\bar\bcS_0(\bv_T,\br_T)}
R_{T,\cS}^*
-
18\nu s_* \bar\mu
\sqrt{\frac{\log T}{T\underline\lambda}}
+
9\nu s_*\bar\mu
\sqrt{\frac{\log T}{T\underline\lambda}}\\
&\quad=
\max_{\cS\in\bar\bcS_0(\bv_T,\br_T)}
R_{T,\cS}^*
-
9\nu s_* \bar\mu
\sqrt{\frac{\log T}{T\underline\lambda}}\\
&\quad\le
\max_{\cS\in\bar\bcS_0(\bv_T,\br_T)}
\left\{
\hat R_{T,\cS}
+
b\,\hat g_{\cS}^{\top}\widehat\bTheta^{1/2}\zeta_i
\right\}.
\end{align*}
Hence
\[
\max_{\cS\in\bcS_0}
\left\{
\hat R_{T,\cS}
+
b\,\hat g_{\cS}^{\top}\widehat\bTheta^{1/2}\zeta_i
\right\}
=
\max_{\cS\in\bar\bcS_0(\bv_T,\br_T)}
\left\{
\hat R_{T,\cS}
+
b\,\hat g_{\cS}^{\top}\widehat\bTheta^{1/2}\zeta_i
\right\},
\quad
\forall i\in[m],\ b\in[0,\tilde a].
\]
Similarly,
\[
\max_{\cS\notin\bcS_0}
\left\{
\hat R_{T,\cS}
+
b\,\hat g_{\cS}^{\top}\widehat\bTheta^{1/2}\zeta_i
\right\}
=
\max_{\cS\in\bar\bcS_1(\bv_T,\br_T)}
\left\{
\hat R_{T,\cS}
+
b\,\hat g_{\cS}^{\top}\widehat\bTheta^{1/2}\zeta_i
\right\},
\quad
\forall i\in[m],\ b\in[0,\tilde a].
\]

Thus, on the event \(\cE\),
\begin{align*}
\PP_{\bbeta^*}(p_m>\alpha\mid\bv_T,\br_T,\hat\bbeta_0)
&\le
O(T^{-1})+o(1)\\
&\quad+
\PP_{\bbeta^*}\Bigg(
\sup_{b\in[0,\tilde a]}
\max_{1\le i\le m}
\Big[
\max_{\cS\in\bar\bcS_0(\bv_T,\br_T)}
\big(\hat R_{T,\cS}+b\,\hat g_{\cS}^{\top}\widehat\bTheta^{1/2}\zeta_i\big)\\
&\hspace{6em}
-
\max_{\cS\in\bar\bcS_1(\bv_T,\br_T)}
\big(\hat R_{T,\cS}+b\,\hat g_{\cS}^{\top}\widehat\bTheta^{1/2}\zeta_i\big)
\Big]
\ge -\kappa
\ \Big|\ \bv_T,\br_T,\hat\bbeta_0
\Bigg).
\end{align*}

Next, by \eqref{eq: hes grad bd subset}, \eqref{eq: kappa cond 1}, and
\(\tilde a=O(\sqrt{s_*})\), with probability at least \(1-O(T^{-1})\) over
\(\cH_{T-1}\mid\hat\bbeta_0\),
\begin{align*}
&\sup_{b\in[0,\tilde a]}
\max_{i\in[m]}
\max_{\cS\in\bar\bcS_0(\bv_T,\br_T)\cup\bar\bcS_1(\bv_T,\br_T)}
\left|
b\,\hat g_{\cS}^{\top}\widehat\bTheta^{1/2}\zeta_i
\right|\\
&\quad\lesssim
\sqrt{s_*}
\max_{\cS\in\bar\bcS_0(\bv_T,\br_T)\cup\bar\bcS_1(\bv_T,\br_T)}
\|\widehat\bTheta^{1/2}\hat g_{\cS}\|_2\\
&\quad\lesssim
\sigma_{\bv_T,\br_T}
\sqrt{\frac{s_*}{T\underline\lambda}} .
\end{align*}
Moreover, on the event \eqref{eq: rev decomp} and
\(\|\boldsymbol S-\bxi\|_2\le\eta\),
\begin{align*}
&\max_{\cS\in\bar\bcS_0(\bv_T,\br_T)\cup\bar\bcS_1(\bv_T,\br_T)}
\left|
\hat R_{T,\cS}
-
R_{T,\cS}^*
+
\frac{1}{\sqrt{T-1}}
g_{\cS}^{\top}
(\bSigma_{\cI_*}^{*})^{-1/2}\bxi
\right|\\
&\quad\lesssim
\bar\mu
\left[
\frac{\nu^4s_*^2\log(Tp)}{T\underline\lambda^2}
\left\{
1+\frac{C_n\nu s+\nu^2s_*}{\underline\lambda}
\right\}
+
\frac{\nu^3s_*}{\underline\lambda}\eta_T
\right]
+
\frac{\sigma_{\bv_T,\br_T}}{\sqrt{T\underline\lambda}}\eta\\
&\quad\le
c\sigma_{\bv_T,\br_T}
\sqrt{\frac{s_*}{T\underline\lambda}},
\end{align*}
for a sufficiently small constant $c >0$, where the last inequality follows from the conditions of Theorem~\ref{thm: valid p}.

For any \(M>0\), since \(\bxi\mid\hat\bbeta_0\sim\cN(0,\Ib_{s_*})\),
\[
    \PP(\|\bxi\|_2>M\mid\hat\bbeta_0)
    \le
    \exp\left((\log5)s_*-\frac{M^2}{8}\right).
\]
Therefore, under the signal strength condition in \eqref{eq: signal strength},
\begin{align*}
&\PP_{\bbeta^*}\left(
\max_{\cS\in\bar\bcS_0(\bv_T,\br_T)\cup\bar\bcS_1(\bv_T,\br_T)}
\frac{
g_{\cS}^{\top}
(\bSigma_{\cI_*}^{*})^{-1/2}\bxi
}{\sqrt{T-1}}
\ge
\frac{1}{3}
\left|
\max_{\cS\in\bcS_0}R_{T,\cS}^*
-
\max_{\cS\notin\bcS_0}R_{T,\cS}^*
\right|
\ \Bigg|\ \bv_T,\br_T,\hat\bbeta_0
\right)\\
&\quad\le
\PP_{\bbeta^*}\left(
\frac{\sigma_{\bv_T,\br_T}}{\sqrt{(T-1)\underline\lambda}}
\|\bxi\|_2
\ge
\frac{1}{3}
\left|
\max_{\cS\in\bcS_0}R_{T,\cS}^*
-
\max_{\cS\notin\bcS_0}R_{T,\cS}^*
\right|
\ \Bigg|\ \bv_T,\br_T,\hat\bbeta_0
\right)
=o(1).
\end{align*}
Combining the preceding displays and using
\(\kappa\asymp\sigma_{\bv_T,\br_T}\sqrt{s_*/(T\underline\lambda)}\), we obtain
\[
\PP_{\bbeta^*}(p_m>\alpha\mid\bv_T,\br_T,\hat\bbeta_0)
\le
O(T^{-1})+o(1)
=
o(1).
\]
Therefore \eqref{eq: power} follows.
\section{Proofs of Technical Results} \label{app: proof tech results}
\subsection{Proof of Lemma~\ref{lem: grad hessian rates}}\label{sec: proof lem grad hessian rates}
We first establish \eqref{eq: grad rate} using a standard exponential inequality for a martingale difference sequence. Fix any \(t\in[T]\) and \(k\in[p]\). Then
\[
[\nabla_{\bbeta}\ell_{t-1}(\bbeta^*)]_k
= -\sum_{t'=1}^{t-1}\eb_k^\top\!\left\{\bv_{t',\,i_{t'}}-\EE_{\bbeta^*,\,t',\,\cS_{t'}(\hat\bbeta_{t' - 1})}\big(\bv_{t',\,i_{t'}}\big)\right\}
:= -\sum_{t'=1}^{t-1} y_{t'k},
\]
where \(y_{t'k}:=\eb_k^\top\!\left\{\bv_{t',\,i_{t'}}-\EE_{\bbeta^*,\,t',\,\cS_{t'}(\hat\bbeta_{t' - 1})}\big(\bv_{t',\,i_{t'}}\big)\right\}\) for \(t'\in[t-1]\) and \(y_{0k}=0\).
Let \(\cF_{t'-1}=\sigma(\cH_{t'-1})\). Then \(\cF_0\subseteq \cF_1\subseteq \cdots \subseteq \cF_{t-1}\) is an increasing sequence of \(\sigma\)-fields, and by the tower property,
\[
\EE\big(y_{t'k}\mid \cF_{t'-1}\big)
= \EE\Big(\EE\big(y_{t'k}\mid \bv_{t'},\br_{t'},\cF_{t'-1}\big) \mid \cF_{t'-1}\Big)=0,
\]
so \(\{(y_{t'k},\cF_{t'})\}_{t'=0}^{t-1}\) is a martingale difference sequence. Moreover, under Assumption~\ref{asp: cov},
\[
\max_{k\in[p],\,t'\in[T]}\EE\big(y_{t'k}^2\mid \cF_{t'-1}\big)
= \max_{k,\,t'} \EE\Big(\EE\big(y_{t'k}^2\mid \bv_{t'},\br_{t'},\cF_{t'-1}\big)\Big)
\lesssim \nu^2,
\qquad
\max_{k\in[p],\,t'\in[T]}\!\big|y_{t'k}\big|\lesssim \nu.
\]
Therefore, by Theorem~A of \cite{fan2015expineq}, with probability at least \(1-O\!\big((Tp)^{-2}\big)\),
\[
\big|[\nabla_{\bbeta}\ell_{t-1}(\bbeta^*)]_k\big|
\;\le\; C\,\nu\left\{\log(Tp)+\sqrt{t\log(Tp)}\right\},
\]
for a sufficiently large constant \(C>0\). Taking a union bound over \(k\in[p]\) and \(t\in[T]\), we obtain that, with probability at least \(1-O(T^{-1})\), for all \(t\in[T]\),
\[
\|\nabla_{\bbeta}\ell_{t-1}(\bbeta^*)\|_\infty
\;\le\; C\,\nu\left\{\log(Tp)+\sqrt{t\log(Tp)}\right\},
\]
which proves \eqref{eq: grad rate}.

Next, we prove \eqref{eq: hes rate}. By the triangle inequality, we decompose the left-hand side of \eqref{eq: hes rate} into three terms:
\begin{align*}
    \|\nabla_{\bbeta}^2 \ell_{t-1}(\bbeta) - (t-1) \cdot \bSigma^*\|_{\max} & \le  \underbrace{\|\nabla_{\bbeta}^2 \ell_{t-1}(\bbeta) - \nabla_{\bbeta}^2 \ell_{t-1}(\bbeta^*)\|_{\max}}_{\rm I} \\
    & \quad + \underbrace{\left\|\nabla_{\bbeta}^2 \ell_{t-1}(\bbeta^*) - \sum_{t'=1}^{t-1} \EE ( \bSigma_{t'} (\hat\bbeta_{t' - 1}) \mid \hat\bbeta_{t' - 1} )\right\|_{\max}}_{\rm II}\\
    & \quad + \underbrace{\left\|\sum_{t'=1}^{t-1} \EE ( \bSigma_{t'} (\hat\bbeta_{t' - 1}) \mid \hat\bbeta_{t' - 1} ) - 
    (t-1)\cdot \bSigma^* \right\|_{\max}}_{\rm III}.
\end{align*}
We first bound term I. For any $\cS \in \bcS^K$, $t' \in [t-1]$, $\bbeta \in \cB_1(\bbeta^*, 3\tau)$ and $k, \ell \in [p]$,
\begin{equation}\label{eq: hess liptz}
    \begin{aligned}
   & \left|\EE_{\bbeta, t', \cS}\left[\bv_{t', i_{t'}}\bv_{t', i_{t'}}^{\top}\right]_{k \ell} - \EE_{\bbeta^*, t', \cS}\left[\bv_{t', i_{t'}}\bv_{t', i_{t'}}^{\top}\right]_{k \ell} \right|\\
   & = \left|\left\{\sum_{j \in \cS} \PP_{\tilde\bbeta,\bv_{t'}}(j|\cS)\cdot \bv_{t' j k } \cdot \bv_{t' j \ell} \cdot \left(\bv_{t'j} - \EE_{\tilde\bbeta, t', \cS} (\bv_{t',i_{t'}})\right)\right\}^{\top} (\bbeta - \bbeta^*)\right|\\
   & \le \left\|\sum_{j \in \cS} \PP_{\tilde\bbeta,\bv_{t'}}(j|\cS)\cdot \bv_{t' j k } \cdot \bv_{t' j \ell} \cdot \left(\bv_{t'j} - \EE_{\tilde\bbeta, t', \cS} (\bv_{t',i_{t'}})\right)\right\|_{\infty} \cdot \|\bbeta - \bbeta^*\|_1\\
   & \le 6 \nu^3 \cdot \tau,
\end{aligned}
\end{equation}
where $\tilde\bbeta$ lies on the line segment between $\bbeta$ and $\bbeta^*$. The same argument gives
\begin{align*}
      & \left|\EE_{\bbeta, t', \cS}\left(\bv_{t', i_{t'}, k} \right)\EE_{\bbeta, t', \cS}\left(\bv_{t', i_{t'}, \ell} \right) - \EE_{\bbeta^*, t', \cS}\left(\bv_{t', i_{t'}, k} \right)\EE_{\bbeta^*, t', \cS}\left(\bv_{t', i_{t'}, \ell} \right) \right|  \le 12 \nu^3 \cdot \tau.
\end{align*}
Combining the above arguments, we have that
\begin{equation}\label{eq: term I bd}
    {\rm I} \lesssim \nu^3 \cdot t \tau, \quad \text{for all } t \in [T].
\end{equation}

As for term II, we can rewrite
\begin{align*}
   \nabla_{\bbeta}^2 \ell_{t-1}(\bbeta^*) - \sum_{t'=1}^{t-1} \EE ( \bSigma_{t'} (\hat\bbeta_{t' - 1}) \mid \hat\bbeta_{t' - 1} ) & =  \sum_{t'=1}^{t-1} \left\{ \bSigma_{t'} (\hat\bbeta_{t' - 1})   - \EE ( \bSigma_{t'} (\hat\bbeta_{t' - 1}) \mid \hat\bbeta_{t' - 1} ) \right\}\\
   & = \sum_{t'=1}^{t-1} \left\{ \bSigma_{t'} (\hat\bbeta_{t' - 1})   - \EE ( \bSigma_{t'} (\hat\bbeta_{t' - 1}) \mid \hat\bbeta_{t' - 1} , \cF_{t'-1}) \right\}\\
    & = \sum_{t'=1}^{t-1} \left\{ \bSigma_{t'} (\hat\bbeta_{t' - 1})   - \EE ( \bSigma_{t'} (\hat\bbeta_{t' - 1}) \mid  \cF_{t'-1}) \right\} ,
\end{align*}
where $\left\{ \bSigma_{t'} (\hat\bbeta_{t' - 1})   - \EE ( \bSigma_{t'} (\hat\bbeta_{t' - 1}) \mid \cF_{t'-1} ) \right\}_{t'=1}^{t-1}$ is a martingale difference sequence. Under Assumption~\ref{asp: cov}, we have
\begin{align*}
    &\max_{t' \in [T], k,\ell \in [p]} \left|\left[ \bSigma_{t'} (\hat\bbeta_{t' - 1})   - \EE ( \bSigma_{t'} (\hat\bbeta_{t' - 1}) \mid \cF_{t'-1} )\right]_{k \ell} \right| \lesssim \nu^2,\\
    & \max_{t' \in [T], k,\ell \in [p]} \EE \left\{ \left[ \bSigma_{t'} (\hat\bbeta_{t' - 1})   - \EE ( \bSigma_{t'} (\hat\bbeta_{t' - 1}) \mid \cF_{t'-1} )\right]_{k \ell}^2 \mid \cF_{t'-1}\right\} \lesssim \nu^4,
\end{align*}
and hence applying Theorem~A in \cite{fan2015expineq}, we have that with probability at least $1 - O(T^{-1})$,
\begin{equation}\label{eq: term II bd}
    {\rm II} \le C\,\nu^2 \left\{\log(Tp)+\sqrt{t\log(Tp)}\right\},\quad \text{for all } t \in [T].
\end{equation}
We bound term III by the following claim, whose proof is in Section~\ref{sec: proof claim cont bd}.
\begin{claim}\label{claim: cont bd}
Under Assumption~\ref{asp: assortment const}, for all \(t\in[T]\),
\begin{equation}\label{eq: term III bd}
    \mathrm{III} \;\le\; C_n\cdot\nu^2  t\tau \;+\; O\!\left({t\nu^2 }/{T}\right).
\end{equation}
In particular, the constant \(C_n\) admits the following forms under the respective conditions of Assumption~\ref{asp: assortment const}:
\begin{enumerate}
    \item If \(\bcS^K=\{\cS\subseteq[n]: |\cS|=K\}\) and \(\rho \le 2\), then
    \[
    C_n \;=\; O\!\left(\bar\mu  K^3 \nu \sigma_r^{-1}  \sqrt{\log n}\right).
    \]
    \item If \(\bcS^K \subseteq \{\cS\subseteq[n]: |\cS|=K\}\) and, for any \(\cS,\cS'\in \bcS^K\), \(|\cS\cap \cS'|\le (K/\rho^2-1)\vee 0\), then
    \[
    C_n \;=\; O\!\left(K \bar\mu   \nu \sigma_r^{-1} \sqrt{\log n}(K/\rho^2 \vee 1)\right).
    \]
\end{enumerate}
\end{claim}
Then \eqref{eq: hes rate} holds by combining \eqref{eq: term I bd}, \eqref{eq: term II bd} and \eqref{eq: term III bd}.
\subsection{Proof of Claim~\ref{claim: cont bd}}\label{sec: proof claim cont bd}
It suffices to prove the following uniform bound: for all \(t\in[T]\),
\(k,\ell\in[p]\), and \(\bbeta\in\cB_1(\bbeta^*,3\tau)\),
\begin{equation}\label{eq: cont rate}
    \left[ \EE ( \bSigma_{t} (\bbeta) \mid \bbeta ) - 
    \EE ( \bSigma_{t} (\bbeta^*) \mid \bbeta^*) \right]_{k \ell}  \le C_n\cdot \nu^2 \tau + \nu^2 / T,
\end{equation}
for some rate \(C_n>0\).

Fix \(k,\ell\in[p]\), \(t\in[T]\), and
\(\bbeta\in\cB_1(\bbeta^*,3\tau)\). By the tower property,
\begin{align*}
\left[\EE\{\bSigma_t(\bbeta)\mid\bbeta\}\right]_{k\ell}
&=
\EE\!\left[
\EE\!\left\{[\bSigma_t(\bbeta)]_{k\ell}\mid \bbeta,\bv_t\right\}
\right] \\
&=
\EE\!\left[
\sum_{\cS\in\bcS^K}
\PP(\cS_t(\bbeta)=\cS\mid \bbeta,\bv_t)\,
x_{t,\cS}^{(k,\ell)}
\right],
\end{align*}
where the second equality follows from Assumption~\ref{asp: assortment const}:
the maximizer \(\cS_t(\bbeta)\) is unique with probability one
conditional on \(\bbeta\) and \(\bv_t\) under the Gaussian distribution of $\br_T$, and
\[
\begin{aligned}
x_{t,\cS}^{(k,\ell)}
&=
\sum_{j\in\cS}
\PP_{\bbeta^*,\bv_t}(j\mid\cS)\,
\bv_{tjk}\bv_{tj\ell}^{\top} \\
&\quad -
\Big(\sum_{j\in\cS}
\PP_{\bbeta^*,\bv_t}(j\mid\cS)\bv_{tjk}\Big)
\Big(\sum_{j\in\cS}
\PP_{\bbeta^*,\bv_t}(j\mid\cS)\bv_{tj\ell}\Big)^{\top}.
\end{aligned}
\]
By Assumption~\ref{asp: cov}, \(|x_{t,\cS}^{(k,\ell)}|\le\nu^2\), and hence
\begin{align*}
& \left| \EE ( [\bSigma_{t} (\bbeta)]_{k \ell} \mid \bbeta ) -   \EE ( [\bSigma_{t} (\bbeta^*)]_{k \ell} \mid \bbeta^*)  \right|  \\
    & \le  \EE \Big( \sum_{\cS \in \bcS^K}\left | \PP(\cS_t(\bbeta) = \cS| \bbeta, \bv_{t}) - \PP(\cS_t(\bbeta^*) = \cS| \bbeta^*, \bv_{t}) \right| \cdot |x_{t, \cS}^{(k,\ell)} |\Big)\\
     & \le \nu^2 \cdot \EE \Big( \sum_{\cS \in \bcS^K}\left | \PP(\cS_t(\bbeta) = \cS| \bbeta, \bv_{t}) - \PP(\cS_t(\bbeta^*) = \cS| \bbeta^*, \bv_{t}) \right| \Big) .
\end{align*}
Conditional on $\bv_t$, we have
\begin{equation}\label{eq: total var}
    \begin{aligned}
   &  \sum_{\cS \in \bcS^K} \Big| \PP(\cS_t (\bbeta) = \cS| \bbeta, \bv_{t}) - \PP(\cS_t(\bbeta^*) = \cS| \bbeta^*, \bv_{t}) \Big| \\
   & = 2 \|\PP(\cS_t(\bbeta) | \bbeta, \bv_{t}) -\PP(\cS_t(\bbeta^*) | \bbeta^*, \bv_{t})\|_{\rm TV}\\
   & = 2 \sup_{\cA \subseteq \bcS^K} \Big| \PP(\cS_t (\bbeta) \in \cA |\bbeta,  \bv_{t}) - \PP(\cS_t (\bbeta^*)\in \cA |\bbeta^*, \bv_{t}) \Big|\\
   & = 2 \sup_{\cA \subsetneq \bcS^K}  \Bigg| \PP\left( \max_{\cS \in \cA } \sum_{j \in \cS}  \PP_{\bbeta, \bv_t}(j|\cS) \cdot r_{t j} \ge \max_{\cS' \in \bcS^K \backslash \cA }\sum_{j \in \cS'}  \PP_{\bbeta, \bv_t}(j|\cS') \cdot r_{t j} \,\, \bigg| \bbeta, \bv_{t}\right)\\
   & \quad - \PP\left( \max_{\cS \in \cA } \sum_{j \in \cS}  \PP_{\bbeta^*, \bv_t}(j|\cS)\cdot r_{t j} \ge \max_{\cS' \in \bcS^K \backslash \cA }\sum_{j \in \cS'}  \PP_{\bbeta^*, \bv_t}(j|\cS')\cdot r_{t j} \,\, \bigg| \bbeta^*, \bv_{t}\right) \Bigg|.
   \end{aligned}
\end{equation}

Under Assumption~\ref{asp: assortment const}, with probability at least \(1-T^{-1}\), we have
  \begin{equation}\label{eq: max rev bd}
        \max_{j \in [n]}|r_{t j}| \le \sigma_r\sqrt{2\log(2n)+2\log T}+\bar\mu,
  \end{equation}
  and in turn 
  \begin{equation}\label{eq: set wise rev diff}
       \begin{aligned}
      \max_{\cS \in \bcS^K} \left| \sum_{j \in \cS} \Big( \PP_{ \bbeta, \bv_t}(j|\cS) -  \PP_{\bbeta^*, \bv_t}(j|\cS) \Big)\cdot r_{t j}  \right| & \le \max_{\cS \in \bcS^K } \| \nabla_{\bbeta} R(\cS \mid \tilde\bbeta, \bv_t, \br_t) \|_{\infty} \cdot \|\bbeta - \bbeta^*\|_1\\
      & \le 2 \max_{j \in [n]}\|\bv_{t j}\|_{\infty} \cdot \max_{j \in [n]}|r_{t j}| \cdot  \|\bbeta - \bbeta^*\|_1\\ 
      & \le 2 \nu (\sigma_r\sqrt{2\log(2n)+2\log T}+\bar\mu) \|\bbeta - \bbeta^*\|_1,
  \end{aligned}
  \end{equation}
where the revenue gradient $\nabla_{\bbeta} R(\cS \mid \bbeta, \bv_t, \br_t)$ is given in \eqref{eq: gradient revenue}, 
and $\tilde\bbeta$ lies between $\bbeta$ and $\bbeta^*$. 

Define \(\varepsilon=12\nu(\sigma_r\sqrt{2\log(2n)+2\log T}+\bar\mu)\tau\).
Combining \eqref{eq: total var} and \eqref{eq: set wise rev diff} gives
   \begin{equation}\label{eq: tv bound}
   \begin{aligned}
        & \|\PP(\cS_t (\bbeta) | \bbeta, \bv_{t}) - \PP(\cS_t(\bbeta^*) |\bbeta^*, \bv_{t})\|_{\rm TV} \\
        & \quad \le \sup_{\cA \subsetneq \bcS^K} \left| \PP \Big(  - \varepsilon  \le \max_{\cS \in \cA } \xi_{\cS} - \max_{\cS' \in \bcS^K \backslash \cA } \xi_{\cS'} \le \varepsilon\mid\bbeta^*, \bv_{t}\Big)\right| + T^{-1},
   \end{aligned}
\end{equation}
where we define
\(\xi_{\cS}:=\sum_{j\in\cS}\PP_{\bbeta^*,\bv_t}(j\mid\cS)\,r_{tj}\)
for \(\cS\in\bcS^K\). It remains to bound the first term on the
right-hand side of \eqref{eq: tv bound} using the anti-concentration
inequality in Theorem~2.4 of \cite{belloni2024anticon}. To apply the
theorem, we verify that the covariance matrix
\(\{\xi_{\cS}\}_{\cS\in\bcS^K}\) satisfies the required pairwise
conditions.
 
Specifically, for any \(\cS, \cS' \in \bcS^K\), let
\[
\sigma_{\cS\cS'} := \EE\!\left[\,(\xi_{\cS}-\mu_{\cS})(\xi_{\cS'}-\mu_{\cS'}) \,\big|\, \bv_t \right]
\]
denote the covariance conditional on \(\bv_t\), where \(\mu_{\cS} := \EE(\xi_{\cS}\mid \bv_t)\), and let \(\sigma_{\cS}^2 := \sigma_{\cS\cS}\) be the marginal variance of \(\xi_{\cS}\). Under Assumption~\ref{asp: cov}, we have
\begin{align*}
    \sigma_r^2 &\ge \sigma_{\cS}^2 = \sigma_r^2 \sum_{j \in \cS}  \PP_{\bbeta^*, \bv_t}(j \mid \cS)^2 \ge \frac{\sigma_r^2}{K}  \left(\sum_{j \in \cS} \PP_{\bbeta^*, \bv_t}(j \mid \cS)\right)^2 \\
    & = \frac{\sigma_r^2}{K}  \left(1 - \PP_{\bbeta^*, \bv_t}(0 \mid \cS)\right)^2\ge \frac{\sigma_r^2}{K} \left(1 - \frac{1}{1 + K /\rho
    }\right)^2 \ge \frac{\sigma_r^2}{4 K},
\end{align*}
where $\rho $ is the constant defined in Assumption~\ref{asp: cov}. Define
\[
\gamma \;:=\; \max_{\substack{\cS,\cS'\in \bcS^K, \,\, \cS\neq \cS'}} \sigma_{\cS}^{-2}\,\sigma_{\cS\cS'}.
\]
We will apply Theorem~2.4 of \cite{belloni2024anticon} by controlling the rate of $\gamma$ to bound the anti-concentration term and derive the corresponding order of \(C_n\) under each of the following conditions listed in Assumption~\ref{asp: assortment const}:

\noindent (1) If $\bcS^K = \{\cS: \cS \subseteq [n], |\cS| = K\}$ and $\rho \le 2$.

\noindent (2) If $\bcS^K \subseteq \{\cS: \cS \subseteq [n], |\cS| = K\}$ and for any $\cS, \cS' \in \bcS^K$, $|\cS \cap \cS'| \le (K /\rho^2 - 1)\vee 0$.

We first compute $\gamma$ and $C_n$ when the condition (1) is satisfied. For any two $\cS, \cS' \in \bcS^K$ and $\cS \ne \cS'$, denote $\cD = \cS \cap \cS'$. For this display, write
\[
A=\sum_{j\in\cD}u_{tj}^*,\quad
B=\sum_{j'\in\cS\backslash\cD}u_{tj'}^*,\quad
C=\sum_{j''\in\cS'\backslash\cD}u_{tj''}^*,
\]
and
\[
A_2=\sum_{j\in\cD}(u_{tj}^*)^2,\qquad
B_2=\sum_{j'\in\cS\backslash\cD}(u_{tj'}^*)^2.
\]
Under Assumption~\ref{asp: cov}, the following holds
\begin{align*}
1 - \sigma_{\cS}^{-2} \sigma_{\cS\cS'}
&= 1 - \frac{(1+A+B)A_2}{(1+A+C)(A_2+B_2)}\\
&= \frac{1}{1+A+C}
\left\{(1+A+C)-\frac{(1+A+B)A_2}{A_2+B_2}\right\}\\
&= \frac{1}{1+A+C}
\left\{C+\frac{(1+A)B_2}{A_2+B_2}
-\frac{BA_2}{A_2+B_2}\right\}\\
&\ge \frac{1}{1+A+C}
\left\{C+\frac{(1+A)B}{\rho A+B}
-\frac{\rho BA}{\rho A+B}\right\}\\
&\ge
\frac{B}{1+A+C}
\times
\frac{1+B/\rho+(2-\rho)A}{\rho A+B}
\gtrsim K^{-2},
\end{align*}
which gives $1 - \gamma \gtrsim K^{-2}$. By Theorem~2.4 in \cite{belloni2024anticon}, we have 
\begin{align*}
    & \max_{\cA \subsetneq \bcS^K} \PP\Big(\big| \max_{\cS' \in \cA^c} \xi_{\cS'} - \max_{\cS \in \cA} \xi_{\cS} \big| \le \varepsilon \,\, \big| \bv_{t} \Big) \lesssim \EE \left({\max_{\cS \in \bcS^K } |\xi_{\cS} - \mu_{\cS} | } \,\, \big| \bv_{t} \right)  \frac{{K}\varepsilon}{(1 - \gamma)\sigma_r^2 } \\
     & \lesssim \EE\left[ \max_{j \in [n]} |r_{t j} - \mu_{rj}| \right]  \frac{K^3 \varepsilon}{\sigma_r^2 } \lesssim 
     \frac{K^3 \varepsilon \sqrt{\log n}}{\sigma_r } \lesssim 
    \bar\mu  K^3 \nu \sigma_r^{-1}  \sqrt{\log n} \tau.
\end{align*}
Plugging this bound into \eqref{eq: tv bound} gives
\begin{align*}
        &\EE \|\PP(\cS_t (\bbeta) | \bbeta, \bv_{t})
        - \PP(\cS_t(\bbeta^*) | \bbeta^*, \bv_{t})\|_{\rm TV} \\
        &\quad \le
        \EE \left(
        \max_{\cA \subsetneq \bcS^K}
        \PP\Big(\big| \max_{\cS' \in \cA^c} \xi_{\cS'}
        - \max_{\cS \in \cA} \xi_{\cS} \big|
        \le \varepsilon \,\, \big| \bv_{t} \Big)
        \right) + T^{-1}\\
        & \quad \lesssim   \bar\mu  K^3 \nu \sigma_r^{-1}  \sqrt{\log n} \tau + T^{-1},
   \end{align*}
   which in turn gives
\begin{align*}
    & \left| \EE ( [\bSigma_{t} (\bbeta)]_{k \ell} \mid \bbeta )
    - \EE ( [\bSigma_{t} (\bbeta^*)]_{k \ell} \mid \bbeta^*) \right| \\
    & \quad \lesssim \nu^2 \cdot
    \EE \|\PP(\cS_t(\bbeta) | \bbeta, \bv_{t})
    - \PP(\cS_t(\bbeta^*) |\bbeta^*, \bv_{t})\|_{\rm TV} \\
    & \lesssim   \bar\mu  K^3 \nu^3 \sigma_r^{-1}  \sqrt{\log n} \tau  +  \nu^2 / T,
\end{align*}
and \eqref{eq: cont rate} holds with $C_n = O(\bar\mu  K^3 \nu \sigma_r^{-1}  \sqrt{\log n})$.

Next consider the rates of \(\gamma\) and \(C_n\) under condition~(2). For any pair of candidate assortments $\cS, \cS' \in \bcS^K$, without loss of generality, assume that $\sigma_{\cS} \le \sigma_{\cS'}$. When $ K /\rho^2 - 1 \ge 0$, we have that
\begin{align*}
     \sigma_{\cS}^2  &= \sigma_r^2 \sum_{j \in \cS} \PP_{\bbeta^*, \bv_t}(j|\cS)^2 \ge \sigma_r^2 K^{-1} \left( \sum_{j \in \cS} \PP_{\bbeta^*, \bv_t}(j|\cS)\right)^2 = \sigma_r^2 / K,\\
   \sigma_{\cS\cS'} &= \sigma_r^2 \sum_{j' \in \cS \cap \cS'} \PP_{\bbeta^*, \bv_t}(j'|\cS) \cdot \PP_{\bbeta^*, \bv_t}(j'|\cS') \le \sigma_r^2 \cdot |\cS \cap \cS'| \cdot \rho^2/K^2 \\
  & \le \sigma_r^2 (K/\rho^2 - 1) \rho^2/K^2 = \sigma_r^2 (1/K -  \rho^2/K^2),
\end{align*}
and in turn 
\[
1-\sigma_{\cS}^{-2}\sigma_{\cS\cS'}
\ge
1-\sigma_r^2(1/K-\rho^2/K^2)K/\sigma_r^2
=
\rho^2/K,
\]
where the last bound gives \(1-\gamma\ge \rho^2/K\).

On the other hand, when $ K /\rho^2 - 1 < 0$, we have $\sigma_{\cS\cS'} = 0$ and $1 - \gamma = 1$. Combining the results when $ K /\rho^2 - 1 \ge 0$ and $ K /\rho^2 - 1 < 0$ respectively, we have that $1 - \gamma  \ge (\rho^2/K\wedge 1)$.

For any $\bv_t$, applying Theorem~2.4 of \cite{belloni2024anticon} leads to the following bound,
\begin{align*}
    & \max_{\cA \subsetneq \bcS^K} \PP\Big(\big| \max_{\cS' \in \cA^c} \xi_{\cS'} - \max_{\cS \in \cA} \xi_{\cS} \big| \le \varepsilon | \bv_{t} \Big) \lesssim \EE\left[ \max_{j \in [n]} |r_{t j} - \mu_{rj}| \right] \frac{K\varepsilon}{\sigma_r^2 } \left(\frac{K}{\rho^2} \vee 1\right)\\
    & \lesssim \frac{K\varepsilon \sqrt{\log n}}{\sigma_r } \left(\frac{K}{\rho^2} \vee 1\right) \le K \bar\mu   \nu \sigma_r^{-1} \sqrt{\log n}(K/\rho^2 \vee 1)\tau. 
\end{align*}
Applying the above bound to \eqref{eq: tv bound} gives
\begin{align*}
    & \left| \EE ( [\bSigma_{t} (\bbeta)]_{k \ell} \mid \bbeta ) -   \EE ( [\bSigma_{t} (\bbeta^*)]_{k \ell} \mid \bbeta^*)  \right|\lesssim \nu^2  \EE \|\PP(\cS_t(\bbeta) |\bbeta, \bv_{t}) - \PP(\cS_t(\bbeta^*) | \bbeta^*, \bv_{t})\|_{\rm TV} \\
    & \lesssim K \bar\mu   \nu^3 \sigma_r^{-1} \sqrt{\log n}(K/\rho^2 \vee 1)\tau +  \nu^2 
     / T,
\end{align*}
and \eqref{eq: cont rate} holds with $C_n = O\big(K \bar\mu   \nu \sigma_r^{-1} \sqrt{\log n}(K/\rho^2 \vee 1)\big)$.
\subsection{Proof of Lemma~\ref{lm: MTG coupling}}
\label{sec: proof lm MTG coupling}

For \(t=0,1,\ldots,T-1\), let \(\cF_t:=\sigma(\cH_t)\). Then
\(\cF_0\subseteq \cF_1\subseteq\cdots\subseteq\cF_{T-1}\) is an increasing sequence of
\(\sigma\)-fields. We adopt the notation of Theorem~2.1 and Proposition~2.1 in
\cite{maias2025mtgcoupling}, replacing their filtration notation
\((\cH_0,\ldots,\cH_n)\) by \((\cF_0,\ldots,\cF_{T-1})\). Under our setup, one may identify
\[
\begin{aligned}
\bS
&=
-(T-1)^{-1/2}\sum_{t=1}^{T-1}
(\bSigma_{\cI_*}^*)^{-1/2}
\Big\{
\tilde\bv_{t,i_t}
-
\EE_{\bbeta^*,t,\cS_t}\big(\tilde\bv_{t,i_t}\big)
\Big\},\\
\tilde X_t
&:=
\EE(\bS\mid \cF_t)-\EE(\bS\mid \cF_{t-1})\\
&=
-(T-1)^{-1/2}
(\bSigma_{\cI_*}^*)^{-1/2}
\Big\{
\tilde\bv_{t,i_t}
-
\EE_{\bbeta^*,t,\cS_t}\big(\tilde\bv_{t,i_t}\big)
\Big\},\\
V_t
&:=
\Cov(\tilde X_t\mid \cF_{t-1})\\
&=
(T-1)^{-1}
(\bSigma_{\cI_*}^*)^{-1/2}
\Big[\EE\!\big(\bSigma_t(\hat\bbeta_{t-1})\mid \hat\bbeta_{t-1}\big)\Big]_{\cI_*,\cI_*}
(\bSigma_{\cI_*}^*)^{-1/2},
\end{aligned}
\]
and hence
\[
\bS= \sum_{t=1}^{T-1}\tilde X_t,
\qquad
U=0,
\qquad
\phi_2(s_*)=\sqrt{2s_*},
\qquad
\Sigma=\Ib_{s_*},
\qquad
\Omega=\sum_{t=1}^{T-1}V_t-\Ib_{s_*}.
\]
Here we write \(\tilde\bv_{tj}:=[\bv_{tj}]_{\cI_*}\) and
\[
\tilde\bv_{t,i_t}
:=
\sum_{j\in \cS_t\cup\{0\}}\tilde\bv_{tj}\,\II(j=i_t).
\]

The expression for \(\tilde X_t\) follows from the martingale difference property. Indeed, for
\(t>t'\),
\[
\begin{aligned}
\EE\!\left(
\tilde\bv_{t,i_t}-\EE_{\bbeta^*,t,\cS_t}\big(\tilde\bv_{t,i_t}\big)
\;\middle|\; \cF_{t'}
\right)
&=
\EE\!\left(
\EE\!\left(
\tilde\bv_{t,i_t}-\EE_{\bbeta^*,t,\cS_t}\big(\tilde\bv_{t,i_t}\big)
\;\middle|\; \bv_t,\br_t,\cF_{t-1}
\right)
\middle|\cF_{t'}
\right)\\
&=
\EE\!\left(
\EE_{\bbeta^*,t,\cS_t}\big(\tilde\bv_{t,i_t}\big)
-
\EE_{\bbeta^*,t,\cS_t}\big(\tilde\bv_{t,i_t}\big)
\;\middle|\; \cF_{t'}
\right)
=0,
\end{aligned}
\]
whereas for \(t\le t'\),
\[
\EE\!\left(
\tilde\bv_{t,i_t}-\EE_{\bbeta^*,t,\cS_t}\big(\tilde\bv_{t,i_t}\big)
\;\middle|\; \cF_{t'}
\right)
=
\tilde\bv_{t,i_t}-\EE_{\bbeta^*,t,\cS_t}\big(\tilde\bv_{t,i_t}\big).
\]

We next bound \(\Omega\). By definition,
\[
\Omega
=
\frac{1}{T-1}
(\bSigma_{\cI_*}^*)^{-1/2}
\left[
\sum_{t=1}^{T-1}
\EE\!\big(\bSigma_t(\hat\bbeta_{t-1})\mid \hat\bbeta_{t-1}\big)
-
(T-1)\bSigma^*
\right]_{\cI_*,\cI_*}
(\bSigma_{\cI_*}^*)^{-1/2}.
\]
The proof of Claim~\ref{claim: cont bd} gives the pointwise local stability bound
\[
    \left\|
        \EE\{\bSigma_t(\bbeta)\mid\bbeta\}
        -
        \bSigma^*
    \right\|_{\max}
    \lesssim
    C_n\nu^2\|\bbeta-\bbeta^*\|_1+\nu^2/T,
    \qquad
    \bbeta\in\cB_1(\bbeta^*,3\tau).
\]
Let
\[
    t_0
    :=
    C\frac{\nu^4s^2\log(Tp)}{\underline\lambda^2},
\]
where $C$ is the constant $C'$ in Theorem~\ref{thm:lasso-rates}.
For \(t<t_0\), the constraint gives
\(\|\hat\bbeta_{t-1}-\bbeta^*\|_1\lesssim\tau\), and hence
\[
    C_n\nu^2\sum_{t<t_0}\|\hat\bbeta_{t-1}-\bbeta^*\|_1
    \lesssim
    C_n\nu^2t_0\tau
    \lesssim
    \nu^2\log(Tp),
\]
where the last inequality uses
\[
    \tau\le c\frac{\underline\lambda^2}{C_n\nu^4s^2}.
\]
For \(t\ge t_0\), Theorem~\ref{thm:lasso-rates} gives uniformly
\[
    \|\hat\bbeta_{t-1}-\bbeta^*\|_1
    \lesssim
    \nu s\sqrt{\frac{\log(Tp)}{\underline\lambda^2t}} .
\]
Therefore,
\[
\begin{aligned}
C_n\nu^2
\sum_{t=t_0}^{T-1}
\|\hat\bbeta_{t-1}-\bbeta^*\|_1
&\lesssim
C_n\nu^3s
\frac{\sqrt{\log(Tp)}}{\underline\lambda}
\sum_{t=t_0}^{T-1}t^{-1/2}\\
&\lesssim
\frac{C_n\nu^3s\sqrt{T\log(Tp)}}{\underline\lambda}.
\end{aligned}
\]
Combining the early and late time ranges, we obtain
\begin{align*}
     \left\|
    \sum_{t=1}^{T-1}
    \EE\!\big(\bSigma_t(\hat\bbeta_{t-1})\mid \hat\bbeta_{t-1}\big)
    -
    (T-1)\bSigma^*
    \right\|_{\max}
    &\lesssim
    \frac{C_n\nu^3s\sqrt{T\log(Tp)}}{\underline\lambda}
    +
    \nu^2\log(Tp)\\
    & \lesssim  \frac{C_n\nu^3s\sqrt{T\log(Tp)}}{\underline\lambda},
\end{align*}
where in the last line the term \(\nu^2\log(Tp)\) is absorbed using
\(C_n\gtrsim\nu\), \(\nu^2s\gtrsim\underline\lambda\), and
\(T\ge C\nu^8s^4\log(Tp)/\underline\lambda^4\) with \(C>0\) sufficiently large.
Consequently,
\[
\begin{aligned}
\|\Omega\|_2
&\le
\frac{s_*}{(T-1)\underline\lambda}
\left\|
    \sum_{t=1}^{T-1}
    \EE\!\big(\bSigma_t(\hat\bbeta_{t-1})\mid \hat\bbeta_{t-1}\big)
    -
    (T-1)\bSigma^*
\right\|_{\max}\\
&\lesssim
\frac{C_n\nu^3s s_*\sqrt{\log(Tp)}}{\underline\lambda^2\sqrt T}.
\end{aligned}
\]

It remains to bound
\[
\beta_{2,2}
:=
\sum_{t=1}^{T-1}
\EE \Big\{\|\tilde X_t\|_2^3+\|V_t^{1/2}Z_t\|_2^3\Big\},
\]
where \(Z_1,\ldots,Z_{T-1}\in\R^{s_*}\) are i.i.d. standard Gaussian vectors, independent of \(\cF_{T-1}\). Since \(\|\tilde\bv_{t,i_t}-\EE_{\bbeta^*,t,\cS_t}(\tilde\bv_{t,i_t})\|_2\lesssim \sqrt{s_*}\nu\) and
\(\lambda_{\min}(\bSigma_{\cI_*}^*)\ge\underline\lambda\), we have
\[
\max_{t\in[T-1]}
\EE \|\tilde X_t\|_2^3
\lesssim
T^{-3/2}\underline\lambda^{-3/2}s_*^{3/2}\nu^3.
\]
We also have
\[
\max_{t\in[T-1]}
\EE \|V_t^{1/2}Z_t\|_2^3
\lesssim
T^{-3/2}s_*^{3/2}.
\]
For the second bound, note that the pointwise version of Claim~\ref{claim: cont bd} gives
\[
    \left\|
        \EE\{\bSigma_t(\hat\bbeta_{t-1})\mid\hat\bbeta_{t-1}\}
        -
        \bSigma^*
    \right\|_{\max}
    \lesssim
    C_n\nu^2\tau+\nu^2/T.
\]
Under
\[
    \tau\le c\frac{\underline\lambda^2}{C_n\nu^4s^2}
\]
with \(c>0\) sufficiently small, this implies
\[
V_t
\preceq
\frac{C}{T}\,
(\bSigma_{\cI_*}^*)^{-1/2}
\bSigma_{\cI_*}^*
(\bSigma_{\cI_*}^*)^{-1/2}
=
\frac{C}{T}\,\Ib_{s_*}.
\]
Therefore
\[
\|V_t^{1/2}Z_t\|_2
\le
(C/T)^{1/2}\|Z_t\|_2,
\]
and hence
\[
\EE\|V_t^{1/2}Z_t\|_2^3
\lesssim
T^{-3/2}\EE\|Z_t\|_2^3
\lesssim
T^{-3/2}s_*^{3/2}.
\]
Combining the preceding bounds yields
\[
    \beta_{2,2}
    \lesssim
    T^{-1/2}\underline\lambda^{-3/2}s_*^{3/2}\nu^3.
\]

Applying Proposition~2.1 in \cite{maias2025mtgcoupling}, for any \(\eta>0\), there exists a random vector
\(\bxi\mid\cF_0\sim\cN(0,\Ib_{s_*})\) such that
\[
\begin{aligned}
\PP(\|\bS-\bxi\|_2>\eta\mid \hat\bbeta_0)
&\lesssim
\left(
    \frac{\phi_2(s_*)^2\beta_{2,2}}{\eta^3}
\right)^{1/3}
+
\left(
    \frac{\phi_2(s_*)^2\|\Omega\|_2}{\eta^2}
\right)^{1/3}\\
&\lesssim
\left(
    \frac{s_*^{5/2}\nu^3}
    {T^{1/2}\underline\lambda^{3/2}\eta^3}
\right)^{1/3} +
\left(
    \frac{
        C_n\nu^3s\,s_*^2\sqrt{\log(Tp)}
    }{
        \underline\lambda^2\sqrt T\,\eta^2
    }
\right)^{1/3}.
\end{aligned}
\]
Since \(\cF_0=\sigma(\hat\bbeta_0)\), this proves \eqref{eq: MTG coupling}.
\subsection{Proof of Lemma~\ref{lm: rev decomp}}
\label{sec: proof lm rev decomp}

By \eqref{eq: gradient revenue}, on the event \(\cE\), \eqref{eq: grad T bd} holds. Fix
\(\cS\in\bcS^K\). For \(j\in\cS_+\),
\[
\nabla_{\bbeta}\PP_{\bbeta,\bv_T}(j\mid \cS)
=
\PP_{\bbeta,\bv_T}(j\mid \cS)
\left(
\bv_{Tj}
-
\sum_{j'\in\cS}\PP_{\bbeta,\bv_T}(j'\mid \cS)\bv_{Tj'}
\right).
\]
Therefore,
\begin{align*}
&\nabla_{\bbeta}^2 R(\cS\mid \bbeta,\bv_T,\br_T)\\
&=
\sum_{j\in\cS}\PP_{\bbeta,\bv_T}(j\mid \cS)\,r_{Tj}\,\bv_{Tj}\bv_{Tj}^{\top}
-
\Big(\sum_{j\in\cS}\PP_{\bbeta,\bv_T}(j\mid \cS)\,r_{Tj}\Big)
\Big(\sum_{j\in\cS}\PP_{\bbeta,\bv_T}(j\mid \cS)\,\bv_{Tj}\bv_{Tj}^{\top}\Big)\\
&\quad -
\Big(\sum_{j\in\cS}\PP_{\bbeta,\bv_T}(j\mid \cS)\,r_{Tj}\,\bv_{Tj}\Big)
\Big(\sum_{j\in\cS}\PP_{\bbeta,\bv_T}(j\mid \cS)\,\bv_{Tj}\Big)^{\top}\\
&\quad -
\Big(\sum_{j\in\cS}\PP_{\bbeta,\bv_T}(j\mid \cS)\,\bv_{Tj}\Big)
\Big(\sum_{j\in\cS}\PP_{\bbeta,\bv_T}(j\mid \cS)\,r_{Tj}\,\bv_{Tj}\Big)^{\top}\\
&\quad +
2\Big(\sum_{j\in\cS}\PP_{\bbeta,\bv_T}(j\mid \cS)\,r_{Tj}\Big)
\Big(\sum_{j\in\cS}\PP_{\bbeta,\bv_T}(j\mid \cS)\,\bv_{Tj}\Big)
\Big(\sum_{j\in\cS}\PP_{\bbeta,\bv_T}(j\mid \cS)\,\bv_{Tj}\Big)^{\top}.
\end{align*}
Hence, on the event \(\cE\), \eqref{eq: hes T bd} holds.

We now prove \eqref{eq: rev decomp}. A second-order Taylor expansion gives
\begin{align*}
R(\cS\mid \tilde\bbeta^{\,d},\bv_T,\br_T)-R(\cS\mid \bbeta^*,\bv_T,\br_T)
&=
\nabla_{\bbeta}R(\cS\mid \bbeta^*,\bv_T,\br_T)^{\top}
(\tilde\bbeta^{\,d}-\bbeta^*) \\
&\quad+
\frac12(\tilde\bbeta^{\,d}-\bbeta^*)^{\top}
\nabla_{\bbeta}^2R(\cS\mid \tilde\bbeta,\bv_T,\br_T)
(\tilde\bbeta^{\,d}-\bbeta^*),
\end{align*}
where \(\tilde\bbeta\) lies on the line segment between \(\tilde\bbeta^{\,d}\) and
\(\bbeta^*\). By Corollary~\ref{col: error decomp debias lasso},
\[
[\tilde\bbeta^{\,d}]_{\cI_*}-[\bbeta^{*}]_{\cI_*}
=
-\frac{1}{T-1}(\bSigma_{\cI_*}^{*})^{-1}
[\nabla_{\bbeta}\ell_{T-1}(\bbeta^{*})]_{\cI_*}
+\Rb,
\]
and
\[
[\tilde\bbeta^{\,d}]_{\cI_{\rm wk}}=\mathbf 0,
\qquad
[\tilde\bbeta^{\,d}]_{\cI_0^c}=\mathbf 0.
\]
Therefore,
\begin{align*}
\hat R_{T,\cS}
-
R^*_{T,\cS}
&=
R(\cS\mid \tilde\bbeta^{\,d},\bv_T,\br_T)
-
R(\cS\mid \bbeta^*,\bv_T,\br_T)\\
&=
-\frac{1}{\sqrt{T-1}}\,
\big[\nabla_{\bbeta}R(\cS\mid \bbeta^*,\bv_T,\br_T)\big]_{\cI_*}^{\top}
[\bSigma_{\cI_*}^*]^{-1/2}\boldsymbol S\\
&\quad+
\big[\nabla_{\bbeta}R(\cS\mid \bbeta^*,\bv_T,\br_T)\big]_{\cI_*}^{\top}\Rb\\
&\quad-
\big[\nabla_{\bbeta}R(\cS\mid \bbeta^*,\bv_T,\br_T)\big]_{\cI_{\rm wk}}^{\top}
[\bbeta^*]_{\cI_{\rm wk}}\\
&\quad+
\frac12(\tilde\bbeta^{\,d}-\bbeta^*)^{\top}
\nabla_{\bbeta}^2R(\cS\mid \tilde\bbeta,\bv_T,\br_T)
(\tilde\bbeta^{\,d}-\bbeta^*).
\end{align*}
Thus
\begin{align*}
&\max_{\cS\in\bcS^K}
\Bigg|
\hat R_{T,\cS}
-
R^*_{T,\cS}
+
\frac{1}{\sqrt{T-1}}\,
g_{\cS}^{\top}
[\bSigma_{\cI_*}^*]^{-1/2}\boldsymbol S
\Bigg|\\
&\quad\le
\max_{\cS\in\bcS^K}
\big\|
\big[\nabla_{\bbeta}R(\cS\mid \bbeta^*,\bv_T,\br_T)\big]_{\cI_*}
\big\|_2
\|\Rb\|_2\\
&\qquad+
\max_{\cS\in\bcS^K}
\big\|
\nabla_{\bbeta}R(\cS\mid \bbeta^*,\bv_T,\br_T)
\big\|_{\infty}
\|[\bbeta^*]_{\cI_{\rm wk}}\|_1\\
&\qquad+
\max_{\cS\in\bcS^K}
\sup_{\bbeta\in\R^p}
\big\|
\nabla_{\bbeta}^2R(\cS\mid \bbeta,\bv_T,\br_T)
\big\|_{\max}
\|\tilde\bbeta^{\,d}-\bbeta^*\|_1^2 .
\end{align*}

By \eqref{eq: grad T bd},
\[
\max_{\cS\in\bcS^K}
\big\|
\nabla_{\bbeta}R(\cS\mid \bbeta^*,\bv_T,\br_T)
\big\|_{\infty}
\le
4\nu\bar\mu,
\]
and
\[
\max_{\cS\in\bcS^K}
\big\|
[\nabla_{\bbeta}R(\cS\mid \bbeta^*,\bv_T,\br_T)]_{\cI_*}
\big\|_2
\le
4\nu\bar\mu\sqrt{s_*}.
\]

By Corollary~\ref{col: error decomp debias lasso},
\[
\|\Rb\|_2
\lesssim
\frac{\nu^3s_*^{3/2}\log(Tp)}{T\underline\lambda^2}
\left\{
    1+\frac{C_n\nu s+\nu^2s_*}{\underline\lambda}
\right\}
+
\frac{\nu^2\sqrt{s_*}}{\underline\lambda}
\|[\boldsymbol\beta^*]_{\cI_{\rm wk}}\|_1 .
\]
Consequently,
\begin{align*}
\max_{\cS\in\bcS^K}
\big\|
[\nabla_{\bbeta}R(\cS\mid \bbeta^*,\bv_T,\br_T)]_{\cI_*}
\big\|_2
\|\Rb\|_2
&\lesssim
\bar\mu
\frac{\nu^4s_*^2\log(Tp)}{T\underline\lambda^2}
\left\{
    1+\frac{C_n\nu s+\nu^2s_*}{\underline\lambda}
\right\}\\
&\quad+
\bar\mu
\frac{\nu^3s_*}{\underline\lambda}
\|[\boldsymbol\beta^*]_{\cI_{\rm wk}}\|_1 .
\end{align*}
The weak first-order term satisfies
\[
\max_{\cS\in\bcS^K}
\big\|
\nabla_{\bbeta}R(\cS\mid \bbeta^*,\bv_T,\br_T)
\big\|_{\infty}
\|[\bbeta^*]_{\cI_{\rm wk}}\|_1
\lesssim
\bar\mu\nu
\|[\bbeta^*]_{\cI_{\rm wk}}\|_1
\lesssim
\bar\mu
\frac{\nu^3s_*}{\underline\lambda}
\|[\bbeta^*]_{\cI_{\rm wk}}\|_1,
\]
where the last inequality follows from
\(\underline\lambda\le \nu^2s_*\).

It remains to bound the second-order revenue term. By the definition of the debiased estimator,
the KKT condition on \(\cI_*\), \eqref{eq: min eigen sig support}, and
\eqref{eq: support l1-rate},
\[
\begin{aligned}
\|\tilde\bbeta^{\,d}-\bbeta^*\|_1
&\le
\|[\hat\bbeta_{T-1}]_{\cI_*}-[\bbeta^*]_{\cI_*}\|_1
+
\left\|
\Big([\nabla_{\bbeta}^2\ell_{T-1}(\hat\bbeta_{T-1})]_{\cI_*,\cI_*}\Big)^{-1}
[\nabla_{\bbeta}\ell_{T-1}(\hat\bbeta_{T-1})]_{\cI_*}
\right\|_1\\
&\quad+
\|[\bbeta^*]_{\cI_{\rm wk}}\|_1\\
&\lesssim
\nu s_*\sqrt{\frac{\log(Tp)}{\underline\lambda^2T}}
+
\|[\bbeta^*]_{\cI_{\rm wk}}\|_1 .
\end{aligned}
\]
Therefore, by \eqref{eq: hes T bd},
\begin{align*}
&\max_{\cS\in\bcS^K}
\sup_{\bbeta\in\R^p}
\big\|
\nabla_{\bbeta}^2R(\cS\mid \bbeta,\bv_T,\br_T)
\big\|_{\max}
\|\tilde\bbeta^{\,d}-\bbeta^*\|_1^2\\
&\quad\lesssim
\bar\mu\nu^2
\left\{
\nu^2s_*^2\frac{\log(Tp)}{\underline\lambda^2T}
+
\|[\bbeta^*]_{\cI_{\rm wk}}\|_1^2
\right\}\\
&\quad\lesssim
\bar\mu
\frac{\nu^4s_*^2\log(Tp)}{\underline\lambda^2T}
+
\bar\mu
\frac{\nu^3s_*}{\underline\lambda}
\|[\bbeta^*]_{\cI_{\rm wk}}\|_1,
\end{align*}
where the last step uses
\(\underline\lambda\le\nu^2s_*\) and
\[
    \|[\bbeta^*]_{\cI_{\rm wk}}\|_1
    =
    o\left(
        \frac{1}{\nu}\sqrt{\frac{\log(Tp)}{T}}
    \right).
\]
Combining the three bounds yields \eqref{eq: rev decomp}, and the proof is complete.
\subsection{Proof of Lemma~\ref{lm: sphere coverage}}
\label{sec: proof lm sphere coverage}

Fix \(0<\eps\le1\) and \(\eb\in S^{s_*-1}\). For \(i\in[m]\), define the spherical cap
\[
\cC_{\eps}(\zeta_i)
:=
\{x\in S^{s_*-1}:\ \|x-\zeta_i\|_2\le \eps\}.
\]
We first show that
\begin{equation}\label{eq: cap area}
\PP\big(\eb\in \cC_{\eps}(\zeta_i)\big)
\ge
\sqrt{\frac{\pi}{8s_*}}\left(\frac{2\eps}{\pi}\right)^{s_*-1}.
\end{equation}
Since \(\zeta_i\) is uniform on \(S^{s_*-1}\), by rotational invariance we may equivalently treat \(\zeta_i\) as fixed and \(\eb\) as uniform on \(S^{s_*-1}\). Hence
\begin{align*}
\PP\big(\eb\in \cC_{\eps}(\zeta_i)\big)
&=
\frac{\operatorname{Area}\big(\cC_{\eps}(\zeta_i)\big)}
{\operatorname{Area}(S^{s_*-1})} \\
&=
\left[\frac{2\pi^{s_*/2}}{\Gamma(s_*/2)}\right]^{-1}
\frac{2\pi^{(s_*-1)/2}}{\Gamma((s_*-1)/2)}
\int_{0}^{2\arcsin(\eps/2)}\sin^{s_*-2}\psi\,d\psi  \\
&=
\frac{\Gamma(s_*/2)}{\Gamma((s_*-1)/2)\sqrt{\pi}}
\int_{0}^{2\arcsin(\eps/2)}\sin^{s_*-2}\psi\,d\psi .
\end{align*}
The second equality uses the hyperspherical cap area formula \cite{li2010concise}. Since \(0<\eps\le1\), we have \(2\arcsin(\eps/2)\le\pi/2\), and thus \(\sin\psi\ge(2/\pi)\psi\) on the integration range. Therefore,
\begin{align*}
\PP\big(\eb\in \cC_{\eps}(\zeta_i)\big)
&\ge
\frac{\Gamma(s_*/2)}{\Gamma((s_*-1)/2)\sqrt{\pi}}
\int_{0}^{2\arcsin(\eps/2)}
\left(\frac{2}{\pi}\psi\right)^{s_*-2}\,d\psi\\
&\ge
\frac{s_*-1}{\sqrt{2s_*\pi}}
\cdot
\frac{1}{s_*-1}
\cdot
\frac{\pi}{2}
\left(\frac{2\eps}{\pi}\right)^{s_*-1}\\
&=
\sqrt{\frac{\pi}{8s_*}}
\left(\frac{2\eps}{\pi}\right)^{s_*-1},
\end{align*}
where the second inequality uses Wendel's inequality together with \(2\arcsin(\eps/2)\ge\eps\). This proves \eqref{eq: cap area}.

Since \(\zeta_1,\ldots,\zeta_m\) are i.i.d., it follows that
\begin{align*}
\PP\big(\|\eb-\zeta_i\|_2>\eps,\ \forall i\in[m]\big)
&=
\Big\{1-\PP\big(\eb\in \cC_{\eps}(\zeta_i)\big)\Big\}^m \\
&\le
\left[
1-\sqrt{\frac{\pi}{8s_*}}\left(\frac{2\eps}{\pi}\right)^{s_*-1}
\right]^m \\
&\le
\exp\left\{-m\sqrt{\frac{\pi}{8s_*}}\left(\frac{2\eps}{\pi}\right)^{s_*-1}\right\},
\end{align*}
where the last inequality follows from \(1-x\le e^{-x}\) for \(x\ge0\). Taking complements completes the proof.
}

\bibliographystyle{imsart-number}
\bibliography{references}

\end{document}